\documentclass[10pt, a4paper]{article}

\usepackage{graphicx}
\usepackage[latin1]{inputenc}
\usepackage[T1]{fontenc} 
\usepackage{amsmath}
\usepackage{amsfonts}
\usepackage{amssymb}
\usepackage{amsthm}
\usepackage[upright]{fourier} 
\usepackage{mathtools}
\usepackage{ifthen}
\usepackage{lettrine}
\usepackage{tikz}
\usetikzlibrary{calc, math}
\usetikzlibrary{positioning}
\usepackage{tikz-cd}
\usepackage{framed}

\usepackage[affil-it]{authblk}
\usepackage{enumitem}
\usepackage[normalem]{ulem}
\usepackage[mathscr]{euscript}
\usepackage{dsfont}
\usepackage{float}
\usepackage{times}
\usepackage{bbm}
\usepackage{xcolor}
\usepackage{xparse}
\usepackage[hang]{footmisc}
\usepackage[colorlinks,linkcolor=blue,citecolor=red,urlcolor=blue]{hyperref}
\usepackage[capitalise]{cleveref}
\usepackage[left=2cm, right=2cm, bottom=2cm, top=2cm]{geometry}
\usepackage{ccicons} 
\usepackage{academicons} 
\usepackage{fontawesome} 
\usepackage{pifont} 
\usepackage{orcidlink}
\usepackage[hyperpageref]{backref}  

\usepackage[numbers]{natbib} 

\newcommand{\mapolicebackref}[1]{
    \hspace*{\fill} \mbox{\textit {\small #1}}
}
\renewcommand*{\backref}[1]{}
\renewcommand*{\backrefalt}[4]{%
\ifcase #1 \mapolicebackref{no quotes}
    \or \mapolicebackref{quoted page #2}
    \else \mapolicebackref{#1 quotes pages #2}
\fi
}

\usepackage{fancyhdr}
\newlist{sk}{enumerate}{1}
\setlist[sk]{label=$\mathbb{(HMV)}$:, ref=$\mathbb{(HMV)}$, wide, labelwidth=!, labelindent=0pt}

\newlist{condH}{enumerate}{1}
\setlist[condH]{label=$\mathbb{(VFP)}$:, ref=$\mathbb{(VFP)}$, wide, labelwidth=!, labelindent=0pt}

\newtheorem{theorem}{Theorem}[section]

\newtheorem{proposition}[theorem]{Proposition}

\theoremstyle{definition}
\newtheorem{definition}[theorem]{Definition}

\theoremstyle{remark}
\newtheorem{remark}[theorem]{Remark}

\numberwithin{equation}{section}

\crefname{example}{Example}{Examples}
\Crefname{example}{Example}{Examples}

\crefname{assumption}{Assumption}{Assumptions}
\Crefname{assumption}{Assumption}{Assumptions}

\crefname{condition}{Condition}{Conditions}
\Crefname{condition}{Condition}{Conditions}

\let\para\S
\crefformat{section}{#2\para#1#3}
\crefformat{subsection}{#2\para#1#3}
\crefformat{subsubsection}{#2\para#1#3}

\setlist{topsep=1ex, itemsep=0.5ex, before={\setlist{topsep=-.5ex}}}

\newcommand{\C}{\ensuremath{\mathcal{C}}}

\newcommand{\E}{\ensuremath{\mathbb{E}}}

\renewcommand{\S}{\ensuremath{\mathcal{S}}}

\newcommand{\N}{\ensuremath{\mathbb{N}}}

\newcommand{\R}{\ensuremath{\mathbb{R}}}

\def\<{\langle}
\def\>{\rangle}

\NewDocumentCommand{\Lin}{om}{\IfNoValueTF{#1}{L(\R^{#2},\R^{#2})}{L(\R^{#1},\R^{#2})}}
\NewDocumentCommand{\Cb}{om}{\IfNoValueTF{#1}{\C_b^{#2}}{\C_b^{#2,#1}}}

\NewDocumentCommand{\Lip}{om}{\IfNoValueTF{#1}{|#2|_{\mathrm{Lip}}}{|#2|_{\mathrm{Lip};\,#1}}}

\renewcommand{\geq}{\geqslant}
\renewcommand{\leq}{\leqslant}
\def\Cb{{\mathrm {BC}}}
\def\${|\!|\!|}

\newcommand{\vertiii}[1]{{\left\vert\kern-0.25ex\left\vert\kern-0.25ex\left\vert #1 \right\vert\kern-0.25ex\right\vert\kern-0.25ex\right\vert}}
\newcommand{\rom}[1]{(\textup{\uppercase\expandafter{\romannumeral#1}})}

\makeatletter
\newcommand{\substackal}[1]{%
  \vcenter{%
    \Let@ \restore@math@cr \default@tag
    \baselineskip\fontdimen10 \scriptfont\tw@
    \advance\baselineskip\fontdimen12 \scriptfont\tw@
    \lineskip\thr@@\fontdimen8 \scriptfont\thr@@
    \lineskiplimit\lineskip
    \ialign{\hfil$\m@th\scriptstyle##$&$\m@th\scriptstyle{}##$\hfil\crcr
      #1\crcr
    }%
  }%
}
\makeatother

\newcommand\blfootnote[1]{%
  \begingroup
  \renewcommand\thefootnote{}\footnote{#1}%
  \addtocounter{footnote}{-1}%
  \endgroup
}

\definecolor{LB}{rgb}{0.29, 0.63, 0.73}

\definecolor{fondtitre}{RGB}{85,85,85}
\definecolor{fonddeboite}{RGB}{232,232,232}
\definecolor{shadecolor}{RGB}{232,232,232}

\newcommand*{\boitesimple}[1]{%
\begin{center}
\begin{minipage}{15cm}
\begin{shaded}
	#1
\end{shaded}
\end{minipage}
\end{center}
}

\begin{document}

\title{On the exponential ergodicity of the McKean-Vlasov SDE depending on a polynomial interaction}
\author[1]{\href{https://math.univ-angers.fr/membre/ag-aboubacrine-assadeck/}{Mohamed Alfaki \textsc{Ag Aboubacrine Assadeck}}\orcidlink{0000-0002-3281-1954}$^{\spadesuit\clubsuit}$}
\affil[1]{\footnotesize{\textcolor{blue}{\href{http://www.univ-angers.fr/}{Univ Angers,} \href{https://www.cnrs.fr/fr}{CNRS,} \href{http://recherche.math.univ-angers.fr/}{LAREMA,} \href{https://sfrmathstic.univ-angers.fr/fr/index.html}{SFR MATHSTIC,} \href{https://www.angers.fr/}{F-49000 Angers, France }}}}
\date{\today}

\newgeometry{top=0cm, bottom=1.5cm}

\maketitle
\thispagestyle{empty}

\blfootnote{$^\spadesuit$Email address: \href{mailto:mohamedalfaki.agaboubacrineassadeck@univ-angers.fr}{mohamedalfaki.agaboubacrineassadeck}@univ-angers.fr \newline $^\clubsuit$\href{https://mon-portfolio-de-chercheur.webnode.fr/}{\textcolor{blue}{Homepage of Mohamed Alfaki \textsc{Ag Aboubacrine Assadeck}}}
}
\vspace{-3em}

\begin{abstract}
\noindent In this paper, we study the long time behaviour of the Fokker-Planck and the kinetic Fokker-Planck equations with \textit{many body interaction}, more precisely with interaction defined by $U$-statistics, whose macroscopic limits are often called McKean-Vlasov and Vlasov-Fokker-Planck equations respectively. In the  continuity of the recent papers \cite[\cite{Ref2},\cite{Ref3}]{Ref1} and  \cite[\cite{Monmarche1},\cite{Monmarche2}]{MonmarchGullin}, we establish \textit{nonlinear functional inequalities} for the limiting McKean-Vlasov SDEs related to our particle systems. In the first order case, our results rely on \textit{large deviations} for $U$-statistics and a \textit{uniform logarithmic Sobolev inequality} in the number of particles for the invariant measure of the particle system. In the kinetic case, we first prove a uniform (in the number of particles) \textit{exponential convergence to equilibrium} for the solutions in the weighted Sobolev space $H^1(\mu)$ with a rate of convergence which is explicitly computable and independent of the number of particles. In a second time, we quantitatively establish an exponential return to equilibrium in Wasserstein's $\mathcal{W}_{2}-$metric for the Vlasov-Fokker-Planck equation.\newline


\end{abstract}
{
\small
\hypersetup{hidelinks}
\setcounter{tocdepth}{2}
\tableofcontents
}
\newpage

\restoregeometry

\section{Introduction}
In the continuity of the recent papers  \cite{Ref2} and \cite{Ref3}, we establish exponential convergence towards equilibrium for a class of McKean-Vlasov and Vlasov-Fokker-Planck with \textit{polynomial interaction} (macroscopic interaction associated with $U$-statistics and defined in \cref{GFFE} and \cref{flatintrinpol}). Before going further into the details, we recall the general setting related to our problem.\newline

\noindent\textbf{General homogeneous McKean-Vlasov diffusion.} The processes studied in this paper belong to the following class of stochastic differential equations:
\begin{equation}\label{Pmcvgeneral}
dX_{t}=b(X_{t},\mathbb{P}_{X_{t}})dt+\sigma(X_{t},\mathbb{P}_{X_{t}})dB_{t},
\end{equation}
with respectively $b:\mathbb{R}^{D}\times\mathcal{P}(\mathbb{R}^{D})\longrightarrow\mathbb{R}^{D}$ the drift coefficient, $\sigma:\mathbb{R}^{D}\times\mathcal{P}(\mathbb{R}^{D})\longrightarrow\mathcal{M}_{D, p}(\mathbb{R})$ the diffusion coefficient and $(B_{t})_{t\geq 0}$ a standard $p-$dimensional Brownian motion. More precisely, we are interested in the study of exponential ergodicity of the process defined by
\begin{equation}\label{McProcess}
dX_{t}=-(\mathcal{D}_{m}F(\mathbb{P}_{X_{t}},X_{t})+\frac{\sigma^2}{2}\nabla V(X_{t}))dt+\sigma dB_{t},
\end{equation}
where $F:\mathcal{P}(\R^{D})\to\overline{\R}$, $\mathcal{D}_{m}F$ is the intrinsic derivative ($\mathbf{L}-$derivation or derivation in the sense of Fr\'echet of $F$ on the probability measure space, see \cref{aboutfi1} for precise definition) defined as $\mathcal{D}_{m}F(m,\cdot):=\nabla\frac{\delta F}{\delta m}(m,\cdot)$ (for example, if $F(m)=\int\varphi dm$, we have $\frac{\delta F}{\delta m}(m,x)=\varphi(x)$ then, $\mathcal{D}_{m}F(m,x)=\nabla\varphi(x)$), $V$ is a \textit{confinement potential} and $\sigma>0$ (in this paper, without loss of generality and for the sake of standardization, we take $\sigma=\sqrt{2}$). Note that
\begin{equation}
   \cref{McProcess}\Longleftrightarrow dX_{t}=-\nabla\frac{\delta H}{\delta m}(\mathbb{P}_{X_{t}},X_{t})dt+\sigma dB_{t}
\end{equation}
with the functional $H$ given by 
\begin{equation}
H(\mu):=F(\mu)+\frac{\sigma^2}{2}\int Vd\mu.
\end{equation}
In the sequel, $F$ will be assumed to be a \textit{polynomial of degree at least two on the probability space} (see \cref{GFFE}), so that $H$ is a \textit{homogeneous polynomial (without constant term) on the probability space}.\newline

\noindent\textbf{General mean-field generators and mean-field limits.}
A mean-field particle system is a system of $n$ particles characterised by a generator of the form
\begin{equation}\label{meanfieldsG}
    \forall\varphi\in\mathcal{A}\subset\mathcal{C}_{b}((\R^{D})^{n}),\quad\mathcal{L}_{n}\varphi(x):=\sum_{i=1}^{n}\mathcal{L}_{\mu_{x}}\blacksquare_{i}\varphi(x),\quad\textnormal{where}\quad\mu_{x}:=\frac{1}{n}\sum\delta_{x_{i}}
\end{equation}
and for a given probability measure $\mu\in\mathcal{P}(\R^{D})$, $\mathcal{L}_{\mu}$ is the generator of a Markov process on $\R^{D}$ defined by
\begin{equation}\label{MarkovG}
    \mathcal{L}_{\mu}:=b(\cdot,\mu)\cdot\nabla+\frac{1}{2}\mathbf{Tr}(\sigma\sigma^{*}(\cdot,\mu)\nabla^{2}),
\end{equation}
and the notation $\mathcal{L}\blacksquare_{i}\varphi$ denotes the action of an operator $\mathcal{L}$ defined on (a subset of) $\mathcal{C}_{b}(\R^{D})$ against the i-th variable of a function $\varphi\in\mathcal{C}_{b}((\R^{D})^{n})$; in other words, $\mathcal{L}\blacksquare_{i}\varphi$ is defined as the function:
$$
x\in(\R^{D})^{n}\longmapsto\mathcal{L}[y\longmapsto\varphi(x_{1},\ldots,x_{i-1},y,x_{i+1},\ldots,x_{n})](x_{i})\in\R.
$$
The $n-$particle generator (\cref{meanfieldsG}) associated to this class of diffusion generators induces the \textit{McKean-Vlasov diffusion process} given by \cref{Pmcvgeneral}. The associated $n-$particle process is governed by the following system of SDEs:
\begin{equation}\label{Pmeanfields}
    \forall i\in\{1,\ldots,n\},\quad dX^{i,n}_{t}=b(X^{i,n}_{t},\mu_{X^{n}_{t}})dt+\sigma(X^{i,n}_{t},\mu_{X^{n}_{t}})dB^{i}_{t},
\end{equation}
where $B^{1},\ldots,B^{n}$ are $n$ independent Brownian motions.\newline 
A \textit{kinetic particle} $Z^{i,n}_{t}:=(X^{i,n}_{t},V^{i,n}_{t})\in\R^{d}\times\R^{d}$ is a particle defined by two arguments, its position $X^{i,n}_{t}$ and its velocity $V^{i,n}_{t}$ defined as the time derivative of the position. The evolution of a system of kinetic particles is usually governed by Newton's laws of motion. In a random setting, the typical system of SDEs is thus the following:
\begin{equation}\label{PKinetic}
\forall i\in\{1,\ldots,n\},\quad
    \begin{cases}
        dX^{i,n}_{t}=V^{i,n}_{t}dt\\
        dV^{i,n}_{t}=\mathbf{F}(X^{i,n}_{t},V^{i,n}_{t},\mu_{X^{n}_{t}})dt+\sigma(X^{i,n}_{t},V^{i,n}_{t},\mu_{X^{n}_{t}})dB^{i}_{t},
    \end{cases}
\end{equation}
where $\mathbf{F}:\R^{d}\times\R^{d}\times\mathcal{P}(\R^{d})\to\R^{d}$ and $\sigma:\R^{d}\times\R^{d}\times\mathcal{P}(\R^{d})\to\mathcal{M}_{d}(\R)$. Note that it is often assumed that the force field induced by the interactions between the particles depends only on their positions. Thus, we consider 
\begin{equation}\label{mesuremep}
    \mu_{X^{n}_{t}}:=\frac{1}{n}\sum_{i=1}^{n}\delta_{X^{i,n}_{t}}
\end{equation}
instead of $\mu_{Z^{n}_{t}}$. Note that in the system \cref{Pmeanfields} there are actually $nD$ independent one-dimensional Brownian motions. In particular, for kinetic particles defined by their positions and velocities, the noise is often added on the velocity variable only (this case is nevertheless covered by \cref{Pmeanfields} with a block-diagonal matrix $\sigma$ with a vanishing block on the position variable). This special case of the McKean-Vlasov diffusion in $\R^{D}=\R^{d}\times\R^{d}$ is also often called a \textit{second order system} by opposition to the \textit{first order systems} when $\R^{D}=\R^{d}$. In this paper, we will establish some uniform exponential convergence of the particle systems \cref{CM1} and \cref{CM2} which in turn  will allow us to derive the same properties for their mean-field limiting dynamics.\newline

\noindent\textbf{Propagation of chaos.} \cref{Pmcvgeneral} appears naturally as the mean field limit of \cref{Pmeanfields}: This phenomenon is called \textit{chaos propagation}. The notion of propagation of chaos for large systems of interacting particles originates in statistical physics and has recently become a central notion in many areas of applied mathematics. Kac gave the first rigorous mathematical definition of chaos (\cite{MKac1}) and introduced the idea that for time-evolving systems (\cref{Pmeanfields}), chaos should be propagated in time, a property therefore called the \textit{propagation of chaos}. Kac was still motivated by the mathematical justification of the classical collisional kinetic theory of Boltzmann for which he developed a simplified probabilistic model. Soon after Kac, McKean (\cite{McKean1})  introduced a class of diffusion models (\cref{Pmcvgeneral}) which were not originally part of Boltzmann theory but which satisfy Kac's propagation of chaos property. In the classical kinetic theory of Boltzmann, the problem is the derivation of continuum models starting from deterministic, Newtonian, systems of particles. In comparison, the fundamental contribution of Kac and McKean is to have shown that the classical equations of kinetic theory also have a natural stochastic interpretation. This philosophical shift is addressed in the enlightening introduction of Kac (\cite{MKac}) written for the centenary of the Boltzmann equation.
On this topic, we refer to (among others) the seminal papers \cite[\cite{McKean1},\cite{MKac1},\cite{MKac},\cite{kacprogramm},\cite{Ref6},\cite{Talay},\cite{Meleard},\cite{Mario},\cite{propachaos1},\cite{BGSR},\cite{GSRT}]{Boltzmann3} and more recently \cite[\cite{JabinWang1},\cite{Jabin},\cite{Golse}]{JabinWang} and for applications, the reader may look at  \cite[\cite{Ref10},\cite{Hauray}]{games} in mean field games, \cite[\cite{optim2},\cite{optim3},\cite{optim4}]{optim1} in optimization, \cite[\cite{data2},\cite{data3}]{data1} and \cite[\cite{birthdeath},\cite{neural1},\cite{neural2}]{Lucasmini} in machine learning, \cite[\cite{Biologie2},\cite{Biologie3},\cite{Biologie4},\cite{Biologie5}]{Biologie1} in biology...\newline

\noindent\textbf{General McKean-Vlasov PDE.} Let us now focus on the martingale problem related to \cref{Pmcvgeneral} and on the associated PDE. It is classically assumed that the domain of the generator $\mathcal{L}_{\mu}$ does not depend on $\mu$. This domain will be denoted by $\mathcal{F}\subset\mathcal{C}_{b}(\R^{D})$. In that case, it is easy to guess the form of the associated nonlinear system obtained when $n\to+\infty$. Taking a test function of the form $\varphi(x_{1},\ldots,x_{n}):=\psi(x_{1})$, where $\psi\in\mathcal{F}$, one obtains the one-particle Kolmogorov equation:
\begin{equation}\label{kolmogorovmeanfield}
   \frac{d}{dt}\langle\mathbb{P}_{X^{1,n}_{t}},\psi\rangle=\int_{(\R^{D})^{n}}\mathcal{L}_{\mu_{x}}\varphi(x)\mathbb{P}_{X^{n}_{t}}(dx)=\mathbb{E}[\mathcal{L}_{\mu_{X^{n}_{t}}}\varphi(X^{n}_{t})].
\end{equation}
Note that the right-hand side depends on the $n-$particle distribution. If the limiting system exists (propagation of chaos) then, its law $\mu_{t}$ at time $t\geq 0$ is typically obtained as the limit of the empirical measure process:
\begin{equation}\label{lawlimit}
    \mu_{X^{n}_{t}}\overset{n\to+\infty}{\longrightarrow}\mu_{t}
\end{equation}
This also implies $\mathbb{P}_{X^{1,n}_{t}}\overset{n\to+\infty}{\longrightarrow}\mu_{t}$. Reporting formally in the previous equation, it follows that $\mu_{t}$ should satisfy
\begin{equation}\label{GMcKeanPDE}
    \bigg(\forall\varphi\in\mathcal{F},\quad\frac{d}{dt}\langle\mu_{t},\varphi\rangle=\langle\mu_{t},\mathcal{L}_{\mu_{t}}\varphi\rangle\bigg)\Longleftrightarrow\partial_{t}\mu_{t}=\mathcal{L}_{\mu_{t}}^{\dag}\mu_{t},\quad\textnormal{where}\quad\mathcal{L}^{\dag}_{\mu_{t}}\quad\textnormal{is the weak adjoint of}\quad\mathcal{L}_{\mu_{t}}.
\end{equation}
This is the \textit{weak form} of an equation that is called the \textit{(nonlinear) evolution equation}. Note that the evolution equation is nonlinear due to the dependency of $\mathcal{L}_{\mu}$ on the measure argument $\mu$. This is a very analytical derivation. Its stochastic equivalent is given by a \textit{nonlinear martingale problem}:
\begin{definition}[Nonlinear Martingale Problem]\label{NLMP}
 Let $T\in(0,+\infty)$ and let us write $I:=[0,T]$. A pathwise law $\mu_{I}\in\mathcal{P}(\mathcal{D}(I,\R^{D}))$ is said to be a solution of the nonlinear mean-field martingale problem issued from $\mu_{0}\in\mathcal{P}(\R^{D})$ whenever $\forall\varphi\in\mathcal{F}$,
 \begin{equation}\label{EquaNLMP}
     M^{\varphi}_{t}:=\varphi(X_{t})-\varphi(X_{0})-\int_{0}^{t}\mathcal{L}_{\mu_{s}}\varphi(X_{s})ds,
 \end{equation}
is a $\mu_{I}$ martingale, where $(X_{t})_{t\in I}$ is the canonical process and for $t\geq 0$, $\mu_{t}:=(X_{t})_{\sharp}\mu_{I}$. The natural filtration of the canonical process is denoted by $\sigma_{X}.$
\end{definition}
\noindent Note that $\mu_{I}$ contains a priori much more information than the evolution equation \cref{GMcKeanPDE} and as the notation implies, $\mu_{t}:=(X_{t})_{\sharp}\mu_{I}$ solves the evolution equation.  If the nonlinear martingale problem is wellposed then the canonical process $(X_{t})_{t\in I}$ is a time inhomogeneous
Markov process on the probability space $(\mathcal{D}(I,\R^{D}),\sigma_{X},\mu_{I})$. This Markov process is called nonlinear in the sense of McKean or simply nonlinear for short. In other words, \cref{Pmcvgeneral} defines the stochastic process whose evolution of laws is governed by the evolution PDE \cref{GMcKeanPDE}.\newline
The evolution equation \cref{GMcKeanPDE} can be written in a \textit{strong form} (at least formally) and reads:
\begin{equation}\label{GNFPlanck}
    \partial_{t}\mu_{t}(x)=-\nabla_{x}\cdot(b(x,\mu_{t})\mu_{t})+\frac{1}{2}\sum_{i,j=1}^{D}\partial_{x_{i}}\partial_{x_{j}}\bigg((\sigma\sigma^{*})_{i,j}(x,\mu_{t})\mu_{t}\bigg).
\end{equation}
This is a \textit{nonlinear Fokker-Planck equation} which is used in many important modelling problems. This equation was obtained (formally) previously using only the generators when $n\to+\infty$. Here, there is an alternative way to derive the limiting system: looking at the SDE system \cref{Pmeanfields}, the empirical measure can be formally replaced by its expected limit $\mu_{t}$. Since all the particles are \textit{exchangeable}, this can be done in any of the $n$ equations.  The result is a process $(\overline{X}_{t})_{t\geq 0}$ which solves the SDE: (McKean-Vlasov process)
\begin{equation}\label{PMckeanG}
    d\overline{X}_{t}=b(\overline{X}_{t},\mu_{t})dt+\sigma(\overline{X}_{t},\mu_{t})dB_{t},
\end{equation}
where $(B_{t})_{t\geq 0}$ is a Brownian motion and $\overline{X}_{0}\sim\mu_{0}$. Moreover, since for all $i$, $X^{i,n}_{t}$ has law $\mathbb{P}_{X^{1,n}_{t}}$ and since it is expected that $\mathbb{P}_{X^{1,n}_{t}}\overset{n\to+\infty}{\longrightarrow}\mu_{t}$, the process $(\overline{X}_{t})_{t\geq0}$ and the distributions $(\mu_{t})_{t\geq 0}$ should be linked by the relation: for all $t\geq0$, $\overline{X}_{t}\sim\mu_{t}$. The dependency of the solution of a SDE on its law  is a special case of what is called a
nonlinear process in the sense of McKean (\cref{Pmcvgeneral} is equivalent to \cref{GMcKeanPDE} via mean-field system given by \cref{meanfieldsG} and nonlinear martingale problem given by \cref{EquaNLMP}). Note that when $\sigma=0$, the limit equation \cref{GNFPlanck} is the renowned \textit{Vlasov equation} which is historically one of the first and most important models in plasma physics and celestial mechanics. Equivalently, our main objective is the study of the long-time behavior of the solution flow of the nonlinear ($\mathcal{D}_{m}F$ must at least depend on the measure otherwise we find the standard Fokker-Planck PDE) Fokker-Planck equation:
\begin{equation}\label{McPDE}
\partial_{t}m=\nabla\cdot\bigg((\mathcal{D}_{m}F(m,\cdot)+\frac{\sigma^2}{2}\nabla V)m+\frac{\sigma^2}{2}\nabla m\bigg).
\end{equation}

\noindent\textbf{General framework.} Under appropriate conditions, the process \cref{PMckeanG}  is well defined or (equivalently) that the PDE \cref{GNFPlanck} or the martingale problem \cref{EquaNLMP} are wellposed. The result in \cite[Proposition.1]{propachaos1} (or \cref{ExistuniqueMcV}) gives the reference framework in which all these objects are well defined. Depending on the form of the drift and diffusion coefficients, the McKean-Vlasov diffusion can be used in a wide range of \textit{modelling problems}. The first case is obtained when $b$ and $\sigma$ depend linearly on the measure argument. Namely, for $n,m\in\N$, let us consider two functions $K_{1}:\R^{d}\times\R^{d}\to\R^{n}$, $ K_{2}:\R^{d}\times\R^{d}\to\R^{m}$, and let us take $b(x,\mu)=\overline{b}(x,K_{1}\star\mu(x))$, $\sigma(x,\mu)=\overline{\sigma}(x,K_{2}\star\mu(x))$, where $\overline{b}:\R^{d}\times\R^{n}\to\R^{d}$, $\overline{\sigma}:\R^{d}\times\R^{m}\to\mathcal{M}_{d}(\R)$ and  $K_{i}\star\mu(x):=\int K_{i}(x,y)\mu(dy)$. When $K_{1},K_{2}$ and  $\overline{b},\overline{\sigma}$ are Lipschitz and bounded, the propagation of chaos result is the given by McKean's theorem.\newline
In many applications, $\sigma$ is a constant diffusion matrix, $K_{1}(x,y)\equiv K(y-x)$ for a fixed symmetric radial kernel $K:\R^{d}\to\R^{d}$ and $b(x,\mu)=K\star\mu(x)$. The case where $K$ has a singularity is much more delicate but contains many important cases. For instance, in fluid dynamics, when $K$ is the Biot-Savart kernel $K(x)=\frac{x^{\perp}}{|x|^2}$ in dimension $d=2$ (defining $(x_{1},x_{2})^{\perp}=(-x_{2},x_{1})$) and $\sigma(x,\mu)=\sqrt{2\sigma}\mathbf{Id}_{2}$ for a fixed $\sigma>0$, the limit
Fokker-Planck equation reads:
\begin{equation}\label{FK2D}
    \partial_{t}\mu_{t}+\nabla\cdot(\mu_{t}K\star\mu_{t})=\frac{\sigma^2}{2}\Delta\mu_{t},
\end{equation}
By translation invariance, the quantity $\omega_{t}=\mu_{t}-1$ is the solution of the famous \textit{vorticity equation} which can be shown to be equivalent to the \textit{2D incompressible Navier-Stokes system} (see \cite{JabinWang1}). The case of \textit{gradient systems} is a sub-case of the previous one when $\sigma(x,\mu)=\sigma\mathbf{Id}$ for a constant $\sigma>0$ and
\begin{equation}\label{gradientsystemcase1}
    b(x,\mu)=-\nabla\mathbf{V}(x)-\int_{\R^{d}}\nabla W(x-y)\mu(dy)
\end{equation}
where $\mathbf{V}, W$ are two \textit{symmetric potentials} on $\R^{d}$ respectively called the \textit{confinement potential} and the \textit{interaction potential}. The limit Fokker-Planck equation
\begin{equation}\label{milieuxgranulaires}
    \partial_{t}\mu_{t}=\frac{\sigma^2}{2}\Delta\mu_{t}+\nabla\cdot\bigg(\mu_{t}\nabla(\mathbf{V}+W\star\mu_{t})\bigg),
\end{equation}
is called the \textit{granular-media equation}. The general case \cref{Pmcvgeneral}, where $b$ and $\sigma$ have a possibly nonlinear dependence on $\mu$ can be extended to even more general cases. A simple extension is the case of time-dependent functions $b$ and $\sigma$ (see e.g. \cite{propachaos1}).
In this article, we consider a polynomial dependence in the measure $\mu$ induced by \textit{order statistics (many-body interaction)} in order to generalize the results obtained in the case of a linear interaction in the measure $\mu$ defined by the convolution via a potential two-body interaction (\cite{Ref2}, \cite{Ref3}). More exactly, under adequate assumptions (\ref{Hypo1},\ref{Hypo2}), we are interested in the exponential return to equilibrium of the solution of \cref{McPDE} in the case
\begin{equation}\label{GFFE}
    F(\mu)=\sum_{k=2}^{N}\int W^{(k)}d\mu^{\otimes k},
\end{equation}
where $\forall k\in\{2,\ldots,N\}$, $W^{(k)}$ is a \textit{symmetric interaction potential} between $k$ particles and $N$ represents the number of such potentials. The intrinsic derivative $\mathcal{D}_{m}F(\nu,y)$ associated with this functional is given by
\begin{equation}\label{flatintrinpol}
   \nabla\frac{\delta F}{\delta m}(\nu,y)= \sum_{k=2}^{N}\sum_{j=1}^{k} \int\nabla_{x_{j}}W^{(k)}(x_{1},\ldots,x_{j-1},y,x_{j+1},\ldots,x_{k})\nu^{\otimes k-1}(dx_{1},\ldots,dx_{j-1},dx_{j+1},\ldots,dx_{k})
\end{equation}
The associated microscopic (particle-level) interaction is given by ($U-$statistic of order $k$ and kernel $\Phi\equiv W^{(k)}$)
\begin{equation}\label{U-stats1}
    U_{n}(W^{(k)}):=\frac{k!(n-k)!}{n!}\sum_{1\leq i_{1}<\ldots<i_{k}\leq n}W^{(k)}(X^{i_{1},n},\ldots,X^{i_{k},n}),\quad\textnormal{where } X^{n}:=(X^{1,n},\ldots,X^{n,n})\in(\R^{D})^{n}.
\end{equation}
 $U(X^{n}):=U_{n}(\Phi)$ is called $U-$statistic of order $k$ and kernel $\Phi$ associated with the sample $X^{n}.$ This statistic corresponds to the arithmetic mean of the kernel $\Phi$ over all the parts at $ k$ elements of the set of sample values. we often write $U_{n}(W^{(n)})(X^{n}):=:U(X^{n})$. We generalize this definition to the space of probabilities by the functional
\begin{equation}\label{U-stats2}
    \mu\in\mathcal{P}(\R^{D})\longmapsto\int_{\R^{kD}}\Phi d{\mu^{\bigotimes k}},
\end{equation}
called monome of degree $k$ and coefficient $\Phi$ on the probability space $\mathcal{P}(\R^{D})$. The link between these two microscopic and macroscopic interactions is given by
\begin{equation}
    \sum_{k=2}^{N}U_{n}(W^{k})=F(\mu_{X^{n}}).
\end{equation}
The \textit{granular-media equation} \cref{milieuxgranulaires} is given by \cref{McPDE} when 
\begin{equation}
    F(\mu)=\int W^{(2)}(x,y)\mu(dx)\mu(dy)\quad W^{(2)}(x,y)\equiv\frac{1}{2}W(x-y)\quad V(x)\equiv\frac{2}{\sigma^2}\mathbf{V}(x).
\end{equation}
Indeed, in this case, we have
\begin{align}
    \frac{\delta F}{\delta m}(\mu,x)=\int W^{(2)}(x,y)\mu(dy)+\int W^{(2)}(y,x)\mu(dy)=\int W(x-y)\mu(dy)=:W\star\mu(x) \\ 
    \mathcal{D}_{m}F(\mu,x):=\nabla\frac{\delta F}{\delta m}(\mu,x).
\end{align}

\noindent\textbf{Energy and Large Deviations.} Consider $G:\mathcal{M}^{p}_{1}(\R^{D})\to\overline{\R}$ (which can be nonlinear) and a probability (Gibbs) measure $\alpha$ associated with potential $V$. For any $\sigma>0$, we put
\begin{equation}\label{Rate2}
        V^{\sigma,G}(m):=G(m)+\frac{\sigma^2}{2}\mathbf{H}[m|\alpha].
\end{equation}
$V^{\sigma,G}$ is an energy function regularised by the $\mathbf{KL-}$divergence $\mathbf{H}[m|\alpha]$ which is given by \cref{defrelent} in \cref{defnot}. It is known (see e.g. \cite[Proposition.2.5]{Lucasmini}) that $V^{\sigma,G}$ is minimized by a measure $m^{\sigma,\star}$ satisfying the following fixed point problem (it is noteworthy that the variational form of the invariant measure of the classic Langevin equation is a particular example of this first order condition)
\begin{equation}\label{Pfix1}
m^{\sigma,\star}(dx)=\frac{1}{Z_{\sigma}}e^{-\frac{2}{\sigma^2}(\frac{\delta G}{\delta m}(m^{\sigma,\star},x)+\frac{\sigma^2}{2}V(x))}dx,
\end{equation}
where $Z_{\sigma}$ is the normalising constant, and for any $m\in\mathcal{M}^{p}_{1}(\R^{D})$ and $x\in\R^{D}$, $\frac{\delta G}{\delta m}(m,x)$ denotes the flat derivative of $G$ with respect to $m$, in the direction of $x$, evaluated at $m$. For any $\Theta_{1},\Theta_{2}\in\mathcal{M}^{p}_{1}(\R^{D})$, the function $\frac{\delta G}{\delta m}:\mathcal{M}^{p}_{1}(\R^{D})\times\R^{D}\to\R$ satisfies
\begin{equation}\label{aboutfi1} 
G(\Theta_{2})-G(\Theta_{1})=\int_{0}^{1}\int_{\R^{D}}\frac{\delta G}{\delta m}(\Theta_{1}+\lambda(\Theta_{2}-\Theta_{1}),x)(\Theta_{2}-\Theta_{1})(dx)\lambda_{\R}(d\lambda).
\end{equation}
This notion of derivative appears in the literature under several different names, including the linear functional derivative (see e.g \cite[Section.5.4.1]{flatderiv}) or the first variation \cite{ambrosio}. It is important to note that $\frac{\delta G}{\delta m}$ is defined only up to a constant, i.e., for any $c$, the function $\frac{\delta G}{\delta m}+c$  is also a flat derivative of $G$. Everywhere in this paper we will adopt a normalizing convention requiring
\begin{equation}
\int_{\R^{D}}\frac{\delta G}{\delta m}(m,x)m(dx)=0,\quad\textnormal{which then makes the choice of the constant unique.}
\end{equation}
Note that 
\begin{equation}
\mathcal{D}_{m}\mathbf{H}[\cdot|\alpha](\nu,y)=\nabla\log\bigg(\frac{d\nu}{d\alpha} \bigg)(y).
\end{equation}
Large Deviation Principles imply propagation of chaos, but they do not always give a way to quantify it since the related results are often purely asymptotic (for instance, Sanov theorem
is non-quantitative). Nevertheless, the results of large deviations turn out to be very useful for the technical passages in the macroscopic limits: when one makes tend the number of particles to infinity. In the seminal article \cite{LDP1}, the authors improve results from \cite{LDP2} and \cite{LDP3} on Large
Deviation Principles (LDP) for Gibbs measures and obtain as a byproduct a pathwise propagation of chaos result for the McKean-Vlasov diffusion. Firstly, \cite[Theorem.A]{LDP1} (or \cref{LDPPolinter}) states a large deviation principle for Gibbs measures with a polynomial potential. \cite[Theorem.B]{LDP1} quantifies the fluctuations of $\mu_{X^{n}}$ in the non-degenerate case. Analogous results for the degenerate case are given in \cite[Theorem.C]{LDP1}. For more details, see also \cite[Theorem.4.7, Corollary.3]{propachaos1}. We use the large deviations results obtained on the order statistics in \cite{Ref1}: In addition to the fact that the mean-field entropy functional (\cref{Rate1} or $V^{\sqrt{2},F}$ defined by \cref{Rate2}) is a rate function (\cref{LDPPolinter}) for the random empirical measure $\mu_{X^{n}}$, the authors show that it is a good rate function that has good tensorization properties.\newline 

\noindent\textbf{Long time behavior.} In the present paper, we are concerned by the long-time convergence towards the solution to an optimization problem on the subspace $\mathcal{M}^{p}_{1}(\R^{D})$ of probability measures $\mathcal{M}_{1}(\R^{D})$: we consider a function $\mathbf{E}:\mathcal{M}^{p}_{1}(\R^{D})\to\overline{\R}$ and we want to find a minimizing measure $m^{\star}:=\mathbf{arginf}_{\mathcal{M}^{p}_{1}(\R^{D})}\mathbf{E}$ such that for a \textit{gradient flow} (see e.g. \cite{ambrosio} and \cite{ottogradients}) $(m_{t})_{t\geq 0}$ associated with $\mathbf{E}$, we have an exponential estimate of the deviation $\mathbf{E}(m_{t})-\mathbf{E}(m^{\star})$ of the form (with $C\geq 1$ and $\rho>0$)
\begin{equation}\label{defhypocoercivity}
\mathbf{E}(m_{t})-\mathbf{E}(m^{\star})\leq C(\mathbf{E}(m_{0})-\mathbf{E}(m^{\star}))e^{-\rho t}.
\end{equation}
\cref{defhypocoercivity}-type Inequalities are called \textit{hypocoercive inequalities}. We call $\mathbf{E}-\mathbf{E}(m^{\star})$ the \textit{entropy functional} (\cite{Entmethods},\cite{BGL-MarkovDiffusion},\cite{Doljean}) of the system and $-\frac{d}{dt}(\mathbf{E}(m_{t})-\mathbf{E}(m^{\star}))$ the \textit{production of entropy} (usually called energy in mathematical literature). Clausius invents the concept of entropy, Boltzmann proposes to derive entropy along the flow. Generally speaking, an entropy is a \textit{Lyapunov functional} of a specific form. It is however hard (and even somewhat artificial) to give a formal narrow definition of entropies that distinguishes them from, say, energies. An entropy is a quantity calculated from a solution, which decreases over time when the solution obeys an evolution equation, and which is stationary only for the stationary solutions of the equation.
In conclusion, the concept of entropy is a tool that adapts to what we want to study. The notion of \textit{hypocoercivity} was proposed by T. Gallay. The objective is typically to control the entropy at time $t$ by the initial entropy multiplied by a constant $C$ (always greater than $1$) and a exponential decay factor, with exponential decay rate as good as possible in big time. This theory is inspired by the \textit{hypoelliptic theory of L. Hormander}, and the terminology hypocoercivity accounts for the relationship between entropy and its derivative with respect to $t$. There would be \textit{coercivity} if $C = 1$, which is clearly not possible in most cases considered in kinetic theory. It is well known that, for the standard \textit{Langevin equation} of Hamiltonian $V$ (given by \cref{McProcess} in the case $F\equiv0$), the following assertions are equivalent:
\begin{equation}\label{ISLdef}
    \exists\rho>0\forall\varphi\in\mathcal{C}^{\infty}_{c}(\mathbb{R}^{d}),\quad\rho\textbf{Ent}_{\alpha}[\varphi^{2}]\leq 2\int||\nabla\varphi||^{2}d\alpha.
\end{equation}
\begin{equation}\label{ISLdual}
    \exists\rho>0,\quad\rho\mathbf{H}[\cdot|\alpha]\leq 2\mathbf{I}[\cdot|\alpha].
\end{equation}
\begin{equation}
    \exists\rho>0\forall t\geq 0,\quad\mathbf{H}[\mu^{V}_{t}|\alpha]\leq\mathbf{H}[\mu^{V}_{0}|\alpha ]e^{-\rho t}.
\end{equation}
These three equivalent assertions imply the $T2-$\textit{Talagrand inequality}
\begin{equation}
    \rho\mathcal{W}^{2}_{2}(\cdot,\alpha)\leq 2\mathbf{H}[\cdot|\alpha],
\end{equation}
inequality which, in turn, implies an \textit{exponential contraction in wasserstein metric} $\mathcal{W}_{2}$, i.e. the exponential convergence of the flow $(\mu^{V}_{t })_{t\geq 0}$ (solution of the Fokker-Planck equation associated with the standard Langevin process of Hamiltonian $V$) to the maxwellian (invariant measure of the Langevin process that can also be seen from equivalently as the unique $\textbf{argmin}_{\mathcal{P}(\mathbb{R}^{d})}\mathbf{H}[\cdot|\alpha]$) $ \alpha$ of the \textit{Fokker-Planck PDE} given by \cref{McPDE} in the case $F\equiv0$:
\begin{equation}
    \forall t\geq 0,\quad\mathcal{W}_{2}^{2}(\mu^{V}_{t},\alpha)\leq\frac{2}{\rho}\mathbf{H}[\mu^{V}_{0}|\alpha]e^{-\rho t}.
\end{equation}
\cref{ISLdef} and \cref{ISLdual} respectively define the \textit{logarithmic Sobolev inequality} (\cite{BGL-MarkovDiffusion}) and its \textit{dual version}. According to the \textit{dimension curvature criterion of Bakry-Emery}, we have 
\begin{equation}
    \bigg(\exists\rho>0\forall (x,h)\in\R^{d}\times\R^{d},\quad\langle\nabla^2V(x)h,h\rangle\geq\rho||h||^{2}_{2}\bigg)\Longrightarrow\cref{ISLdef}.
\end{equation}
Note that in the case of the symmetric Langevin-Kolmogorov process, we have
\begin{equation}
    m_{t}=\mu^{V}_{t},\quad m^{\star}=\alpha,\quad\mathbf{E}=\mathbf{H}[\cdot|\alpha],\quad\mathbf{E}(m_{t})-\mathbf{E}(m^{\star})=\mathbf{H}[\mu^{V}_{t}|\alpha],
\end{equation}
\begin{equation}
   -\frac{d}{dt}(\mathbf{E}(m_{t})-\mathbf{E}(m^{\star}))=-\frac{d}{dt}\mathbf{H}[\mu^{V}_{t}|\alpha]=\mathbf{I}[\mu^{V}_{t}|\alpha].
\end{equation}
The objective of this work is to identify a flow of measures $(m^{\sigma,F}_{t})_{t\geq 0}$ (flow solution of \cref{McPDE}) such that 
\begin{equation}
V^{\sigma,F}(m^{\sigma,F}_{t})-V^{\sigma,F}(m^{\sigma,\star})\overset{t\to+\infty}{\longrightarrow}0,
\end{equation}
as well as conditions (\ref{Hypo1},\ref{Hypo2}) that ensure that this convergence is exponential. To this end, we equip the space $\mathcal{M}^{p}_{1}(\R^{D})$ with a suitable distance function $\mathbf{d}:\mathcal{M}^{p}_{1}(\R^{D})\times\mathcal{M}^{p}_{1}(\R^{D})\to\R_{+}$ and consider a corresponding gradient flow, where the form of the flow is dictated by the choice of $\mathbf{d}$. Such a problem has been dealt with in the case of the Fisher-Rao metric (see \cite{birthdeath}): the authors established from a Polyak-Lojasiewicz inequality the exponential convergence of the gradient flow $(m^{\sigma,G}_{t})_{t\geq 0}$ described by the birth-death equation along $V^{\sigma,G}$ towards $V^{\sigma,G}(m^{\sigma,\star})$. In our case, \cref{defhypocoercivity} implies the exponential decay in \textbf{d}-metric (transport distance): 
\begin{equation}\label{CVexpD}
    \mathbf{d}(m^{\sigma,F}_{t},m^{\sigma,\star})\leq\gamma(V^{\sigma,F}(m^{\sigma,F}_{0})-V^{\sigma,F}(m^{\sigma,\star}))e^{-\rho t}.
\end{equation}
\cref{CVexpD} is a consequence of transport inequalities (see \cite{Villani}).
Moreover, given a measure $m^{\sigma,\star}$ satisfying the first order condition \cref{Pfix1}, it is formally a stationary solution to \cref{McPDE} called the \textit{Maxwellian of the McKean-Vlasov PDE}. Therefore, formally, we have already obtained the correspondence between the \textit{minimiser of the free energy function} and the \textit{invariant measure} of \cref{McProcess}. In this paper, the connection is rigorously proved mainly with a probabilistic argument.
The study of stationary solutions to nonlocal, diffusive \cref{McPDE} is classical topic with it roots in statistical physics literature and with strong links to Kac's program in Kinetic theory \cite{kacprogramm}. We also refer reader to the excellent monographs \cite{ambrosio} and \cite{BakryEmery}. An important issue is the \textit{long-time behaviour of gradient systems} which is often studied under \textit{convexity assumptions} on the potentials. In particular, variational approach has been developed in \cite{Ref9} and \cite{ottogradients} where authors studied \textit{dissipation of entropy} for granular media equations \cref{milieuxgranulaires} with the symmetric interaction potential of convolution type (interaction potential corresponds to term $\mathcal{D}_{m}F$ in \cref{McPDE}). Following on from the work done in \cite{ottogradients} and \cite{Ref9} (among others) on the long-time behavior of \cref{milieuxgranulaires}, in \cite{Ref2}, the authors proved via a \textit{uniform logarithmic Sobolev inequality} in the number of particles that
\begin{equation}\label{GuillinWu}
    \forall t\geq 0,\quad H_{W}[\nu_{t}]\leq H_{W}[\nu_{0}]e^{-\rho_{ LS}\frac{t}{2}}\quad\textnormal{and}\quad\mathcal{W}^{2}_{2}(\nu_{t},\nu_{\infty})\leq\frac{2}{\rho_{LS}}H_{W}[\nu_{0}]e^{-\rho_{LS}\frac{t}{2}}.
\end{equation}
\cref{GuillinWu} translates the exponential decrease of the \textit{mean field entropy} $H_{W}$ (given by \cref{defhypocoercivity} with $\mathbf{E}=V^{\sigma,F}$) and the \textit{contraction} in Wasserstein metric ($\mathbf{d}=\mathcal{W}_{2}$) of the solution flow of \cref{McPDE} in the case
\begin{equation}
    \sigma=\sqrt{2}\quad\textnormal{and}\quad F(\mu)=\frac{1}{2}\int W(x,y)\mu(dx)\mu(dy).
\end{equation}
The study of the long-time behaviour for the VFP equation is often more difficult than that of the McKean-Vlasov equation because of two reasons: 
\begin{enumerate}
    \item it is a degenerate diffusion process where the Laplacian acts only on the volocity variable and; \item it is not a gradient flows but simultaneously  presents both Hamiltonian and gradient flows effects.
\end{enumerate}
In \cite{MonmarchGullin}, combining the results of \cite{Ref2} and \cite{Monmarche1}, the trend to equilibrium in large time is studied for a large particle system (given by \cref{CM2} in case of a two-body interaction) associated to a \textit{Vlasov-Fokker-Planck equation} by the authors: they showed that under some conditions (that allow \textit{non-convex confining potentials}), the convergence rate is proven to be independent from the number of particles. From this are derived \textit{uniform in time propagation of chaos estimates} and an \textit{exponentially fast convergence for the nonlinear equation itself}.\newline

\noindent\textbf{Contributions.} In this paper, we are going to prove
\begin{enumerate}
    \item entropic convergence to equilibrium for the nonlinear McKean-Vlasov SDE (mean field limit of the first order system given by \cref{CM1}) generalizing results (given in \cref{GuillinWu}) of \cite{Ref2}.
    \item by \textit{Villani's hypocoercivity theorem} (see e.g. \cite[Theorem.3]{Ref3} or \cite[Theorem.35]{villanihypo}) the $H^{1}-$convergence for the kinetic Fokker-Planck equation with mean field interaction given by \cref{CM2}.
    \item exponential convergence towards equilibrium in metric $\mathcal{W}_{2}-$Wasserstein for the flow solution of the Vlasov-Fokker-Planck equation: mean field limit of the second order system given by \cref{CM2}.
\end{enumerate}
In the literature, these results are obtained by purely analytical tools such as, among others, the \textit{gradient flow structure}, the \textit{Gronwall lemma}. In this paper, we give \textit{rigorously probabilistic proofs} (see \cref{sketchandpreli}, \cref{fig1} and \cref{proofsmain}) based directly on the \textit{propagation of chaos}, the \textit{large deviations principle} (see \cref{Cesaro1} and \cref{Cesaro2}), the \textit{uniform log-Sobolev inequality} (see \cref{transport}), Villani's hypocoercivity (\cite[Theorem.3]{Ref3} or \cite[Theorem.18 and Theorem.35]{villanihypo}) theorem (see \cref{prelvfp2}) and \textit{Hormander's form} (see e.g. respectively Theorem.7 and Theorem.10 in \cite[\cite{Monmarche2}]{Monmarche1}). The fact that the interaction is polynomial is important in calculations, among other things, for passing to the limit in the number of particles: technical passage to the limit given by LDP.\newline

\noindent\textbf{Plan of the paper.}
Let us finish this introduction by the plan of the paper. In the next three sections, we will present
our mean field systems (\cref{CM1},\cref{CM2}), our set of assumptions (\ref{Hypo1},\ref{Hypo2}) and the main results (in \cref{main}) of the paper concerning logarithmic Sobolev inequality of mean field particles systems as well as exponential convergences to equilibrium for McKean-Vlasov (\cref{theo1},\cref{theo2}), kinetic Fokker-Planck (\cref{theo3}) and Vlasov-Fokker-Planck (\cref{theo4})  SDEs. In \cref{sketchandpreli}, we sketch a proof of our results and we introduce the pre-proof tools. In \cref{proofsmain}, we prove our main results. And we end the paper with the appendix, the acknowledgments and the bibliographical references.

\section{Notations and Definitions}\label{defnot}
We try to keep coherent definitions and notations throughout the article, but as the various objects and what they represent may become confusing, we list them here for reference :\newline

\noindent\textbf{Notations.}
 We note $|||\cdot|||_{H^{1}\to H^{1}}$ the operator norm associated with the weighted Sobolev $H^1(\mu^{n}_{Z})$ space induced by the invariant measure $\mu^{n}_{Z}$ of our second-order system given by \cref{CM2}. We have
 \begin{align}
     H^{1}(\mu^{n}_{Z}):=\bigg\{\varphi\in L^{2}(\mu^{n}_{Z}),\quad\nabla\varphi\in (L^{2}(\mu^{n}_{Z}))^{n}\bigg\},\quad
     ||\varphi||^2_{H^{1}}:=||\varphi||^2_{L^2(\mu^{n}_{Z})}+\int\bigg( ||\nabla_{x}\varphi||^2_{2}+||\nabla_{v}\varphi||^2_{2}\bigg)d\mu^{n}_{Z}.
 \end{align}
 $(B_{t})_{t\geq 0}$ represents the standard Brownian motion. We consider $((B^{i})_{t\geq 0})_{i\in\{1,\ldots,n\}}$ $n$ independent copies of $(B_{t})_{t\geq 0}$. For all $n\geq 1$, $\mathfrak{G}_{n}$ is the n-th symmetric group. For all $p\in [1+\infty)$, the Wasserstein $p$-distance between two probability measures $\mu$ and $\nu$ on $\R^{D}$ with finite  $p$-moments is given by
 \begin{align}
     \mathcal{W}_{p}(\mu,\nu):=\bigg(\inf_{\gamma\in\Gamma(\mu,\nu)}\int_{\R^{D}\times\R^{D}} ||x-y||^{p}_{2}\gamma(dxdy)\bigg)^{\frac{1}{p}},\quad
     \Gamma(\mu,\nu):=\bigg\{\gamma\in\mathcal{P}(\R^{D}\times\R^{D}),\quad\pi_{1}\gamma=\mu\quad\textnormal{and}\quad\pi_{2}\gamma=\nu\bigg\}.
\end{align}
 We note $\mathcal{M}^{p}_{1}(\R^{D})$ the space of probability measures with finite $p-$moments and $||\cdot||_{\mathbf{op}}$ the matrix subordinate norm.\newline
 
\noindent\textbf{Definitions.}\newline
\textit{Good rate function:} We recall the definition of a rate function on a Polish space $E$ and
the LDP for a sequence of probability measures on $(E, \mathcal{B}(E))$. $I$ is said to be a rate function on $E$ if it is a lower semi-continuous function from $E$ to $[0, \infty]$ (i.e., for all $L\geq 0$, the level set $\{I\leq L\}$ is closed). $I$ is said to be a good rate function if it is inf-compact, i.e. $\{I\leq L\}$ is compact for any $L\in\R$. A consequence of a rate function being good is that its infimum is achieved over any non-empty closed set.\newline
$\Phi$-\textit{entropy:} Let $(\Omega,\mathcal{T},\mathbb{P})$ be a probability space, $\Phi:I\subset\R^{D}\longrightarrow\R$ convex and $X:\Omega\longrightarrow \ I\subset \R^{D}$ a random vector such as $X\in L^{1}$, $\Phi(X)\in L^{1}$ and $\E[X]\in I$. We call $\Phi$-entropy of $X$, the quantity defined by:
\begin{equation}
\E^{\Phi}[X]:=\E[\Phi(X)]-\Phi(\E[X]).
\end{equation}
By assumptions, it is easy to see that $\mathcal{D}^{\Phi}:=\mathbf{Dom}(\E^{\Phi})\subset L^{1}$ is convex and moreover, $\E^{\Phi}:\mathcal{D}^{\Phi}\longrightarrow\R_{+}$ by Jensen's inequality.
\begin{remark}
Variance and entropy are examples of $\Phi-$entropies for $\Phi=||\cdot||^{2}$ and $\Phi:x\longmapsto x\log(x).$ It is easy to prove (by the theorem of the orthogonal projection on a closed convex set of a Hilbert space) that the variance of $X$ is exactly the square of the distance in norm $L^{2}$ of $X$ to the subspace of almost surely constant random variables, that is:
\begin{equation}
\mathbb{V}[X]=\underset{\lambda\in\R}{\mathbf{inf}}\textnormal{ }\E[|X-\lambda|^{2}]=\underset{\lambda \in\R}{\mathbf{inf}}\textnormal{ }\E[\Phi(X)+\Phi(\lambda)-\Phi'(\lambda)X]
\end{equation}
and this lower bound is reached at $\lambda=\E[X]$. We talk about variational formulation of $\Phi-$entropies (Monge-Kantorovitch duality and Wasserstein spaces). Moreover, we have:
\begin{equation}
\mathbf{Ent}[(1+\varepsilon X)^{2}]=2\varepsilon^{2}\mathbb{V}[X]+\mathbf{O}(\varepsilon^{3}).
\end{equation}
If $\Phi$ is strictly convex, then the $\Phi-$entropy of $X$ is zero if and only if $X$ is constant almost surely.
\end{remark}
\noindent\textit{Relative entropy:} Let $\mu\in\mathcal{P}(\R^{D})$. We define $\mathbf{H}[\cdot|\mu]:\mathcal{P}(\R^{d})\longrightarrow[0,+\infty]$ such that 
\begin{equation}\label{defrelent}
\mathbf{H}[\nu |\mu]=\begin{cases}
\E_{\nu}[\log\frac{d\nu}{d\mu}]=:&\mathbf{Ent}_{\mu}[\frac{d\nu}{d \mu}]\quad\textnormal{if }\nu\ll\mu,\\
&+\infty\quad\quad\quad\textnormal{otherwise.}
\end{cases}
\end{equation}
And we recall that in the first case of absolute continuity, $\frac{d\nu}{d\mu}$ is the \textit{Radon-Nikodym density} of $\nu$ with respect to $\mu.$\newline
\textit{Relative Fisher information:} We also define the \textit{Fisher-Donsker-Varadhan information} of $\nu$ with respect to $\mu$ by: 
\begin{equation}
\mathbf{I}[\nu|\mu]=\int\bigg|\bigg|\nabla\sqrt{\frac{d\nu}{d\mu}}\bigg|\bigg|^{2}d\mu=\frac{1}{4}\int\bigg|\bigg|\nabla\log\frac{d\nu}{d\mu}\bigg|\bigg|^{2}d\nu=\frac{1}{4}\int\bigg|\bigg|\nabla\frac{\delta\mathbf{H}[\cdot|\mu]}{\delta m}(\nu,y)\bigg|\bigg|^2\nu(dy)
\end{equation}
if $\nu\ll\mu$ and $\sqrt{\frac{d\nu}{d\mu}}\in\mathbf{H}^{1}_{\mu}$, and $\mathbf{I}[\nu|\mu]=+\infty$ otherwise. $\mathbf{H}^{1}_{\mu}$ is the domain of the Dirichlet form 
\begin{equation}
\mathcal{E}_{\mu}: g\longmapsto\int||\nabla g||^ {2}d\mu.
\end{equation}
\textbf{UPI.} We say that $\mu(dx):=\frac{1}{Z}e^{-H(x)}dx$ (Gibbs probability measure of hamiltonian $H:\R^{nD}\to\R$) satisfies a uniform Poincar\'e inequality if
    \begin{equation}
    \exists\lambda>0\quad\forall n\geq 2\quad\forall\varphi\in\mathcal{C}^{\infty}_{c}(\R^{nD}),\quad\lambda\mathbb{V} _{\mu}[\varphi]\leq\mathbb{E}_{\mu}[||\nabla\varphi||^2].
    \end{equation}
And we call Poincar\'e constant the best constant $\lambda_{1}(\mu)$ for which we have such an inequality.\newline   
\textbf{ULSI.} We say that $\mu$ satisfies a uniform logarithmic Sobolev inequality if
    \begin{equation}
    \exists\rho>0\quad\forall n\geq 2\quad\forall\varphi\in\mathcal{C}^{\infty}_{c}(\R^{nD}),\quad\rho\mathbf{Ent} _{\mu}[\varphi^2]\leq\mathbb{E}_{\mu}[||\nabla\varphi||^2].
    \end{equation}
And the best constant $\rho_{LS}(\mu)$ for which such an inequality holds is called the logarithmic Sobolev constant.
\begin{remark} We recall that
    \begin{equation}
        \textbf{ULSI.}\Longrightarrow\textbf{UPI.}
    \end{equation}
The  Poincar\'e and log-Sobolev inequalities for $\mu$ are equivalent to exponential decreases of the semigroup $(P_{t})_{t\geq 0}$ respectively in variance and in entropy, i.e.
\begin{itemize}
    \item \textbf{Poincar\'e}
    \begin{equation}
    \forall\quad f\in L^{2}(\mu)\quad t\geq 0,\quad ||P_{t}f-\langle\mu,f\rangle||_{L^{2} (\mu)}\leq e^{-\lambda_{1}(\mu)t} ||f-\langle\mu,f\rangle||_{L^{2}(\mu)}.
    \end{equation}
    \item \textbf{Log-Sobolev}
    \begin{equation}
    \forall\quad f\in L^{1}(\mu)\log L^{1}(\mu)\quad t\geq 0,\quad\mathbf{Ent}_{\mu}[P_{t }f]\leq e^{-\rho_{LS}(\mu)t}\mathbf{Ent}_{\mu}[f].
    \end{equation}
    Here, the notation $L^{1}(\mu)\log L^{1}(\mu)$ denotes the entropy definition domain under $\mu.$
\end{itemize}

\end{remark}
\noindent We say that $\mu$ satisfies a $T_{p}-$transport (Talagrand) inequality if there exists $\alpha>0$ such that $\mathcal{W}_{p}(\cdot,\mu)\leq\sqrt{\alpha\mathbf {H}[\cdot|\mu]}$.
\begin{remark}
Moreover, as with the  Poincar\'e and log-Sobolev inequalities, the second implies the first. The class of probabilities verifying $T_{1}$ is identical to that having an exponential moment of finite order $2$. The $T_{2}$ inequality is significantly more structured than the $T_{1}$ inequality since it involves a spectral gap inequality.
\end{remark}

\section{Mean-Field Systems and Assumptions}
Throughout the paper, we consider a \textit{confinement potential of a particle} $V:\R^{d}\longrightarrow\R\in\mathcal{C}^{2}(\R^{d})$ and $N$ \textit{interaction potentials} such that
\begin{equation}
\forall k\in\{2,\ldots,N\},\quad W^{(k)}:(\R^{d})^{k}\longrightarrow\R\in\mathcal{C}^{ 2}((\R^{d})^{k}).\end{equation} 
 We recall that $\forall\sigma\in\mathfrak{G}_{k}$ and $\forall x=(x_{1},\ldots,x_{k}),$
\begin{align}
W^{(k)}(\sigma\cdot x)=W^{(k)}(x),\quad
 \alpha(dx):=\frac{1}{C}e^{-V(x)}dx,\quad
U_{n}(W^{(k)}):=\frac{1}{|I^{k}_{n}|}\underset{(i_{1},\ldots,i_{ k})\in I^{k}_{n}}{\sum}W^{(k)}(x_{i_{1}},\ldots,x_{i_{k}}),
\end{align}
 where $I^{k}_{n}:=\{(i_{1},\ldots,i_{k})\in\mathbb{N}^{k}|i_{p}\neq i_{q},\quad 1\leq i_{p}\leq n\}$ is the set of possible arrangements of $k$ integers of the set of $n$ first nonzero integers, which gives $|I^{k}_{n}|=A^{k}_{n}:=\frac{n!}{(n-k)!}$.
We define $W^{(k),-}:=\max(-W^{(k)},0)$ and $W^{(k),+}:=\max(W^{ (k)},0)$ the negative and positive parts of $W^{(k)}$. $\forall\mu$ such that $W^{(k),-}\in L^{1}(\mu^{\bigotimes k})$,
 \begin{equation}
 \mathbf{W}^{(k)}[\mu]:=\mathbb{E}_{\mu^{\bigotimes k}}[W^{(k) }]=\mathbb{E}_{\mu^{\bigotimes k}}[W^{(k),+}]-\mathbb{E}_{\mu^{\bigotimes k}}[W^{(k),-}].
\end{equation}

\subsection{Our Systems}
\textbf{First order case.}
We consider the \textit{microscopic mean-field many-body interaction energy} given by
\begin{equation}
H_{n}(x_{1},\ldots,x_{n}):=\sum_{j=1}^{n}V(x_{j})+n\sum_{k=2}^{N}U_{n }(W^{(k)}).
\end{equation}
The (non-kinetic) McKean-Vlasov process is defined as the mean field limit (under adequate assumptions given below) of the sequence $(X^{n}) _{n\geq N}$ of Langevin-Kolmogorov process of Hamiltonian $H_{n}$, i.e.: ($N$ fixed)
\begin{equation}\label{CM1}
\forall n\geq N,\quad dX^{n}_{t}=\sqrt{2}dB_{t}-\nabla H_{n}(X^{n}_{t})dt.\end{equation}
Let 
\begin{equation}
\mathcal{L}_{n}:=\Delta-\nabla H_{n}\cdot\nabla
\end{equation} 
be the \textit{infinitesimal generator} and $(P^{n}_{t})_{t\geq0}$ the associated \textit{semigroup} of unique invariant measure (under \ref{Hypo1} below), the \textit{Gibbs measure} 
\begin{equation}
\mu_{n}(dx):=\frac{1}{Z_{n}}e^{-H_{n}(x)}dx\quad\textnormal{with}\quad Z_{n}:=\int_{(\R^{d})^{n}}e^{-H_{n}(x)}dx<+\infty
\end{equation} 
is the normalization constant (called \textit{partition function}).
Note that 
\begin{equation}\label{IM1}
\mu_{n}(dx)=\frac{C^{n}}{Z_{n}}e^{-n\sum_{k=2}^{N}U_{n}(W^ {(k)})}\alpha^{\bigotimes n}(dx).
\end{equation}
Without interaction (i.e. $\forall k$, $W^{(k)}\equiv 0$ or constant), $\mu_{n}=\alpha^{\bigotimes n}$ (i.e. the particles are independent).
We denote 
\begin{equation}
L_{n}(x;\cdot):=\frac{1}{n}\sum_{i=1}^{n}\delta_{x_{i}}(\cdot)
\end{equation} 
the map empirical measurement (which may be deterministic or random depending on the nature of the configurations). We know that under general conditions, by propagation of chaos (\cite{Ref6}), $L_{n}(X^{n};\cdot)$ converges weakly towards the solution of the nonlinear partial differential equation of McKean-Vlasov associated with the system of particles. We define 
\begin{equation}
\mu^{*}_{n}(dx):=e^{-n\sum_{k=2}^{N}U_{n}(W^{(k)})}\alpha ^{\bigotimes n}(dx)=\frac{Z_{n}}{C^{n}}\mu_{n}(dx).
\end{equation}
The \textit{macroscopic mean-field energy} is given by
\begin{equation}
\mathbf{E}_{W}[\mu]:=\begin{cases}\mathbf{H}[\mu|\alpha]+\sum_{k=2}^{N}\mathbf{W}^{(k)}[\mu]\quad &\textnormal{if }\mathbf{H}[\mu|\alpha]<+\infty\textnormal{ and } W^{(k),-}\in L^{1}(\mu^{\bigotimes k}) ,\\ +\infty &\textnormal{otherwise.}\end{cases}
\end{equation} 
Let
\begin{equation}
\textbf{dom}(\mathbf{H}_{W}):=\bigg\{\mu,\quad\mathbf{H}[\mu|\alpha]<+\infty\quad\textnormal{and}\quad\forall k,\quad W^{(k),-}\in L^{1}(\mu^{\bigotimes k})\bigg\}.
\end{equation}
\begin{remark}
$\mathbf{H}_{W}:=\mathbf{E}_{W}-\inf\mathbf{E}_{W}$ is called the \textit{mean field entropy}. We can prove that $\mathbf{H}_{W}$ is \textit{inf-compact} (\cref{LDP}) and that there is at least one minimizer usually called \textit{equilibrium point}. From the point of view of statistical physics, $\mathbf{H}_{W}$ is an entropy or \textit{free energy} associated to the \textit{nonlinear McKean-Vlasov equation} given by \cref{CM1}. The uniqueness of the minimizer means that there is no phase transition for the mean-field. Works on the uniqueness in the case of pair interaction: \cite{Ref2}, \cite{Ref8} and \cite{Ref9}. These authors (\cite{Ref8},\cite{Ref9}) showed that $\mathbf{H}_{W}$ is \textit{strictly displacement convex} (i.e. along the $\mathcal{W}_{2}$-geodesic) under various sufficient conditions on the convexity of the confinement potential $V$ and the pair interaction potential $W^{(2)}$. In case of a many-body interaction, under assumptions in \ref{Hypo1}, we prove in \cref{Pfixunique} the uniqueness: then we denote $\mu_{\infty}$ this minimizer. 
\end{remark}

\noindent Analogously, we define the \textit{mean-field Fisher information} by:
\begin{equation}
\mathbf{I}_{W}[\mu]:=  \frac{1}{4}\int\bigg|\bigg|\nabla\frac{\delta\mathbf{E}_{W}}{\delta m}(\mu,y)\bigg|\bigg|^2\mu(dy).  
\end{equation}
\begin{remark}
Without interaction ($\forall k$, $W^{(k)}\equiv\textit{constant}$), we find the Lyapunov functionals associated with the standard symmetric Langevin-Kolmogorov process whose Hamiltonian is given by the confinement potential $V$. More precisely, in this case:
\begin{equation}
\mathbf{E}_{W}=\mathbf{H}[\cdot|\alpha]+\sum_{k=2}^{N}\textit{constant},\quad\mathbf{H}_{W}=\mathbf{H}[\cdot|\alpha]\quad\textnormal{and}\quad\mathbf{I}_{W}=\mathbf{I}[\cdot|\alpha].
\end{equation}
\end{remark}
\noindent\textbf{Kinetic case.}
 Set 
 \begin{equation}
 z:=(x_{1},\ldots,x_{n},v_{1},\ldots,v_{n})\in\R^{2nd},\quad H^{Z}_{n}(z)=\frac{1}{2}\sum_{j=1}^{n}|v_{j}|^{2}+2V(x_{j})+n\sum_{k=2}^{N}U_{n}(W^{(k)})
 \end{equation} 
 and $Z^{n}:=(X^{n,1},\ldots,X^{n,n},V^{n,1},\ldots,V^{n,n})\in(\R^{d}\times\R^{d})^{n}$ such that
\begin{equation}\label{CM2}
\begin{cases}
dX^{n,i}_{t}=\nabla_{v_{i}}H^{Z}_{n}(Z^{n}_{t})dt\\ dV^{n,i}_{t}=-\bigg(\nabla_{x_{i}}H^{Z}_{n}(Z^{n}_{t})+\nabla_{v_{i}}H^{Z}_{n}(Z^{n}_{t})\bigg)dt+\sqrt{2}dB^{i}_{t}.
\end{cases}
\end{equation}
We are going to study the long-time behavior of the mean-field limit of the Langevin process $(Z^{n}_{t})_{t\geq 0}$ of Hamiltonian $H^{Z}_{n}(x,v):=S_{1,n}(x)+S_{2,n}(v)$ with $S_{1,n}$ is none other than the Hamiltonian $H_{n}$ of the McKean-Vlasov case and $S_{2,n}$ the velocity part ($S_{2,n}:=H^{Z}_{n}-S_{1,n}$). Invariant measure of the Langevin process is given by
\begin{equation}
\mu^{n}_{Z}(dxdv)=\frac{1}{\widetilde{C}}e^{-H^{Z}_{n}(z) }dxdv=\frac{1}{C_{1,n}}e^{-S_{1,n}(x)}dx\frac{1}{C_{2,n}}e^{-S_{ 2,n}(v)}dv=\mu_{1,n}\bigotimes\mu_{2,n}(dxdv).
\end{equation}
And the parabolic PDE in the sense of the distributions associated with this Kolmogorov-Fokker-Planck SDE is:
\begin{equation}\label{equavfpcv1}
\partial_{t}\mu=\Delta_v\mu+\nabla S_{2,n}\cdot\nabla_{v}\mu-\nabla S_{1,n}\cdot\nabla_{v}\mu+\nabla S_ {2,n}\cdot\nabla_{x}\mu=\Delta_v\mu+v\cdot\nabla_{v}\mu-\nabla S_{1,n}\cdot\nabla_{v}\mu+v \cdot\nabla_{x}\mu=\mathcal{L}^{\dag}_{Z,n}\mu
\end{equation}
with 
\begin{equation}\label{generatorcm2}
\mathcal{L}_{Z,n}:=\Delta_v-v\cdot\nabla_{v}+\nabla S_{1,n}\cdot\nabla_{v}-v\cdot\nabla_{x}
\end{equation} 
the generator of the strongly continuous semigroup $(P^{Z,(n)}_{t})_{t\geq0}$ (if the \textit{hessian} $\nabla^{2}S_{1,n }$ is bounded, it is a Markovian semigroup defined by the Kolmogorov-Fokker-Planck SDE) and we note $\mathcal{L}^{\dag}_{Z,n}$ adjoint in the sense of distributions. In other words, for any test function $\varphi\in\mathcal{C}^{\infty}_{c}((\R^{d}\times\R^{d})^{n})$, the function $(t,z)\longmapsto P^{Z,(n)}_{t}\varphi(z)$ is the unique solution of the Cauchy problem:
\begin{equation}
\begin{cases}\frac{\partial h}{\partial t}=\mathcal{L}_{Z,n}h,\\ h(0,\cdot)=\varphi.\end{cases}\Longleftrightarrow\begin{cases}\frac {\partial\mu_{t}}{\partial t}=\mathcal{L}^{\dag}_{Z,n}\mu_{t},\\ \mu_{0}=\delta_{z}.\end{cases}
\end{equation}
Vlasov Fokker Planck free energy and associated mean field entropy are given by
\begin{align}
\mathcal{E}[\mu]&:=\mathbf{H}[\mu|dxdv]+\frac{1}{2}\int_{\R^{d}\times\R^{d}}| |v||^2\mu(dxdv)+\sum_{k=2}^{N}\int_{(\R^{d}\times\R^{d}) ^{k}} W^{(k)}d\mu^{\otimes k}+\int V(x)\mu(dxdv)\\
&=\mathbf{H}[\mu|\alpha\otimes\mathcal{N}(0,\mathbf{Id}_{d})]+\sum_{k=2}^{N}\int_{(\R^{d}\times\R^{d}) ^{k}} W^{(k)}d\mu^{\otimes k}\nonumber
\end{align}
and
\begin{equation}\label{Entropy2}
\mathcal{S}:=\mathcal{E}-\mathbf{Inf}\mathcal{E}=\mathcal{E}-\mathcal{E}[\mu^{Z}_{\infty}].
\end{equation}
They are \textit{Lyapunov functionals for the Vlasov-Fokker-Planck partial differential equation} whose solutions are obtained as mean-field limits of our kinetic Fokker-Planck particle system given by \cref{CM2}.
Mean Field Fisher Information for Vlasov-Fokker-Planck is given by
$\bigg(A:=\begin{pmatrix}0\\ \mathbf{Id}_{d}\end{pmatrix}\in\mathcal{M}_{2d,d}(\R)\bigg)$
\begin{equation}
\mathcal{I}[ \mu]:=\int\bigg\langle\nabla_{x,v}\frac{\delta}{\delta m}\mathcal{E}(\mu,x,v),AA^{*}\nabla_{x, v}\frac{\delta}{\delta m}\mathcal{E}(\mu,x,v)\bigg\rangle\mu(dxdv)=\int\bigg|\bigg|\nabla_{x,v}\frac{\delta}{\delta m}\mathcal{E}(\mu,x,v)\bigg|\bigg|^2_{AA^{*}}\mu(dxdv).
\end{equation}
The functional obtained by replacing $A$ by $Z:=\begin{pmatrix}z_{1}\mathbf{Id}_{d}\\ z_{2}\mathbf{Id}_{d}\end{pmatrix}$ $\in\mathcal{M}_{2d,d}(\R)$, we will talk about auxiliary Fisher information.
We have
\begin{equation}
\frac{d}{dt}\mathcal{E}[\mu^{\textbf{VFP}}_{t}]=\frac{d}{dt}\mathcal{S}[\mu^{\textbf {VFP}}_{t}]=-\mathcal{I}[\mu^{\textbf{VFP}}_{t}]\leq0.
\end{equation}

\subsection{Our Assumptions}
\begin{sk}\hypertarget{sk}{}
\item\label{Hypo1} 
We put the following hypotheses on the potentials which will ensure properties of existence, uniqueness and contraction:
\begin{quotation}
\noindent $\mathbf{(H1)}$\textbf{(Hessian)}\label{Hypo1:H1} The hessian of the confinement potential is bounded from below and the hessians of the interaction potentials are bounded.
\begin{remark}
This is a regularity condition. It also provides good properties on the confinement potential and the interaction potentials: Since the Hessian $\nabla^{2}V$ of $V$ is bounded from below, and $V$ satisfies a Lyapunov condition $\mathbf{(H2)}$, by Cattiaux-Guillin-Wu, $\mathbf{\alpha}$ satisfies a \textit{logarithmic Sobolev inequality}.
\end{remark}

\noindent $\mathbf{(H2)}$\textbf{(Lyapunov)}\label{Hypo1:H2} There are two positive constants $c_{1}$ and $c_{2}$ such that 
\begin{equation}
\forall x\in\R^{d},\quad x\cdot\nabla V(x)\geq c_{1}||x||^{2}-c_{2}.
\end{equation}
This hypothesis is a Lyapunov condition.\newline

\noindent $\mathbf{(H3)}$ $\mathbf{(H^{(1)}_{VW})}$\label{Hypo1:H3} 
\begin{equation}
\mathbf{H}[\mu|\alpha]+\mathbb{E}_{\mu^{\bigotimes k}}[W^{( k),+}]<+\infty,\quad\forall\lambda>0,\quad \int e^{\lambda W^{(k),-}(x)-\sum_{j=1}^{k}V(x_{j})}dx<+\infty
\end{equation} 
for some measure $\mu$.
\begin{remark}
For exemple, can to take 
\begin{equation}\mu\in\textbf{dom}(\mathbf{H}_{W})\bigcap\bigcap_{k=2}^{N}\bigg\{\mu,\quad W^{(k),+}\in L^{1}(\mu^{\bigotimes k})\bigg\}.\end{equation}
\end{remark}

\noindent $\mathbf{(H4)}$ $\mathbf{(H^{(2)}_{VW})}$\label{Hypo1:H4} There exists $p\geq 2$ such that for some $x_{0}$ (hence any $x_{0}$) 
\begin{equation}
\forall\lambda>0,\quad\int_{\R^{d}}e^{\lambda||x-x_{0}||^{p}}\alpha(dx)<+\infty.
\end{equation}

\noindent $\mathbf{(H5)}$\textbf{(Logsob)}\label{Hypo1:H5} The invariant measure $\mu_{n}$ of the system satisfies a logarithmic sobolev inequality such that 
\begin{equation}\limsup_{n\to +\infty}\rho_{LS}(\mu_{n})>0.
\end{equation} 
\begin{remark}
This assumption is usually given by Zegarlinski conditions (\cite{Ref2},\cite{Zegarlinski},\cite{Z96}).
\end{remark}

\noindent $\mathbf{(H6)}$\textbf{(Lipschitz)} There exists a distance $\mathbf{d}_{Lip}$ on a subset $\mathcal{Z}$ of $\mathcal{P}(\R^{d})$ such that $(\mathcal{M }^{2}_{1}(\R^{d}),\mathcal{W}_{2})$ \textit{continuously injects} into $(\mathcal{Z},\mathbf{d}_{Lip })$ and $\Phi:\mu\in\mathcal{Z}\longmapsto\Phi(dx):=\frac{1}{Z_{\mu}}e^{-\frac{\delta F}{\delta m}(\mu,x)-V(x)}dx\in\mathcal{Z}$ satisfies
\begin{equation}\label{Lipcond}
        \exists k\in (0,1[,\quad\forall\mu,\nu\in\mathcal{Z},\quad\mathbf{d}_{Lip}(\Phi(\mu),\Phi(\nu))\leq k\mathbf{d}_{Lip}(\mu,\nu).
\end{equation}
In others terms, $\Phi$ is $k$-Lipschitz (\textit{contraction}) for $\mathbf{d}_{Lip}.$
\begin{remark}
This assumption is verified in the case
\begin{equation}
        \sup_{\mu,\nu\in\mathcal{M}^{1}_{1}(\R^{d}), \mu\neq\nu}\frac{1}{\mathcal{W}_{1}(\mu,\nu)}\int||x-y||\Phi(\mu)(dx)\Phi(\nu)(dy)<1,
\end{equation}
and in this case, we have
\begin{equation}
        \mathcal{Z}=\mathcal{M}^{1}_{1}(\R^{d})\quad\mathbf{d}_{Lip}=\mathcal{W}_{1}\quad k=\sup_{\mu,\nu\in\mathcal{M}^{1}_{1}(\R^{d}), \mu\neq\nu}\frac{1}{\mathcal{W}_{1}(\mu,\nu)}\int||x-y||\Phi(\mu)(dx)\Phi(\nu)(dy).
\end{equation}
    Note that
\begin{align}\label{condeberl}
        \frac{1}{\mathcal{W}_{1}(\mu,\nu)}\int||x-y||\Phi(\mu)(dx)\Phi(\nu)(dy)=\frac{1}{\mathcal{W}_{1}(\mu,\nu)Z_{\mu}Z_{\nu}}\int||x-y||e^{-\frac{\delta F}{\delta m}(\mu,x)-\frac{\delta F}{\delta m}(\nu,y)-V(x)-V(y)}dxdy.
\end{align}
Here, we assume a contraction assumption to ensure uniqueness.
In \cite{Ref2}, the authors used \textit{Eberle conditions} applied to \cref{condeberl} to establish this Lipschitz hypothesis \cref{Lipcond}: \textit{Lipschitzian spectral gap condition for one particle.}
Some authors (see e.g. \cite{ottogradients},\cite{Ref9}) rather use the \textit{displacement-convexity} by assuming that the functional $G$ in $V^{\sigma,G}$ is displacement-convex. And as the relative entropy is strictly displacement-convex, $V^{\sigma,G}$ is also strictly displacement-convex, which implies the existence of an entropy minimizer ensuring its uniqueness.
\end{remark}
\end{quotation}
\end{sk}

\begin{condH}\hypertarget{condH}{}
\item\label{Hypo2} In this case, we assume the following assumptions:
\begin{quotation}
\noindent\ref{Hypo1}\newline
\noindent\textbf{VFP1.} Lipschitz interactions:
\begin{equation}
\forall k\in\{2,3,\ldots,N\}\quad\exists K>0,\quad||\nabla W^{(k)}| |\leq K.
\end{equation}
\noindent\textbf{VFP2.} Lyapunov condition on confinement: 
\begin{equation}
||\nabla^2V||_{\textbf{op}}\leq K_{1}||\nabla V||+K_{2 }.
\end{equation}
\begin{remark}
Either of these conditions ensures that the kinetic Fokker-Planck semigroup converges exponentially ( as a family of operators of $\mathcal{H}^{1}(\mu^{n}_ {Z})$ indexed by time ) towards $\mu^{n}_{Z}$ and uniformly in the number of particles (see \cite{Ref3} or \cite{villanihypo}).
\end{remark}
\end{quotation}
\end{condH}
\section{Main Theorems}\label{main}
\subsection{First-order case}
Under \ref{Hypo1}, we establish (see \cref{proofsmain} for the proof) the following two main results (thus generalizing those of \cite{Ref2}). Let $(\mu_{t})_{t\geq 0}$ (given by the arrow $\mathbf{(1)}$ in \cref{fig1}) be the flow of solution distributions of the McKean-Vlasov equation associated with the particle system defined by the $U-$statistic and the confinement potential. Then for any initial condition admitting a moment of order $2$, the mean field entropy $\mathbf{H}_{W}$ decreases exponentially along the flow, i.e.:
\begin{theorem}[Exponential decreasing of mean-field entropy]\label{theo1}
Assume \ref{Hypo1} and let $\mu_{0}\in\mathcal{M}^{2}_{1}(\R^{d})$ be an initial condition. Then
\begin{equation}
        \forall t\geq 0,\quad\mathbf{H}_{W}[\mu_{t}]\leq\mathbf{H}_{W}[\mu_{0}]e^{-\rho_{ LS}\frac{t}{2}}.
\end{equation}
\end{theorem}
\noindent From the exponential decrease of the mean field entropy along the flow, we deduce the following exponential convergence in Wassertein metric:
\begin{theorem}[Exponential convergence in Wasserstein metric from flow to equilibrium]\label{theo2}
Assume \ref{Hypo1} and give us an initial condition $\mu_{0}\in\mathcal{M}^{2}_{1}(\R^{d})$. Then
\begin{equation}
        \forall t\geq0,\quad\mathcal{W}^{2}_{2}(\mu_{t},\mu_{\infty})\leq\frac{2}{\rho_{LS}}\mathbf{H}_{W}[\mu_{0}]e^{-\rho_{LS}\frac{t}{2}}.
\end{equation}
\end{theorem} 
\subsection{Kinetic case}
\noindent For kinetic type models, the extension of the above results relies on applications of hypocoercivity arguments (see e.g. \cite{Ref3} or \cite{villanihypo} for background). In this setting, we first obtain an exponential decrease in $|||\cdot|||_{H^{1}\to H^{1}}$ norm (defined in \cref{defnot}).
\begin{theorem}[Uniform exponential convergence to equilibrium in the weighted Sobolev space]\label{theo3}
Assume \ref{Hypo2} and give us an initial condition $\mu\in\mathcal{M}^{2}_{1}(\R^{d}\times\R^{d})$. Then
    \begin{equation}
        \exists\alpha>0\quad\exists\beta>0\quad\forall n\geq 2,\quad\bigg|\bigg|\bigg|P^{Z,(n)}_{t}-\mu^{n}_{Z}\bigg|\bigg|\bigg|_ {H^1\to H^{1}}\leq\alpha e^{-\beta t}.
    \end{equation}
\end{theorem}
\begin{remark} 
We still have \cref{theo3} if we replace the uniform logarithmic Sobolev inequality given in \ref{Hypo1:H5} by a uniform Poincar\'e inequality. We keep the logarithmic Sobolev inequality to have the following \cref{theo4}. Note that the constants $\alpha>0$ and $\beta>0$ can be made explicit uniform. The originality of the proof relies on functional inequalities and hypocoercivity with Lyapunov type conditions, usually not suitable to provide adimensional results.
\end{remark}

\begin{theorem}[Exponential decay in Wasserstein metric]\label{theo4}
Under \ref{Hypo2}, there are constants $C>0$, $\xi>0$ and $\kappa>0$ such that $\forall\mu\in\mathcal{P}_{2}(\R^{d}\times\R^{d})$, $\forall n\geq 2$ and $\forall t>0,$
    \begin{align}
        &\mathbf{H}[\mu^{n}_{Z}(t)|\mu^{n}_{Z}]\leq C\mathbf{H}[\mu^{n}_{Z}(0)|\mu^{n}_{Z}]e^{-\xi t},\label{vfpcv1}\\
        &\mathcal{W}^2_{2}(\mu^{\mathbf{VFP}}_{t},\mu^{Z}_{\infty})\leq\kappa C\mathcal{S}[\mu] e^{-\xi t},\label{vfpcv2}
    \end{align}
   where $\mu$ is the initial condition and $\mathcal{S}$ (defined in \cref{Entropy2}) is the mean-field entropy associated with our second order system given by \cref{CM2}.
\end{theorem}

\section{Sketch of proofs and preliminaries}\label{sketchandpreli}
\subsection{Sketch of proofs}
\begin{figure}[htb!]
    \centering
    \includegraphics[height=106.2mm]{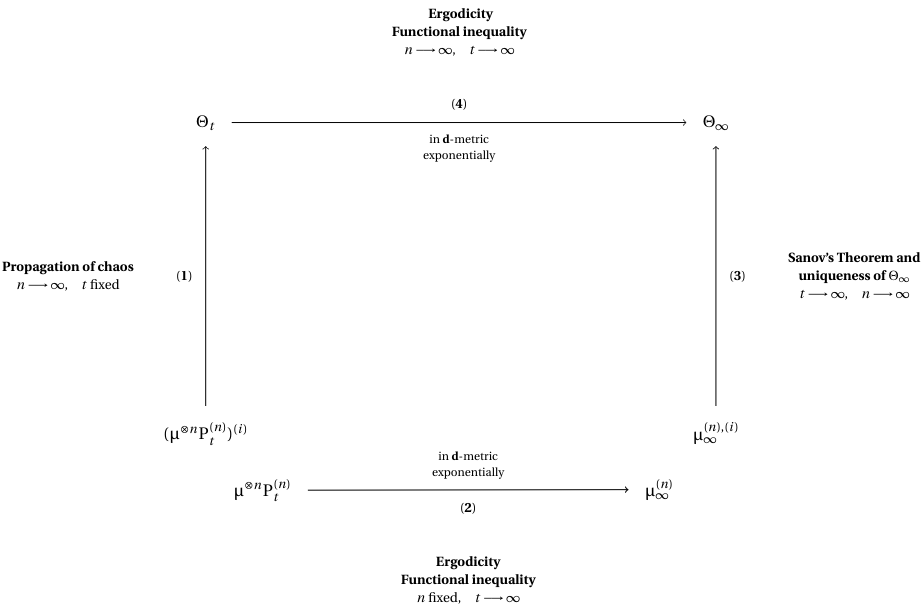}
    \caption{Diagram of convergences}
    \label{fig1}
\end{figure}
\noindent\textbf{First order case.}
The diagram given in \cref{fig1} summarizes the strategy of proof: we show $\mathbf{(4)}$ from $\mathbf{(1)}$ , $\mathbf{(2)}$  and $\mathbf{(3)}$.
And in this diagram, the quantities involved are:
\begin{itemize}
    \item $\mu^{\otimes n} P^{(n)}_{t}=\mu_{n}(t)$ the law at time $t$ of the particle system induced by the confinement potential and the $U-$statistics;
    \item $\mu^{(i)}_{n}(t)$ the $i-$th marginal of $\mu_{n}(t)$;
    \item $\mu^{(n)}_{\infty}=\mu_{n}$ the invariant measure of the particle system;
    \item $\mu^{(i)}_{n}$ the $i-$th marginal of $\mu_{n}$;
    \item $\Theta_{t}=\mu_{t}$ the law at time $t$ of the McKean-Vlasov process obtained by propagation of chaos;
    \item $\Theta_{\infty}=\mu_{\infty}$ the invariant measure of the McKean-Vlasov process;
    \item $\mathbf{d}=\mathcal{W}_{2}.$
\end{itemize}

\noindent\textbf{Arrow $\mathbf{(1)}$.}
The McKean-Vlasov process classically appears as the mean-field limit of a particle system. This property is recalled and studied, among others, in \cite{propachaos1}.\newline

\noindent\textbf{Arrow $\mathbf{(2)}$.}
The process $X^{n}$ is a homogeneous diffusion process of the Langevin-Kolmogorov type which is a class of Markov processes. In the literature, the long-time behavior for this class is classically studied (see e.g. \cite[\cite{BakryEmery}]{BGL-MarkovDiffusion}). In order to ensure this property (see \cref{prel}.\cref{transport}), exponentially in time and uniformly in number of particle $n$, we rely on  $\mathbf{(H5)}$ in \ref{Hypo1} and the equivalence between Sobolev's inequality, exponential decay of entropy and Talagrand's second inequality for Gibbs measures.\newline

\noindent\textbf{Arrow $\mathbf{(3)}$.}
This arrow is ensured by $\mathbf{(H3)}$  and $\mathbf{(H4)}$ in \ref{Hypo1} which allow us to obtain large deviations principle and Sanov-type theorem (see  \cref{prel}.\cref{LDP}.\cref{Thmsanov}).\newline

\noindent\textbf{Arrow $\mathbf{(4)}$.}
To establish this last arrow, we will use the fact that the nonlinear Sobolev inequality ($\rho_{LS}\mathbf{H}_{W}\leq 2\mathbf{I}_{W}$) given in \cref{prel}.\cref{transport} is also equivalent to the exponential decrease of the mean field entropy $\mathbf{H}_{W}$ along the flow $(\mu_{t})_{t\geq 0}$ of the McKean-Vlasov distributions and to the second nonlinear Talagrand inequality ($\rho_{LS}\mathcal{W}^{2}_{2}(\cdot,\mu_{\infty})\leq 2\mathbf{H}_{W}$). Note that Talagrand inequalities allow to recover usual Wasserstein convergence (and then convergence in law) from entropic convergence. Note that concentration inequalities could also stem from Talagrand inequalities, although the stronger Logarithmic Sobolev inequality is more often used in this context.
\begin{remark}
The exponential convergence in entropy (given in \cref{theo1}) should be equivalent to the mean field log-Sobolev inequality $\rho_{LS}\mathbf{H}_{W}\leq 2\mathbf{I}_{W}$ (in \cref{transport}), basing on (gradient flow and Gronwall lemma) 
\begin{equation}
-\frac{d}{dt}\mathbf{H}_{W}[\mu_{t}]=\mathbf{I}_{W}[\mu_{t}]\Longrightarrow\frac{d}{dt}\mathbf{H}_{W}[\mu_{t}]\leq -\frac{1}{2}\rho_{LS}\mathbf{H}_{W}[\mu_{t}]\Longrightarrow\mathbf{H}_{W}[\mu_{t}]\leq\mathbf{H}_{W}[\mu_{0}]e^{-\rho_{LS}\frac{t}{2}}
\end{equation} 
noted by Carrillo-McCann-Villani in their convex framework. The proof of $-\frac{d}{dt}\mathbf{H}_{W}[\mu_{t}]=\mathbf{I}_{W}[\mu_{t}]$  demands the \textit{regularity of} $t\longmapsto\mu_{t}$ (\cref{fig1}) which requires the \textit{PDE theory of the McKean-Vlasov equation}. That is why we prefer to give a \textit{rigorously probabilistic proof based directly on the log-Sobolev inequality of} $\mu_{n}$ (\cref{fig1}) in \ref{Hypo1}.$\mathbf{H5}$. As for \cref{theo2} on exponential decay in Wasserstein metric, it follows from the previous one (\cref{theo1}) via Talagrand's T2-inequality.
\end{remark}

\noindent\textbf{Second order case.}
The proof in this case, can also be described by the diagram given in \cref{fig1} but with the following notations:
\begin{itemize}
    \item $\mu^{\otimes n} P^{(n)}_{t}=\mu^{n}_{Z}(t)$ the law at time $t$ of the kinetic particle system induced by the confinement potential and the $U-$statistics;
    \item $(\mu^{\otimes n} P^{(n)})^{(i)}=:\mu^{n,(i)}_{Z}(t)$ the $i-$th marginal of $\mu^{n}_{Z}(t)$;
    \item $\mu^{(n)}_{\infty}=\mu^{n}_{Z}$ the invariant measure of the particle system;
    \item $\mu^{n,(i)}_{Z}$ the $i-$th marginal of $\mu^{n}_{Z}$;
    \item $\Theta_{t}=\mu^{\textbf{VFP}}_{t}$ the law at time $t$ of the Vlasov-Fokker-Planck process obtained by propagation of chaos;
    \item $\Theta_{\infty}=\mu^{Z}_{\infty}$ the invariant measure of the Vlasov-Fokker-Planck process;
    \item $\mathbf{d}=|||\cdot-\cdot|||_{H^{1}\to H^{1}}$ or $\mathbf{d}=\mathcal{W}_{2}$.
\end{itemize}

\noindent\textbf{Arrow $\mathbf{(1)}$.}
We first recall the generator $\mathcal{L}_{Z,n}$ defined (in Hormander form) by \cref{hormanderkfp} is a non-symmetric hypoelliptic  operator (see \cref{remarkhormanderform}). The related $n$-particle system given by \cref{CM2} converges to the Vlasov-Fokker-Planck equation (mean-field limit of \cref{CM2}) when  $n\to +\infty$ (see e.g. \cite{propachaos1}).\newline

\noindent\textbf{Arrow $\mathbf{(2)}$.}
The process $Z^{n}$ is a homogeneous diffusion process of the Langevin type usually called kinetic Fokker-Planck process. The study of the long-time behavior of the particle system requires the help of hypocoercivity tools (see  e.g. \cite{Ref3} and \cite{villanihypo}).  We recall that 
\begin{equation}
\forall x\in\R^{nd},\quad S_{1,n}(x):=\sum V(x_{i})+n\sum_{k=2}^{N }U_{n}(W^{(k)}).
\end{equation}
In particular, \ref{Hypo2} ensures the following Poincar\'e and log-Sobolev inequalities
\begin{itemize}
    \item\textbf{UPI.} We say that $\mu_{1,n}$ satisfies a uniform Poincar\'e inequality if
    \begin{equation}
    \exists\lambda>0\quad\forall n\geq 2\quad\forall\varphi\in\mathcal{C}^{\infty}_{c}(\R^{nd}),\quad\lambda\mathbb{V} _{\mu_{1,n}}[\varphi]\leq\mathbb{E}_{\mu_{1,n}}[||\nabla_{x}\varphi||^2].
    \end{equation}
    \item\textbf{ULSI.} We say that $\mu_{1,n}$ satisfies a uniform logarithmic Sobolev inequality if
    \begin{equation}
    \exists\rho>0\quad\forall n\geq 2\quad\forall\varphi\in\mathcal{C}^{\infty}_{c}(\R^{nd}),\quad\rho\mathbf{Ent} _{\mu_{1,n}}[\varphi^2]\leq\mathbb{E}_{\mu_{1,n}}[||\nabla_{x}\varphi||^2].
    \end{equation}
\end{itemize}
 Under \textbf{(UPI)}, we are able to obtain as an application of Villani's theorem the following exponential rate to equilibrium 
\begin{align}
\forall n\geq 2,\quad\bigg|\bigg|\bigg|P^{Z,(n)}_{t}-\mu^{n}_{Z}\bigg|\bigg|\bigg|_ {H^1\to H^{1}}\leq\alpha e^{-\beta t}
\end{align}
with constants $\alpha>0$ and $\beta>0$ make explicit uniform. The idea in Villani's proof of \cite[Theorem.3]{Ref3} is as follows: if one could find a Hilbert space such that the operator $\mathcal{L}_{Z,n}$ is coercive with respect to its norm, then one has exponential convergence for the semigroup $(P^{Z,(n)}_{t})_{t\geq 0}$ under such a norm. If, in addition, this norm is equivalent to some usual norm (such as $\mathcal{H}^1(\mu^{n}_{Z})-$norm), then one obtains exponential convergence under the usual norm as well. In his statement of \cite[Theorem.35]{villanihypo}, the \textit{boundedness condition} is verified by $||\nabla^2S_{1,n}||_{\textbf{op}}\leq C(1+||\nabla S_{1,n}||)$ with a constant $M$ depending unfortunately on the dimension. The $L^2$ and $H^1$ norms are not suitable to obtain a result on the non-linear system (such as \cref{vfpcv1} and \cref{vfpcv2}). On the other hand, thanks to \textbf{(ULSI)} playing a fundamental role in the exponential return in Wasserstein metric (see e.g. \cite[Theorem.7]{Monmarche1} or \cite[Theorem.10]{Monmarche2}), we are able to prove \cref{vfpcv1} which in turn will allow us to deduce \cref{vfpcv2}.\newline

\noindent\textbf{Arrow $\mathbf{(3)}$.}
The results of large deviations on the U$-$statistics in the non-kinetic case in \cref{prel} and the fact that $\mu^{n}_{Z}=\mu_{1,n}\bigotimes\mu_{2,n}$ allow to deduce that the random empire measurements of the kinetic particle system satisfy the principle of large deviations under $\mu^{n}_{Z}$ of good rate function defined by 
\begin{equation}
\forall (\mu_{x },\mu_{v})\in\mathcal{P}_{x}(\R^{d})\times\mathcal{P}_{v}(\R^{d}),\quad\mathbf{I}(\mu_{x},\mu_{v}):=\mathbf{H}_{W}[\mu_{x}]+\mathbf{H}[\mu_{v}|\mathcal {N}(0,\mathbf{Id}_{d})].
\end{equation}
Thus there exists by inf-compactness a Maxwellian to the nonlinear Vlasov-Fokker-Planck equation and this equilibrium (invariant measure of the nonlinear Vlasov-Fokker-Planck process) is unique. See  \cref{prel}.\cref{LDP}.\cref{Thmsanov}.\cref{Pfixunique}.\cref{ldpcontraction}.\newline

\noindent\textbf{Arrow $\mathbf{(4)}$.} This part is obtained by the first-order case by exploiting the uniform logarithmic Sobolev inequality and the Hormander form given  by \cref{hormanderkfp} (see \cref{prelvfp}).
\begin{remark}
 By applying hypocoercivity tools to the system with $n$ particles given by \cref{CM2}, we obtain a (uniform in $n$) convergence rate to equilibrium which in turn extends to the limiting non linear system.
\end{remark}

\subsection{Preliminaries}\label{prel}
The results on the $U-$statistics (\ref{tcl},\ref{couplage},\ref{jensen},\ref{decouplage1}) and the inf-compactness of the entropy functional $\mathbf{H}_{W}$ (\ref{lower},\ref{LDP},\ref{Thmsanov}) are inspired by \cite{Ref1} in the case $S=\mathbb{R}^{d}$. We recall that the expectation of $W^{(k)}$ under $\mu^{\bigotimes k}$ exists if and only if
  \begin{equation}
  \mathbb{E}_{\mu^{\bigotimes k}}[W^{(k),+}]<+\infty\quad\textnormal{or}\quad\mathbb{E}_{\mu^{\bigotimes k}}[W^{(k),-}]<+\infty.
  \end{equation}

\noindent\textbf{Back to U-statistics.}\label{TCL}
First we present the law of large numbers of the $U-$statistic (see [\cite{Ref4}, Corollary 3.1.1] or [\cite{Ref1}, Lemma 3.1]). We recall that $U$-statistics are defined in \cref{U-stats1}.
\begin{proposition}[law of large numbers for $U-$statistics]\label{tcl}
Let $(X_{n})_{n\geq 1}$ be a sequence of independent and identically distributed random variables with values in a measurable space $(E,\mathcal{B}(E))$ equipped with its Borelian tribe and $\Phi:E^{k}\longrightarrow\R$ a symmetric measurable function such that
\begin{equation}
\mathbb{E}[|\Phi(X_{1},\ldots,X_{k})|]<+\infty,\quad\textnormal{then}\quad U_{n}(\Phi)\overset{ n\to+\infty}{\longrightarrow}\mathbb{E}[\Phi(X_{1},\ldots,X_{k})]\quad\textnormal{ with probability}\quad 1.
\end{equation}
\end{proposition}
\begin{proof}
See \cref{proof1} or \cite{Ref1}.
\end{proof}
\noindent In terms of integrals, this result means that for any function $\Phi\in\mathcal{M}_{sym}( E^{k},\R)$ with $E^{k}$ provided with the tensor tribe (or product) and any measure $\mu\in\mathcal{P}(E)$ such that $\Phi\in L^{1}(\mu^{\bigotimes k})$, we almost surely have 
\begin{equation}
U_{n}(\Phi)\overset{n\to+\infty}{\longrightarrow}\mathbb{E}_ {\mu^{\bigotimes k}}[\Phi]:=\int_{E^{k}}\Phi(x)\mu^{\bigotimes k}(dx).
\end{equation} 
This result can also be seen as a law of large numbers for $U-$statistics. From this result, we deduce that $\forall k\in\{2,\ldots,N\}$, if $W^{(k)}\in L^{1}(\mu^{\bigotimes k})$, then we almost surely have $U_{n}(W^{(k)})$ tends to $\mathbf{W}^{(k) }[\mu]$.
We first recall the decoupling inequality of Victor H. De La Pe$\Tilde{\textnormal{n}}$a (see [\cite{Ref5}, 1992]).
\begin{proposition}[Decoupling and Khintchine inequalities for $U-$statistics]\label{couplage}
Let $(X_{n})_{n\geq 1}$ be a sequence of random variables with values in a measurable space $(E,\mathcal{B}(E))$, independent and identically distributed. We assume that 
\begin{equation}
(X^{j}_{1},\ldots,X^{j}_{n})_{j=1,\ldots,k}
\end{equation} 
are $k$ independent copies of $(X_ {1},\ldots,X_{n})$. Then for all increasing convex functions $\Psi:[0,+\infty)\longrightarrow\R$ and measurable symmetric $\Phi:E^{k}\longrightarrow\R$ such that $\mathbb{E}[|\Phi(X_{1},\ldots,X_{k})|]<+\infty$, we have
\begin{equation}
\mathbb{E}\bigg[\Psi\bigg(\bigg|\sum_{(i_{1},\ldots,i_{k})\in I^{k}_{n}}\Phi(X_{ i_{1}},\ldots,X_{i_{k}})\bigg|\bigg)\bigg]\leq\mathbb{E}\bigg[\Psi\bigg(C_{k}\bigg|\sum_ {(i_{1},\ldots,i_{k})\in I^{k}_{n}}\Phi(X^{1}_{i_{1}},\ldots,X^{k }_{i_{k}})\bigg|\bigg)\bigg] 
\end{equation}
with
\begin{equation}
C_{2}:=8\quad\textnormal{and}\quad\forall k\geq 3,\quad  C_{k} :=2^{k}\prod_{j=2}^{k}(j^{j}-1).
\end{equation}
\end{proposition}

\begin{proposition}\label{jensen}
Let $1\leq k\leq n$, $(X^{j}_{i})_{1\leq i\leq n,1\leq j\leq k}$ be independent random variables with values in $( E,\mathcal{B}(E))$. For all $(i_{1},\ldots,i_{k})\in I^{k}_{n}$, defining $\Phi_{i_{1},\ldots,i_{k}}: E^{k}\longrightarrow\R$ a measurable function of $k$ variables, we have
\begin{align}
\log\mathbb{E}\bigg[\exp\bigg(\frac{1}{|I^{k}_{n}|}\sum_{i\in I^{k}_{n}}\Phi_{i}(X^{1}_{i_{1}},\ldots,X^{k}_{i_{k}})\bigg)\bigg]\leq 
\frac{n-k+1 }{|I^{k}_{n}|}\sum_{i\in I^{k}_{n}}\log\mathbb{E}\bigg[\exp\bigg(\frac{1} {n-k+1}\Phi_{i}(X^{1}_{i_{1}},\ldots,X^{k}_{i_{k}})\bigg)\bigg].
\end{align}
\end{proposition}
\begin{proof}
See \cref{proof2} or \cite{Ref1}.
\end{proof}

\begin{proposition}[Decoupling corollary]\label{decouplage1}
For $(X_{i})_{i\geq 1}$ a sequence of independent and identically distributed random variables according to $\alpha$, we denote $\Lambda_{n}(\cdot,W^{(k)} )$ the log-Laplace transformation associated with the $U-$statistic of order $k$, i.e. to within a factor, the logarithm of the moment generating function, namely
\begin{equation}
\forall n\geq k\geq 2,\quad\forall \lambda>0,\quad\Lambda_{n}(\lambda,W^{(k)}):=\frac{1}{n}\log \mathbb{E}\bigg[e^{n\lambda U_{n}(W^{(k)})}\bigg].
\end{equation}
If $W^{(k)}\in L^{1}(\alpha^{\bigotimes k})$, then
\begin{equation}
\Lambda_{n}(\lambda,W^{(k)})\leq\frac{1}{k}\log\mathbb{E}\bigg[\exp\bigg(kC_{k}\lambda|W ^{(k)}(X_{1},\ldots,X_{k})|\bigg)\bigg].
\end{equation}
\end{proposition}
\begin{proof}
See \cref{proof3} or \cite{Ref1}.
\end{proof}

\noindent\textbf{Large deviations: inf-compactness of mean-field entropy and existence of an equilibrium point.}\label{sanov}
We will use a \textit{large deviations result} ensuring the \textit{infcompactness of the entropy functional} to show the \textit{existence of an invariant measure} for the nonlinear process studied.

\begin{proposition}[Lower bound of large deviations for $L_{n}$ under $\mu^{*}_{n}$]\label{lower}
Under the integrability assumptions on the interaction potentials $(W^{(k)})_{2\leq k\leq N}$, we have the lower bound of large deviations for $\{\mu^{ *}_{n}(L_{n}\in\cdot)\}_{n\geq N}$, i.e.
\begin{align}
\forall\mathcal{O}\subset\mathcal{M}_{1}(\R^{d})\quad\textnormal{open},\quad l^{*}(\mathcal{O})&:= \underset{n\to+\infty}{\liminf}\frac{1}{n}\log(\mu^{*}_{n}(L_{n}\in\mathcal{O}))\\
&\geq-\inf\bigg\{\mathbf{E}_{W}[\mu]\bigg|\quad\mu\in\mathcal{O},\quad\forall 2\leq k\leq N,\quad W ^{(k)}\in L^{1}(\mu^{\bigotimes k})\bigg\}.\nonumber
\end{align}
In particular, we have
\begin{align}
\underset{n\to+\infty}{\liminf}\bigg\{\frac{1}{n}\log(Z_{n})-\log(C)\bigg\}\geq
-\inf\bigg \{\mathbf{E}_{W}[\mu]\bigg|\quad\mu\in\mathcal{M}_{1}(\R^{d}),\quad\forall 2\leq k \leq N,\quad W^{(k)}\in L^{1}(\mu^{\bigotimes k})\bigg\}.
\end{align}
\end{proposition}
\begin{proof}
See \cref{proof4} or \cite{Ref1}.
\end{proof}
\begin{proposition}[Exponential approximation of the $U-$statistic]\label{approx}
Assuming that for all $\lambda>0$, 
\begin{equation}
\mathbb{E}[\exp(\lambda|W^{(k)}|(X_{1},\ldots,X_{k}))]< +\infty,
\end{equation}
then there exists a sequence $(W^{(k)}_{m})_{m\geq 1}$ of bounded continuous functions such that
\begin{equation}
\forall\delta>0,\quad\underset{m\to+\infty}{\lim}\underset{n\to+\infty}{\limsup}\frac{1}{n}\log\mathbb{P} (|U_{n}(W^{(k)}_{m})-U_{n}(W^{(k)})|>\delta)=-\infty.
\end{equation}
\end{proposition}
\begin{proof}
See \cref{proof5} or \cite{Ref1}.
\end{proof}

\begin{theorem}[Large deviations principle for $U-$statistics]\label{LDP}
Let $(X_{i})_{i\geq 1}$ be a sequence of independent and identically distributed random variables with distribution $\alpha.$ We assume that we have exponential integrability of the interaction potentials under the tensor products of $\alpha$ by itself, i.e.
\begin{equation}
\forall k\in\{2,\ldots,N\},\quad\forall\lambda>0,\quad\bigg(\mathbb{E}\bigg[e^{\lambda|W^{(k) }(X_{1},\ldots,X_{k})|}\bigg]<+\infty\Longleftrightarrow e^{\lambda|W^{(k)}|}\in L^{1}(\alpha^{\bigotimes k})\bigg).
\end{equation}
Then 
\begin{equation}
\bigg\{\mathbb{P}\bigg((L_{n},U_{n}(W^{(2)}),\ldots,U_{n}(W^{(N)}) )\in\cdot\bigg)\bigg\}_{n\geq N}
\end{equation} satisfies a large deviations principle on the product space $\mathcal{M}_{1}(\R^{d})\times\R^{N-1}$ and good rate function given by 
\begin{equation}
\mathbf{I}_{U}(\mu,x_{2},\ldots,x_{N}):=\begin{cases}\mathbf{H}[\mu|\alpha],\quad&\textnormal{if}\quad\forall k,\quad x_{k}=\mathbf{W}^{(k)}[\mu],\\+\infty\quad &\textnormal{otherwise.}\end{cases}
\end{equation}
\end{theorem}
\begin{proof}
Let $(W^{(k)}_{m})_{m\geq 1}$ be the sequence of bounded continuous functions of the proof of \cref{approx} (see \cref{proof5}) such that for all $\lambda>0$,
\begin{equation}
\varepsilon(\lambda,m,k):=\log\int_{(\R^{d})^{k}}e^{\lambda|W^{(k)}-W^{(k) }_{m}|}d\alpha^{\bigotimes k}\overset{m\to+\infty}{\longrightarrow}0.
\end{equation}
For all $m\geq 1$, we set
\begin{equation}
f_{m}(\mu):=\bigg(\mu,\mathbf{W}^{(2)}_{m}[\mu],\ldots,\mathbf{W}^{(N)} _{m}[\mu]\bigg),\quad f(\mu):=\bigg(\mu,\mathbf{W}^{(2)}[\mu],\ldots,\mathbf{W }^{(N)}[\mu]\bigg).
\end{equation}
We consider the following metric on the product space
\begin{equation}
\mathbf{d}\bigg((\mu,x_{2},\ldots,x_{N}),(\nu,y_{2},\ldots,y_{N})\bigg):=d_{ LP}(\mu,\nu)+\sum_{k=2}^{N}|x_{k}-y_{k}|=d_{LP}(\mu,\nu)+||x-y|| _{1},
\end{equation}
and note that
\begin{equation}
\mathbf{d}(f_{m}(\mu),f(\mu))=\sum_{k=2}^{N}\bigg|\int_{(\R^{d})^{k }}(W^{(k)}_{m}-W^{(k)})d\mu^{\bigotimes k}\bigg|.
\end{equation}
The sequel of the proof is divided in three steps.\newline
 
 \noindent\textbf{Step 1: Continuity of $f_{m}$.} For this step, it suffices to show that for all $k\in\{2,\ldots,N\}$, $\mu\in\mathcal{M}_{1}(\R^{d})\longmapsto\mathbf{W}^{(k)}_{m}[\mu]$ is continuous for the convergence topology weak. Let $\mu_{n}\overset{n\to+\infty}{\longrightarrow}\mu$ in $(\mathcal{M}_{1}(\R^{d}),d_{LP})$ . By the Skorokhod representation theorem, there exists a sequence $(Y_{n})_{n}$ of random variables with values in $\R^{d}$ such that $Y_{n}\sim\mu_{ n}$ and almost surely, $Y_{n}\overset{n\to+\infty}{\longrightarrow}Y\sim\mu$. Let $(Y^{(i)}_{n}, n\geq 0, Y^{(i)})_{1\leq i\leq k}$ be independent copies of $(Y_{n}, n\geq 0, Y)$. We have for all $i$, almost surely, $Y^{(i)}_{n}\overset{n\to+\infty}{\longrightarrow}Y^{(i)}$, which implies that almost surely, $(Y^{(1)}_{n},\ldots,Y^{(k)}_{n})\overset{n\to+\infty}{\longrightarrow}(Y^{(1 )},\ldots,Y^{(k)})$. In particular, $\mu^{\bigotimes k}_{n}$ tends weakly to $\mu^{\bigotimes k}$, which proves the continuity of the above functional.\newline
 
 \noindent\textbf{Step 2: Good exponential approximation of $(L_{n},U_{n}(W^{(2)}),\ldots,U_{n}(W^{(N) }))$ by $f_{m}(L_{n})$.} By exponential approximation of the $U-$statistic, we have for all $\delta>0,$
\begin{align}
\lim_{m\to+\infty}\limsup_{n\to+\infty}\frac{1}{n}\log\mathbb{P}\bigg(\mathbf{d}\bigg((L_{n}, U_{n}(W^{(2)}),\ldots,U_{n}(W^{(N)})),
(L_{n},U_{n}(W^{(2)} _{m}),\ldots,U_{n}(W^{(N)}_{m}))\bigg)>\delta\bigg)=-\infty,\nonumber
\end{align}
i.e. $(L_{n},U_{n}(W^{(2)}_{m}),\ldots,U_{n}(W^{(N)}_{ m}))$ is a good exponential approximation of $(L_{n},U_{n}(W^{(2)}),\ldots,U_{n}(W^{(N)}))$ .\newline Moreover, $(L_{n},U_{n}(W^{(2)}),\ldots,U_{n}(W^{(N)}))$ and $f_{m}(L_ {n})$ are exponentially equivalent because we have the following uniform estimate
\begin{align}
\bigg|U_{n}(W^{(k)}_{m})-\int W^{(k)}_{m}dL^{\bigotimes k}_{n}\bigg|&\leq \bigg(1-\frac{|I^{k}_{n}|}{n^{k}}\bigg)\bigg(|U_{n}(W^{(k)}_{m} )|+||W^{(k)}_{m}||_{\infty}\bigg)\\
&\leq 2\bigg(1-\frac{|I^{k}_{n}|}{ n^{k}}\bigg)||W^{(k)}_{m}||_{\infty}\overset{n\to+\infty}{\longrightarrow}0.\nonumber
\end{align}
We get that when $m\to+\infty$, $f_{m}(L_{n})$ is a good exponential approximation of $(L_{n},U_{n}(W^{(2)} ),\ldots,U_{n}(W^{(N)}))$.\newline
    
\noindent\textbf{Step 3: LDP.} By Sanov theorem and the LDP approximation theorems, to get the desired LDP, it suffices to show that for all $L>0$, 
\begin{equation}
\sup_{\mu,\quad\mathbf{H}[\mu|\alpha]\leq L}\mathbf{d}(f_{m}(\mu),f(\mu))\overset{m\to+\infty}{\longrightarrow}0.
\end{equation} 
Indeed, for all $\lambda>0$, $L>0$ and $\mu$ such that $\mathbf{H}[\mu|\alpha]\leq L$, by the variational formula of Donsker-Varadhan and Fatou's lemma, we have for all $k\in\{2\ldots,N\}$,
\begin{align}
\int|W^{(k)}_{m}-W^{(k)}|d\mu^{\bigotimes k}\leq\frac{1}{\lambda}\bigg(\mathbf{H }[\mu^{\bigotimes k}|\alpha^{\bigotimes k}]+\log\int e^{\lambda|W^{(k)}_{m}-W^{(k)} |}d\alpha^{\bigotimes k}\bigg) \leq\frac{1}{\lambda}\bigg(kL+\varepsilon(\lambda,m,k)\bigg).
\end{align}
This completes the proof of the theorem because $\lambda$ is arbitrary and for all $\lambda>0$, $\varepsilon(\lambda,m,k)\overset{m\to+\infty}{\longrightarrow}0 $.
\end{proof}
\noindent We are now able to prove the inf-compactness of the mean-field entropy functional.
\begin{proposition}[Inf-compactness of the mean-field entropy functional]\label{infcompactHW}
  From $\mathbf{(H^{(1)}_{VW})}$ and $(\mathbf{H^{(2)}_{VW}})$ in \ref{Hypo1}, the mean-field entropy functional is inf-compact. 
\end{proposition}
\begin{proof}[Proof of \cref{infcompactHW}]\label{exist}
We will do the proof in three steps. We recall that if we have a good rate function, then its infimum on any closed nonempty is reached, that is to say that this infimum is a minimum.\newline

\noindent\textbf{Step 1: $W^{(k)}$ bounded from above.}
In this case, from $\mathbf{(H^{(1)}_{VW})}$ and $(\mathbf{H^{(2)}_{VW}})$ in \ref{Hypo1}, we have for all $\lambda>0,$ 
\begin{equation}
\mathbb{E}[e^{\lambda|W^{(k)}|(X_{1},\ldots,X_{k})}]<+\infty.
\end{equation} 
In principle, large deviations for the $U-$statistic, under $\mathbb{P}:=\alpha^{\bigotimes N}$, $(L_{n},U_{n}(W ^{(2)}),\ldots,U_{n}(W^{(N)}))$ satisfies a large deviations principle on $\mathcal{M}_{1}(\R^{d} )\times\R^{N-1}$ of good rate function $\mathbf{I}_{U}$. As
\begin{align}
\sum_{k=2}^{N}U_{n}(W^{(k)})\quad\textnormal{is continuous in}\quad(L_{n},U_{n}(W^{( 2)}),\ldots,U_{n}(W^{(N)})),\\
\forall p>1,\quad\limsup_{n\to+\infty}\frac{1}{n} \log\mathbb{E}\bigg[e^{-np\sum_{k=2}^{N}U_{n}(W^{(k)})}\bigg]<+\infty,\nonumber
\end{align}
by what precedes and the theorem of R.Ellis, we deduce that $\mu_{n}((L_{n},U_{n}(W^{(2)}),\ldots,U_{n} (W^{(N)}))\in\cdot)$ satisfies a large deviations principle with rate function defined by
\begin{equation}
\Tilde{I}(\mu,x_{2},\ldots,x_{N})=\mathbf{I}_{U}(\mu,x_{2},\ldots,x_{N})+ \sum_{k=2}^{N}x_{k}-\inf_{\mu,x_{2},\ldots,x_{N}}\bigg\{\mathbf{I}_{U}(\mu,x_{2},\ldots,x_{N})+\sum_{k=2}^{N}x_{k}\bigg\}.
\end{equation}
So
\begin{equation}
\Tilde{I}(\mu,x_{2},\ldots,x_{N})=\begin{cases}\mathbf{E}_{W}[\mu]-\inf_{\eta}\mathbf {E} _{W}[\eta]\quad &\textnormal{if}\quad\mathbf{H}[\mu|\alpha]<+\infty,\quad\forall k,\quad x_{k}= \mathbf{W}^{(k)}[\mu],\\+\infty\quad &\textnormal{otherwise}. \end{cases}
\end{equation}
We conclude by the principle of contraction that $\mu_{n}(L_{n}\in\cdot)$ satisfies a PGD of rate function $\mathbf{H}_{W}$. Note in this case that $\mathbf{E}_{W}$ is inf-compact, so $\mathbf{H}_{W}$ too.\newline

\noindent\textbf{Step 2: General case.}
In this case, for all $L>0$, we set $W^{(k)}_{L}:=\min(W^{(k)},L)$. So
\begin{equation}
\mathbf{E}_{W_{L}}[\mu]=\begin{cases}\mathbf{H}[\mu|\alpha]+\sum_{k=2}^{N}\mathbf{W }^{(k)}_{L}[\mu]\quad &\textnormal{if}\quad\mathbf{H}[\mu|\alpha]<+\infty,\\+\infty\quad &\textnormal{otherwise}.\end{cases}
\end{equation}
is inf-compact on $\mathcal{M}_{1}(\R^{d})$ by \textbf{step 1}. This proves that $\mathbf{H}_{W}$ is also inf -compact by passing to the monotonous limit. For all closed $\mathcal{F}\subset\mathcal{M}_{1}(\R^{d})$ and $L>0$, we have
\begin{align}
\mu^{*}_{n}(L_{n}\in\mathcal{F})=\int\mathbb{I}_{L_{n}\in\mathcal{F}}\exp\bigg(-n\sum_{k=2}^{N}U_{n}(W^{(k)})\bigg)d\alpha^{\bigotimes n}&\leq
\int\mathbb{I}_{L_{n}\in\mathcal{F}}\exp\bigg(-n\sum_{k=2}^{N}U_{n}(W^{(k)}_{L})\bigg)d \alpha^{\bigotimes n}\\ &\leq\exp\bigg(-n\inf_{\mu\in\mathcal{F}}\mathbf{E}_{W_{L}}[\mu]+\mathbf{o}(n)\bigg)\nonumber
\end{align}
and this last inequality is given by the LDP for the $U-$statistic and the Varadhan-Laplace lemma. It follows
\begin{equation}
\limsup_{n\to+\infty}\frac{1}{n}\log\mu^{*}_{n}(L_{n}\in\mathcal{F})\leq-\inf_{\mu \in\mathcal{F}}\mathbf{E}_{W_{L}}[\mu]\Longrightarrow\limsup_{n\to+\infty}\frac{1}{n}\log\mu^{* }_{n}(L_{n}\in\mathcal{F})\leq-\inf_{\mu\in\mathcal{F}}\mathbf{E}_{W}[\mu]
\end{equation}
by monotone limit and inf-compactness. In particular, for $\mathcal{F}=\mathcal{M}_{1}(\R^{d})$, we deduce that
\begin{equation}
\limsup_{n\to+\infty}\bigg\{\frac{1}{n}\log Z_{n}-\log C\bigg\}\leq-\inf_{\mu\in\mathcal{M} _{1}(\R^{d})}\mathbf{E}_{W}[\mu].
\end{equation}
By the lower bound of the large deviations for $L_{n}$ under $\mu^{*}_{n}$ obtained, this upper bound and given that $\mathbf{E}_{W}[\mu]= +\infty$ if for a $k\in\{2,\ldots,N\}$, $W^{(k)}\notin L^{1}(\mu^{\bigotimes k})$, we derive that
\begin{equation}
\lim_{n\to+\infty}\bigg\{\frac{1}{n}\log Z_{n}-\log C\bigg\}=-\inf_{\mu\in\mathcal{M}_ {1}(\R^{d})}\mathbf{E}_{W}[\mu]
\end{equation}
which is a finite quantity by assumptions and inf-compactness. With this equality, we thus obtain upper and lower bounds of large deviations for $\{\mu_{n}(L_{n}\in\cdot)\}_{n\geq N}$.

\end{proof}
\begin{proposition}[Sanov's theorem for the Wasserstein metric by Wang et.al]\label{Thmsanov}
Let $(X_{n})_{n\geq 1}$ be a sequence of independent random variables, identically distributed, with values in $\R^{d}$ endowed with one of its norms that we will denote $||\cdot||$ and law $\alpha$. We have equivalence between the following two assertions
\begin{enumerate}
    \item $(\mathbb{P}(L_{n}\in\cdot))_{n\geq 1}$ satisfies a principle of large deviations on the Wasserstein space $(\mathcal{M}^{p }_{1}(\R^{d}),\mathcal{W}_{p})$ with speed $n$ and good rate function $\mathbf{H}[\cdot|\alpha]$ .
    \item \begin{equation}
    \forall\lambda>0\quad x_{0}\in\R^{d},\quad\int_{\R^{d}}e^{\lambda||x-x_{0} ||^{p}}\alpha(dx)<+\infty.\end{equation}
\end{enumerate}
\end{proposition}
\begin{proof}
Since we have established a LDP for the random empirical measure $L_{n}$ under $\mu_{n}$ on $\mathcal{M}_{1}(\R^{d})$ equipped with the topology of weak convergence, it suffices to prove the exponential tension of $(\mu_{n}(L_{n}\in\cdot))_{n\geq N}$ on $(\mathcal{M}^{ p}_{1}(\R^{d}),\mathcal{W}_{p})$.\\
Let $K\subset\mathcal{M}^{p}_{1}(\R^{d})$ be compact and $(a,b)\in [1,+\infty]^{2}$ a pair of conjugate exponents ($\frac{1}{a}+\frac{1}{b}=1$). By Holder's inequality, we have
\begin{align}
\mu_{n}(L_{n}\notin K)&=\frac{C^{n}}{Z_{n}}\int\mathbb{I}_{L_{n}\notin K}\exp\bigg(-n\sum_{k=2}^{N}U_{n}(W^{(k)})\bigg)d\alpha^{\bigotimes n}\\ 
&\leq\frac{C^{n} }{Z_{n}}\bigg(\alpha^{\bigotimes n}(L_{n}\notin K)\bigg)^{\frac{1}{a}}\bigg(\int\exp\bigg (-nb\sum_{k=2}^{N}U_{n}(W^{(k)})\bigg)d\alpha^{\bigotimes n}\bigg)^{\frac{1}{ b}}.\nonumber
\end{align}
It is deduced that
\begin{align}
\limsup_{n\to+\infty}\frac{1}{n}\log\mu_{n}(L_{n}\notin K)&\leq
\frac{1}{a}\limsup_{n\to+\infty}\frac{1}{n}\log\alpha^{\bigotimes n}(L_{n}\notin K)- \limsup_{n\to+\infty}\frac{1}{n}\log\frac{Z_{n}}{C^{n}}\\ &+\frac{1}{b}\limsup_{n\to+\infty}\frac{1}{n}\int\exp\bigg(-n\sum_{k=2}^{N}U_{n}(bW^{(k)})\bigg)d\alpha^{\bigotimes n}.\nonumber
\end{align}
Now the right-hand side of this inequality is upper bounded by
\begin{equation}
\frac{1}{a}\limsup_{n\to+\infty}\frac{1}{n}\log\alpha^{\bigotimes n}(L_{n}\notin K)+\inf_{\mu \in\mathcal{M}_{1}(\R^{d})}\mathbf{E}_{W}[\mu]-\frac{1}{b}\inf_{\mu\in\mathcal{M}_{1}(\R^{d})}\mathbf{E}_{bW}[\mu],
\end{equation}
and from the above, $\inf_{\mu\in\mathcal{M}_{1}(\R^{d})}\mathbf{E}_{W}[\mu]$ and
\begin{equation}
\inf_{\mu\in\mathcal{M}_{1}(\R^{d})}\mathbf{E}_{bW}[\mu]:=\inf_{\mu\in\mathcal{ M}_{1}(\R^{d})}\bigg\{\mathbf{H}[\mu|\alpha]+\sum_{k=2}^{N}\int bW^{(k )}d\mu^{\bigotimes k}\bigg\},\quad\textnormal{are finite quantities.}
\end{equation}
Under $\mathbf{(H4)}$ in \ref{Hypo1}, the LDP holds for $L_{n}$ under $\alpha^{\bigotimes n}$
on the Wasserstein space. So, for all $L>0$, there is a compact $K_{L}\subset\mathcal{M}^{p}_{1}(\R^{d})$ such that
\begin{equation}
\limsup_{n\to+\infty}\frac{1}{n}\log\alpha^{\bigotimes n}(L_{n}\notin K_{L})\leq -aL-a\inf_{\mu \in\mathcal{M}_{1}(\R^{d})}\mathbf{E}_{W}[\mu]+\frac{a}{b}\inf_{\mu\in\mathcal{M}_{1}(\R^{d})}\mathbf{E}_{bW}[\mu].
\end{equation}
It follows that
\begin{equation}
\limsup_{n\to+\infty}\frac{1}{n}\log\mu_{n}(L_{n}\notin K_{L})\leq -L.
\end{equation}
\end{proof}

\noindent\textbf{Uniqueness of invariant measure.}
The assumptions on the interaction potentials and the confinement potential ensure the existence (via the inf-compactness of the entropy functional proven in \cref{exist} and \cite{Ref1}) of an invariant measure (global minimum point for the entropy functional) for the McKean-Vlasov process obtained by propagation of chaos. It remains to prove the uniqueness. To do this, we will use the characterization of the local extrema of a differentiable functional in the sense of Fr\'echet (flat derivation) on an open set. Let 
\begin{align}
\mathcal{O}&:=\bigg\{\mu\in\mathcal{P}(\R^{d}),\quad\mathbf{H}[\mu|\alpha]<+\infty,\quad \forall k,\quad\int W^{(k),-}d\mu^{\bigotimes k}<+\infty\bigg\}\\
&=\mathbf{H}[\cdot|\alpha]^{- 1}(]-\infty,+\infty[)\bigcap\Psi^{-1}(]-\infty,+\infty[^{N-1}),\nonumber
\end{align}
with 
\begin{equation}
\Psi:\mu\longmapsto\bigg(\int W^{(2),-}d\mu^{\bigotimes 2},\ldots,\int W^{(N),-}d\mu^{ \bigotimes N}\bigg).\end{equation} 
We know that $\mathbf{E}_{W}\equiv+\infty$ over $\mathcal{O}^{C}$. By Fr\'echet differentiability of the relative entropy $\mathbf{H}[\cdot|\alpha]$ and of $\Psi$ on $\mathcal{M}_{1}(\R^{d})$ endowed with its structure of differential Fr\'echet manifold, $\mathcal{O}$ is open as an intersection of open sets. We deduce that the local extrema (here minimum) of $\mathbf{E_{W}}$ are critical points on $\mathcal{O}$, i.e.  $\mu\in\mathcal{O}$ such that
\begin{equation}
   Z_{\mu}:=\int e^{-\frac{\delta F}{\delta m}(\mu,x)-V(x)}dx<+\infty,\quad\frac{\delta\mathbf{E}_{W}}{\delta m}(\mu,\cdot)\equiv0\Longleftrightarrow\mu(dx)=\frac{1}{Z_{\mu}}e^{-\frac{\delta F}{\delta m}(\mu,x)-V(x)}dx.
\end{equation}

\begin{proposition}[Fixed point uniqueness]\label{Pfixunique}
The functional
   \begin{equation}
       \Gamma:\mu\in\mathcal{O}\subset\mathcal{M}^{2}_{1}(\R^{d})\longmapsto\Gamma(\mu)(dx):=\frac{1}{Z_{\mu}}e^{-\frac{\delta F}{\delta m}(\mu,x)-V(x)}dx\in\mathcal{O}
   \end{equation}
   admits a unique fixed point.
\end{proposition}
\begin{proof}
    Indeed, according to the hypothesis $\mathbf{H6.}$ of \ref{Hypo1}, we have
    \begin{equation}
        \mathbf{d}_{Lip}(\Gamma(\mu),\Gamma(\nu))\leq k\mathbf{d}_{Lip}(\mu,\nu),
    \end{equation}
    and since there is a fixed point, suppose by absurd that there is more than one, i.e. there is $\mu_{1},\mu_{2}\in\mathcal{O}$ such that $\mu_{1}\neq\mu_{2}$ and for all $i$, $\Gamma(\mu_{i})=\mu_{i}$. It follows that $k\geq 1$ which is absurd because $k<1$.
\end{proof}

\noindent\textbf{Ces\~aro tensorial: About entropies and Fisher Information.}\label{Fisher}
We will establish \textit{convergences in entropy and Fisher information} which are useful for the proof of the \textit{exponential decrease of the mean field entropy} and the establishment of the \textit{nonlinear Talagrand inequality}.
\begin{proposition}[Convergence in relative entropy]\label{Cesaro1}
For any probability measure $\nu$ on $\R^{d}$ such that $\mathbf{H}[\nu|\alpha]<+\infty$, we have:
\begin{equation}
\frac{1}{n}\mathbf{H}[\nu^{\otimes n}|\mu_{n}]\overset{n\to +\infty}{\longrightarrow}\mathbf{H}_ {W}[\nu],\quad\textnormal{where}\quad\mu_{n}\quad\textnormal{is defined in}\quad\cref{IM1}.
\end{equation}
\end{proposition}
\begin{proof}
For $\mu$ such that $\mu\ll\alpha$ and for all $k\in\{2,\ldots,N\}$, $W^{(k),-}\in L^{1 }(\mu^{\bigotimes k})$, we have
\begin{align}\label{cesaro}
\frac{1}{n}\mathbf{H}[\mu^{\bigotimes n}|\mu_{n}]&=\frac{1}{n}\mathbb{E}_{\mu^{\bigotimes n}}\bigg[\frac{d\mu^{\bigotimes n}}{d\alpha^{\bigotimes n}}+n\sum_{k=2}^{N}U_{n}(W^{(k)})+\log\frac{Z_{n}}{C^{n}}\bigg]\\
&=\mathbf{H}[\mu|\alpha]+\sum_{k=2}^ {N}\int W^{(k)}d\mu^{\bigotimes n}+\frac{1}{n}\log Z_{n}-\log C.\nonumber
\end{align}
We recall that $\alpha(dx):=\frac{e^{-V(x)}}{C}dx$. Under the assumption $\mathbf{(H2)}$ in \ref{Hypo1}, we know that $\exists$ $\lambda_{0}>0$ such that:
\begin{equation}
\int_{\R^{d}}e^{\lambda_{0}||x||^{2}}\alpha(dx)<+\infty.
\end{equation}
By asking:
\begin{equation}
\widetilde{Z}_{n}:=\int_{(\R^{d})^{n}}e^{-n\sum_{k=2}^{N}U_{n}(W^{(k)})}\alpha^{\otimes n}(dx_{1},\ldots,dx_{n}),
\end{equation}
we get: (by Fubini-Tonelli)
\begin{equation}
\mu_{n}(dx)=\frac{C^{n}}{Z_{n}}e^{-n\sum_{k=2}^{N}U_{n}(W^ {(k)})}\alpha^{\bigotimes n}(dx).
\end{equation}
Let $\nu\in\mathcal{M}_{1}(\R^{d})$ be such that $\mathbf{H}[\nu|\alpha]<+\infty$. Since 
\begin{equation}
\mathbf{H}[\nu^{\otimes k}|\alpha^{\otimes k}]=k\mathbf{H}[\nu|\alpha],
\end{equation} 
$x\longmapsto e^{ \lambda_{0}||x||^{2}}\in L^{1}(\alpha)$ and 
\begin{equation}
\forall k\in\{2,\ldots,N\}\quad\forall x\in\R^{kd},\quad|W^{(k)}(x)|\leq\beta(1+\sum_{j=1}^{k}||x_{j}||^ {2})
\end{equation} 
by boundedness of its hessian $\nabla^{2}W^{(k)}$ (hypothesis $\mathbf{(H1)}$ in \ref{Hypo1}), according to Donsker-Varadhan variational formula of entropy, we have $W^{(k)}\in L^{1}(\nu^{\otimes k})$. We have successively: (by a direct calculation and application of the Fubini-Tonelli theorem)
\begin{equation}
\frac{1}{n}\mathbf{H}[\nu^{\otimes n}|\mu_{n}]=\frac{1}{n}\mathbf{Ent}_{\mu_{ n}}\bigg[\frac{d\nu^{\otimes n}}{d\mu_{n}}\bigg]=\frac{1}{n}\int_{(\R^{d} )^{n}}\log\bigg(\frac{d\nu^{\otimes n}}{d\mu_{n}}\bigg)d\nu^{\otimes n}
\end{equation}
We deduce that:
\begin{equation}
\frac{1}{n}\mathbf{H}[\nu^{\otimes n}|\mu_{n}]=\frac{1}{n}\int\sum_{i=1}^{ n}\log\bigg(\frac{d\nu}{d\alpha}(x_{i})\bigg)d\nu^{\otimes n}+\sum_{k=2}^{N}\int U_{n}(W^{(k)}d\nu^{\otimes n}+\frac{1}{n}\log(\widetilde{Z}_ {n})
\end{equation}
$\underset{n\to+\infty}{\lim}\frac{1}{n}\log(\widetilde{Z}_{n})=-\underset{\eta\in\mathcal{M} _{1}(\R^{d})}{\mathbf{inf}}\mathbf{E}_{W}[\eta]$ (see \cref{LDP}), 
\begin{equation}
\frac{1} {n}\int\sum_{i=1}^{n}\log\bigg(\frac{d\nu}{d\alpha}(x_{i})\bigg)d\nu^{\otimes n }=\mathbf{H}[\nu|\alpha]
\end{equation} 
and finally, we also have: (see \cref{tcl}) 
\begin{equation}
\sum_{k=2}^{N}\int U_{n}(W^{(k)})d\nu^{\otimes n}=\sum_{k=2}^{N}\int W^{(k)}(x)\nu^{\otimes k}(dx).
\end{equation}
Thereby: 
\begin{equation}
\frac{1}{n}\mathbf{H}[\nu^{\otimes n}|\mu_{n}]\overset{n\to +\infty}{\longrightarrow}\mathbf{H}[ \nu|\alpha]+\sum_{k=2}^{N}\int W^{(k)}(x)\nu^{\otimes k}(dx)-\underset{\eta\in\mathcal{M}_{ 1}(\R^{d})}{\mathbf{inf}}\mathbf{E}_{W}[\eta]=\mathbf{H}_{W}[\nu].
\end{equation}
What needed to be proven.
\end{proof}
\begin{proposition}[Fisher Information Convergence]\label{Cesaro2}
If $\mathbf{I}[\nu|\alpha]<+\infty$, we have:
\begin{equation}
\frac{1}{n}\mathbf{I}[\nu^{\otimes n}|\mu_{n}]\overset{n\to+\infty}{\longrightarrow}\mathbf{I}_{ W}[\nu].
\end{equation}
\end{proposition}
\begin{proof}
For any probability measure $\nu$ on $\R^{d}$ such that $\mathbf{I}[\nu|\alpha]<+\infty$, by the Lyapunov condition $(\mathbf{H2})$ in \ref{Hypo1} on the potential $V$, we have:
\begin{equation}
c_{1}\int||x||^{2}d\nu\leq c_{2}+\mathbf{I}[\nu|\alpha]<+\infty.
\end{equation}
As the second order derivatives of $W^{(k)}$ are bounded by the condition $\mathbf{(H1)}$ in \ref{Hypo1} on its Hessian, $\nabla_{x_{j}}W^{(k)}$ has a linear increase. So $\nabla_{x_{j}}W^{(k)}\in L^{2}(\nu^{\otimes k})$. By the law of large numbers for independent and identically distributed sequences, we have successively:
\begin{align}
\frac{1}{n}\mathbf{I}[\nu^{\otimes n}|\mu_{n}]&=\frac{1}{4n}\int\bigg|\bigg|\nabla\log\bigg(
\frac{d\nu^{\otimes n}}{d\mu_{n}}\bigg)\bigg|\bigg|^{2}d\nu^{\otimes n}\\
&=\frac{1} {4n}\int\sum_{i=1}^{n}\bigg|\bigg|\nabla_{x_{i}}\log\bigg(
\frac{d\nu^{\otimes n}}{d\alpha^{\otimes n}}\bigg)+\sum_{k=2}^{N}\nabla_{x_{i} }U_{n}(W^{(k)})\bigg|\bigg|^{2}d\nu^{\otimes n}\nonumber\\
&=\frac{1}{4}\int\bigg|\bigg|\nabla\log\bigg(\frac{d\nu}{d\alpha}\bigg)(x_{1})+\sum_{k=2}^{N}\nabla_{x_{1} }U_{n}(W^{(k)})\bigg|\bigg|^{2}d\nu^{\otimes n}\nonumber\\
&\overset{n\to+\infty}{\longrightarrow}
\frac{1}{4}\int\bigg|\bigg|\nabla\log\bigg(\frac{d\nu}{d\alpha} \bigg)(y)+\sum_{k=2}^{N}\sum_{j=1}^{k}\int\nabla_{x_{j}}W^{(k)}(x_{1},\ldots,x_{j-1},y,x_{j+1},\ldots,x_{k}) 
\nu^{\otimes k-1}\bigg(\prod_{i=1,i\neq j}^{k}dx_{i}\bigg)\bigg|\bigg|^{2}\nu(dy )\nonumber\\
&=\mathbf{I}_{W}[\nu].\nonumber
\end{align}
\end{proof}
\noindent We recall the tensorisation property of relative entropy: The \cref{entropie} on entropy and tensor product allows us, in what follows, to show the exponential decreasing of mean-field entropy along the flow of solution distributions of the McKean-Vlasov equation associated with the particle system.
\begin{proposition}[Relative entropy and tensor product]\label{entropie}
Let $\prod_{i=1}^{N}\alpha_{i}$ and $Q$ respectively be a product probability measure and a probability measure defined on $E_{1}\times\cdots\times E_{N }$ a product of Polish spaces. Denoting $Q_{i}$ the marginal distribution of $x_{i}$ under $Q$, we have:
\begin{equation}
\mathbf{H}[Q|\prod_{i=1}^{N}\alpha_{i}]\geq\sum_{i=1}^{N}\mathbf{H}[Q_{i}|\alpha_{i}].
\end{equation}
\end{proposition}
\begin{proof}
See \cref{proof6} or \cite{Ref2}.
\end{proof}
\begin{proposition}[Relative entropy and Boltzmann measure]\label{bolz}
Let $\mu$ be a probability measure on a Polish space $E$ and $U:E\longrightarrow(-\infty,+\infty]$ be a measurable potential such that:
\begin{equation}
\int e^{-pU}d\mu<+\infty
\end{equation}
for some $p>1.$ Considering the Boltzmann probability measure $\mu_{U}:=\frac{e^{-U}}{C}d\mu$, if for some measure $\nu $, $\mathbf{H}[\nu|\mu_{U}]<+\infty$, we have successively:
\begin{enumerate}
    \item $\mathbf{H}[\nu|\mu]<+\infty$ and $U\in L^{1}(\nu)$.
    \item \begin{equation}
    \mathbf{H}[\nu|\mu_{U}]=\mathbf{H}[\nu|\mu]+\int Ud\nu+\log\int e^{-U}d\mu.
    \end{equation}
\end{enumerate}
\end{proposition}
\begin{proof}
See \cref{proof7} or \cite{Ref2}.
\end{proof}
\noindent\textbf{Functional and transportation inequalities.}
Functional inequalities are powerful tools to quantify the trend to equilibrium of Markov semigroups and have
a wide range of important applications to the concentration of measure phenomenon and hypercontractivity.
$\forall n$, we recall that $\mu_{n}(t):=\mathbb{P}\circ(X^{n}_{t})^{-1}$ and $\beta_{n}:=\rho_{LS}(\mu_{n})$.
\begin{theorem}[Transportation inequalities]\label{transport}
Under the assumptions in \ref{Hypo1}, we have
\begin{enumerate}
    \item 
    \begin{align}
    &\mathbf{H}[\mu_{n}(t)|\mu_{n}]\leq\mathbf{H}[\mu_{n}(0)|\mu_{n}]e^{-\beta_{n}\frac{t}{2}}=\mathbf{H}[\mu^{\bigotimes n}_{0}|\mu_{n}]e^{-\beta_{n}\frac{ t}{2}};\\
    &\rho_{LS}(\mu_{n})\mathbf{H}[\cdot|\mu_{n}]\leq 2\mathbf{I}[\cdot|\mu_{n}];\\
    &\rho_{LS}(\mu_{n})\mathcal{W}^{2}_{2}(\cdot,\mu_{n})\leq 2\mathbf{H}[\cdot|\mu_{n}].
    \end{align}
    \item $\exists$! $\mu_{\infty}\in\mathcal{M}_{1}(\R^{d})$ such that: (\cref{sanov}.\cref{Pfixunique})
    \begin{equation}
    \mu_{\infty}=\mathbf{argmin}\bigg\{\mathbf{H}_{W}[\nu], \nu\in\mathcal{M}_{1}(\R^{d} )\bigg\},
    \end{equation}
    with $\mathbf{H}_{W}$ the mean field entropy.
    \item $\rho_{LS}:=\underset{n\to+\infty}{\limsup}\textnormal{ }\rho_{LS}(\mu_{n})>0$ checks:
    \begin{align}
    \forall\nu\in\mathcal{M}_{1}(\R^{d}),\quad\rho_{LS}\mathbf{H}_{W}[\nu]\leq 2\mathbf{I}_{W}[\nu]\quad\textnormal{and}\quad\rho_{LS}\mathcal{W}^ {2}_{2}(\nu,\mu_{\infty})\leq 2\mathbf{H}_{W}[\nu].
    \end{align}
    We say that we have a nonlinear log-Sobolev inequality for the first inequality and a Talagrand transport inequality for the second.
\end{enumerate}
\end{theorem}
\begin{proof}[Proof of \cref{transport}]\label{sketch}
The logarithmic Sobolev inequality of constant $\beta_{n}:=\rho_{LS}(\mu_{n})$ for $\mu_{n}$ given by $\mathbf{(H5)}$ in \ref{Hypo1}, the large deviations principle ( Sanov's theorem) in \cref{sanov} and the uniqueness of the minimum argument ($\mu_{\infty}$) in \cref{Pfixunique} of the mean field entropy ensure that we have successively:
\begin{itemize}
    \item $\forall\mu$ such as  $\mathbf{H}[\mu|\alpha]<+\infty,$ (\cref{Fisher}.\cref{Cesaro1}.\cref{Cesaro2})
    \begin{align}
    \frac{1}{n}\mathbf{H}[\mu^{\bigotimes n}|\mu_{n}]\overset{n\to+\infty}{\longrightarrow}\mathbf{H}_{W}[\mu]\quad\textnormal{and}\quad
    \frac{1}{n}\mathbf{I }[\mu^{\bigotimes n}|\mu_{n}]\overset{n\to+\infty}{\longrightarrow}\mathbf{I}_{W}[\mu].
    \end{align}
    \item Equivalence between Sobolev's inequality, exponential decay of entropy and Talagrand's second inequality for Gibbs measures (Otto-Villani,\cite{OV00},\cite{Villani}) 
    \begin{align}
    \beta_{n}\mathbf{H}[\cdot|\mu_{n}]\leq 2\mathbf{I}[\cdot|\mu_{n}]\quad\textnormal{and}\quad 
    \beta_{n}\mathcal{W}^{2}_{2}(\cdot,\mu_{n})\leq 2\mathbf{H}[\cdot|\mu_{n}].
    \end{align}
    \item \textit{Chaos propagation.} (\cite[\cite{kacprogramm}]{propachaos1}) Denoting $(\mu_{t})_{t\geq 0}$ the flow of solution distributions of the McKean-Vlasov equation associated with the particle system defined by the $U-$ statistic and the confinement potential, if $\mu_{0}\in\mathcal{M}^{2}_{1}(\R^{d})$, then for any non-empty set $ I\subset\N^{*}$ of finite cardinality, $\mathbb{P}_{(X^{n}_{t}(i))_{i\in I}}$ converges in metric $L^{2}-$Wasserstein to $\mu^{\bigotimes\mathbf{Card}(I)}_{t}$ (arrow $\mathbf{(1)}$ in \cref{fig1}).
    \item Denoting $\mu^{(i)}_{n}$ the i-th marginal distribution of $\mu_{n}$, we have by uniqueness and LDP (arrow $\mathbf{(3)}$ in \cref{fig1})
    \begin{equation}
    \mu^{(i)}_ {n}\overset{\mathcal{L}}{\longrightarrow}\mu_{\infty}.
    \end{equation}
    \item By symmetry of $\mu_{n}$, all its marginal distributions are identical and as 
    \begin{equation}
    \mathcal{W}^{2}_{2}(\mu^{\bigotimes n},\mu_{n })\geq\sum_{i=1}^{n}\mathcal{W}^{2}_{2}(\mu^{(i)}_{n},\mu)=n\mathcal{ W}^{2}_{2}(\mu^{(1)}_{n},\mu),
    \end{equation}
    we deduce: 
    \begin{equation}
    n\beta_{n}\mathcal{W}^{2}_{2}(\mu^{(1)}_{n},\mu)\leq 2\mathbf{H}[ \mu^{\bigotimes n}|\mu_{n}].\end{equation}
\end{itemize}
By equivalence of the logarithmic Sobolev inequality to the exponential decrease of entropy along the semigroup, we have (arrow $\mathbf{(2)}$ in \cref{fig1})
\begin{equation}
\mathbf{H}[\mu_{n}(t)|\mu_{n}]\leq\mathbf{H}[\mu_{n}(0)|\mu_{n}]e^{-\beta_{n}\frac{t}{2}}=\mathbf{H}[\mu^{\bigotimes n}_{0}|\mu_{n}]e^{-\beta_{n}\frac{ t}{2}},\quad\mu_{n}(t):=\mathbb{P}\circ (X^{n}_{t})^{-1}.
\end{equation}
And by lower semi-continuity of the Wasserstein metric, we deduce the \textit{nonlinear $T_{2}-$Talagrand inequality} given by (arrow $\mathbf{(4)}$ in \cref{fig1})
\begin{equation}
\rho_{LS}\mathcal{W}^{2}_{2}(\mu,\mu_{\infty})\leq\rho_{LS}\liminf_{n\to+\infty}\mathcal{W} ^{2}_{2}(\mu,\mu^{(1)}_{n})\leq 2\mathbf{H}_{W}[\mu],\quad\rho_{LS}: =\limsup_{n\to+\infty}\beta_{n}>0.
\end{equation}
We also have the \textit{nonlinear logarithmic Sobolev inequality} given by (arrow $\mathbf{(4)}$ in \cref{fig1})
\begin{equation}
\rho_{LS}\mathbf{H}_{W}[\cdot]\leq 2\mathbf{I}_{W}[\cdot].
\end{equation}
\end{proof}

\noindent\textbf{In Kinetic case.}\label{prelvfp} 
We consider $\mathcal{H}_{n}:=\Delta_{x}-\nabla_{x}S_{1,n}\cdot\nabla=\mathcal{L}_{n}$ the elliptical generator associated with $\mu_{1,n}=\mu_{n}$ . 
\begin{remark}\label{remarkhormanderform}
$\mathcal{L}_{Z,n}$ admits the following Hormander form
\begin{equation}\label{hormanderkfp}
\mathcal{L}_{Z,n}=X_{0}+Y+\sum_{i=1}^{n}\sum_{j=1}^{d}X^{2}_{i,j}, \quad X_{i,j}=\frac{\partial}{\partial v_{i,j}}\quad X_{0}=-v\cdot\nabla_{v}\quad Y=\nabla S_{1 ,n}\cdot\nabla_{v}-v\cdot\nabla_{x}
\end{equation}
The family 
\begin{equation}
\bigg\{X_{1,1},\ldots,X_{i,j},\ldots,X_{i,j},\cdots,X_{n,d},[Y,X_{ 1,1}],\ldots,[Y,X_{i,j}],\ldots,[Y,X_{n,d}]\bigg\}
\end{equation}
form a basis of $\R^{2nd} $ at any point. Which implies by Hormander's theorem that $\mathcal{L}_{Z,n}$ is hypoelliptic. Moreover, $\mathcal{L}_{Z,n}$ is non-symmetric, i.e. in $L^{2}(\mu^{n}_{Z})$, we have :
\begin{equation}
\mathcal{L}^{*}_{Z,n}=\mathcal{L}_{Z,n}-2Y\Longrightarrow(\mathcal{L}^{*}_{Z,n},\mathcal{D}(\mathcal {L}^{*}_{Z,n})) \quad\textnormal{is not a closed extension of}\quad (\mathcal{L}_{Z,n},\mathcal{D}(\mathcal{L }_{Z,n})).
\end{equation}
\end{remark}
\noindent The following known lemma is a key to the Lyapunov type conditions. We include its simple proof for completeness.
\begin{proposition}[Lemma.8 in \cite{Ref3}]\label{prelvfp1}
 For any function $\varphi\in\mathcal{C}^2(\R^{nd})$ strictly positive ($\varphi>0$), we have
\begin{equation}
\forall\psi\in\mathcal{H}^{1}(\mu_{1,n}),\quad\int-\frac{\mathcal{H}_{n}\varphi}{\varphi}\psi^2d\mu_{1,n}\leq\int|\nabla\psi|^2d\mu_{1,n}.
\end{equation}
\end{proposition}
\begin{proof}[Proof of \cref{prelvfp1}]
 Indeed, by integrating by parts, we successively obtain
\begin{align}
\int-\frac{\mathcal{H}_{n}\varphi}{\varphi}\psi^2d\mu_{1,n}&\leq\int\bigg\langle\nabla\varphi,\nabla\frac {\psi^2}{\varphi}\bigg\rangle d\mu_{1,n}\\ 
&\leq\int\bigg\langle\nabla\varphi,\frac{2\psi\nabla\psi}{\varphi} -\frac{\psi^2\nabla\varphi}{\varphi^2}\bigg\rangle d\mu_{1,n}\nonumber\\ 
&\leq\int|\nabla\psi|^2d\mu_{1,n} .\nonumber
\end{align}
And this last inequality follows from the inequality
\begin{equation}
\bigg\langle 2\psi\nabla\psi,\frac{\nabla\varphi}{\varphi}\bigg\rangle\leq\frac{\psi^2|\nabla\varphi|^2}{\varphi^ 2}+|\nabla\psi|^2.
\end{equation}   
\end{proof}
\noindent This second \cref{prelvfp2} is the heart of the proof of \cref{theo3}: this proposition is inspired by \cite[Lemma.10]{Ref3} for the two-body interaction. It uses Lyapunov conditions, yet well know for being highly dimensional, but at the marginal level, thus providing results independent of the number of particles.

\begin{proposition}\label{prelvfp2}
Under the conditions in \ref{Hypo2} giving $\mathbf{UPI}$, there are two constants $C_{1}$ and $C_{2}$ depending on $N,K,K_{1},K_{2}$ and $d$ (dimension of $\R^{d}$) and such that
\begin{equation}
\forall\psi\in\mathcal{H}^1(\mu_{1,n}),\quad\int||\nabla^2 V(x_{i})||^2_{\textbf{op} }\psi^2d\mu_{1,n}\leq C_{1}\int|\nabla_{x}\psi|^2d\mu_{1,n}+C_{2}\int\psi^2d\mu_{1,n}.
\end{equation}  
\end{proposition}
\begin{proof}[Proof of \cref{prelvfp2}]
 This lemma follows from the Lyapunov property, from the particular form of the invariant measure generator\footnote{We have $$\mathcal{H}_{n}=\sum_{i=1}^{n}\mathcal{T}_{i},\quad\mathcal{T}_{i}:=\Delta_{x_{i}}-\nabla V(x_{i})\cdot\nabla_{x_{i}}-n\sum_{k=2}^{N}\nabla_{x_{i}}U_{n}(W^{(k)})\cdot\nabla_{x_{i}}.$$} $\mu_{1,n}$ and from the previous \cref{prelvfp1}. Indeed, we have:
    \begin{enumerate}
        \item\begin{align} ||\nabla^2V||^2_{\textbf{op}}\leq\eta_{1}\bigg((1-\gamma)||\nabla V ||^2-\Delta V\bigg)+\eta_{2},\\ \eta_{1}:=5K^2_{1}\quad\eta_{2}:=4K^2_{2}+\frac{25K^4_1d^2}{4}\quad\gamma:=\frac{1}{5}.\end{align}
        \item Since the interactions are Lipschitz, we know that $\forall k\in\{2,\ldots,N\}\exists K^{(k)}$ such that $||\nabla W^{(k) }||\leq K^{(k)}$.\newline Let $K:=\max\{K^{(k)},\quad k=2,\ldots,N\}$. It follows that
        \begin{equation}
        -n\sum_{k=2}^{N}\nabla_{x_{i}}U_{n}(W^{(k)})\cdot\nabla V(x_{i})\leq (N- 1)K|\nabla V|(x_{i})\leq(N-1)\bigg(\frac{K^2}{2\gamma}+\frac{\gamma}{2}|\nabla V |^2(x_{i})\bigg).
        \end{equation}
        But for $\varphi(x):=e^{\frac{\gamma}{2}V(x_{i})}$, we have
        \begin{equation}
        \frac{\mathcal{H}_{n}\varphi}{\varphi}=\frac{\mathcal{T}_{i}\varphi}{\varphi}=\frac{\gamma}{2}\bigg(\Delta V(x_{i})+(\frac{\gamma}{2}-1)|\nabla V|^2(x_{i})-n\sum_{k=2}^{N }\nabla_{x_{i}}U_{n}(W^{(k)})\cdot\nabla V(x_{i})\bigg).
        \end{equation}
        Thereby
        \begin{equation}
        2\frac{\mathcal{H}_{n}\varphi}{\gamma\varphi}\leq\Delta V(x_{i})+(\frac{N\gamma}{2}-1)||\nabla V||^2(x_{i})+\frac{(N-1)K^2}{2\gamma}.
        \end{equation}
        Moreover, we have
        \begin{equation}
            (1-\gamma)||\nabla V||^{2}(x_{i})-\Delta V(x_{i})\leq-2\frac{\mathcal{H}_{n}\varphi}{\gamma\varphi}+\frac{(N-1)K^2}{2\gamma}.
        \end{equation}
        Therefore, by the inequality obtained in \textbf{(i)},
        \begin{equation}
            ||\nabla^{2}V(x_{i})||^{2}_{\textbf{op}}\leq\eta_{1}\bigg(-2\frac{\mathcal{H}_{n}\varphi}{\gamma\varphi}+\frac{(N-1)K^2}{2\gamma}\bigg)+\eta_{2}
        \end{equation}
        Integrating with respect to $\psi^{2}d\mu_{1,n}$, we obtain
        \begin{equation}
            \int ||\nabla^{2}V(x_{i})||^{2}_{\textbf{op}}\psi^{2}d\mu_{1,n}\leq\frac{2\eta_{1}}{\gamma}\int-\frac{\mathcal{H}_{n}\varphi}{\varphi}\psi^2d\mu_{1,n}+\bigg(\eta_{2}+\frac{(N-1)K^2}{2\gamma}\eta_{1}\bigg)\int\psi^2d\mu_{1,n}.
        \end{equation}
        And we conclude by the previous \cref{prelvfp1} that
        \begin{equation}
        \int||\nabla^2V(x_{i})||^2_{\textbf{op}}\psi^2d\mu_{1,n}\leq C_{1}\int|\nabla_{x}\psi| ^2d\mu_{1,n}+C_{2}\int\psi^2d\mu_{1,n},
        \end{equation}
        where $C_{1}=\frac{2\eta_{1}}{\gamma}$ and $C_{2}=\eta_{2}+\frac{(N-1)K^2}{2\gamma}\eta_{1}$.
    \end{enumerate}
   
\end{proof}

\section{Proofs of Main Theorems}\label{proofsmain}
\begin{proof}[Proof of \cref{theo1}]
Indeed, we have the inequality (\cref{entropie}) 
\begin{equation}\frac{1}{n}\mathbf{H}[\mu_{n}(t)|\alpha^{\bigotimes n}]\geq\mathbf{H}[\mu^{(1)}_{n}(t)|\alpha]
\end{equation} 
and by lower semi-continuity of relative entropy and propagation of chaos, 
\begin{equation}\underset{n\to+\infty}{\liminf}\quad\mathbf{H}[\mu^{(1)}_{n}(t)|\alpha]\geq\mathbf{H}[\mu_{t}|\alpha]\quad.
\end{equation} 
On the other hand, we have (\cref{LDP}.\cref{transport}.\cref{cesaro})
\begin{align}
\frac{1}{n}\mathbf{H}[\mu_{n}(t)|\mu_{n}]\leq\frac{1}{n}\mathbf{H}[\mu^{\bigotimes n}_{0}|\mu_{n}]e^{-\beta_{n}\frac{t}{2}}\quad\textnormal{and}\quad
\liminf_{n\to+\infty}\frac{1}{ n}\mathbf{H}[\mu^{\bigotimes n}_{0}|\mu_{n}]e^{-\beta_{n}\frac{t}{2}}=\mathbf{H }_{W}[\mu_{0}]e^{-\rho_{LS}\frac{t}{2}}.
\end{align}
Also, as 
\begin{equation}\mu_{n}(dx)=\frac{C^{n}}{Z_{n}}e^{-n\sum_{k=2}^{N}U_{n}( W^{(k)})}\alpha^{\bigotimes n}(dx),
\end{equation}
we also have
\begin{equation}
\frac{1}{n}\mathbf{H}[\mu_{n}(t)|\mu_{n}]=\frac{1}{n}\mathbf{H}[\mu_{n}( t)|\alpha^{\bigotimes n}]+\sum_{k=2}^{N}\int U_{n}(W^{(k)})d\mu_{n}(t)+\bigg(\frac{1}{n}\log(Z_{n})-\log(C)\bigg),
\end{equation}
and (\cref{TCL}.\cref{cesaro})
\begin{align}
\sum_{k=2}^{N}\int U_{n}(W^{(k)})d\mu_{n}(t)\overset{n\to+\infty}{\longrightarrow}\sum_ {k=2}^{N}\int W^{(k)}d\mu^{\bigotimes k}_{t}=\sum_{k=2}^{N}\mathbf{W}^{(k )}[\mu_{t}],\quad
\frac{1}{n}\log(Z_{n})-\log(C)\overset{n\to+\infty}{\longrightarrow}-\inf_ {\mu\in\mathcal{M}_{1}(\R^{d})}\mathbf{E}_{W}[\mu].
\end{align}
It is deduced that
\begin{align}
\forall t\geq 0,\quad\mathbf{H}_{W}[\mu_{0}]e^{-\rho_{LS}\frac{t}{2}}\geq\liminf_{n\to+\infty}\frac{1}{n}\mathbf{H}[\mu_{n}(t)|\mu_{n}]\geq
\mathbf{H}[\mu_{t}|\alpha]+\sum_{k=2}^{N}\mathbf{W}^{(k)}[\mu_{t}]-\inf_{\mu\in\mathcal{M}_{ 1}(\R^{d})}\mathbf{E}_{W}[\mu]
=\mathbf{H}_{W}[\mu_{t}].
\end{align}
This completes the proof of the exponential decrease of entropy along the flow.
\end{proof}
\begin{proof}[Proof of \cref{theo2}]
Just use the nonlinear $T_{2}-$Talagrand inequality, i.e.: (\cref{transport})
\begin{equation}
\forall t\geq 0,\quad\rho_{LS}\mathcal{W}^{2}_{2}(\mu_{t},\mu_{\infty})\leq 2\mathbf{H}_ {W}[\mu_{t}].
\end{equation}
This completes the proof of the desired inequality: We conclude with the \cref{theo1}.
\end{proof}

\begin{proof}[Proof of \cref{theo3}]
By the Lyapunov condition in the assumptions \ref{Hypo2}, we can apply \cref{prelvfp2} and obtain that for any $\psi\in\mathcal{H}^{1}(\mu_{1,n})$, it holds 
\begin{equation}
        \int||\nabla^2V(x_{i})||^2_{\textbf{op}}\psi^2d\mu_{1,n}\leq C_{1}\int|\nabla_{x}\psi| ^2d\mu_{1,n}+C_{2}\int\psi^2d\mu_{1,n},
\end{equation}
with $C_{1}=\frac{2\eta_{1}}{\gamma}$ and $C_{2}=\eta_{2}+\frac{(N-1)K^2}{2\gamma}\eta_{1}$ for instance which are independent of the number $n$ of particles. It follows that the boundedness condition in Villani's theorem holds. Since the uniform Sobolev inequality implies the uniform Poincar\'e inequality, we can apply Villani's hypocoercivity theorem (\cite[Theorem.3]{Ref3} or \cite[Theorem.18 and Theorem.35]{villanihypo}), which completes the proof.
\end{proof}

\begin{proof}[Proof of \cref{theo4}]
Note that $(\mu^{n}_{Z}(t))_{t\geq 0}$ is a solution of a (large dimensional) linear Fokker-Planck equation, for which the exponential decay of the entropy is already known under assumptions including \ref{Hypo2} (see e.g. \cite{villanihypo}). Consider the generator $\mathcal{L}_{Z,n}$ given by \cref{generatorcm2}. Then $\Psi_{n}:=\frac{d\mu^{n}_{Z}(t)}{d\mu^{n}_{Z}}$, the density of the law of the particle system given by  \cref{CM2} with respect to its equilibrium distribution, solves
\begin{equation}
    \partial_{t}\Psi_{n}=\mathcal{L}^{*}_{Z,n}\Psi_{n}.
\end{equation}
This is a linear kinetic Fokker-Planck equation, for which convergence to equilibrium has been proven by many ways. All we need to check is that the explicit estimates we obtain do not depend on $n$ (see e.g. respectively Theorem.7 and Theorem.10 in \cite[\cite{Monmarche2}]{Monmarche1}). The key point in \cref{vfpcv1} is that $C$ and $\xi$ do not depend on $n$: Indeed, as $\mu^{n}_{Z}=\mu_{1,n}\bigotimes\mu_{2,n}$ and these measures satisfy logarithmic Sobolev inequalities of constants $\rho_{LS }(\mu_{1,n})=\rho$ and $\rho_{LS}(\mu_{2,n})=1$, $\mu^{n}_{Z}$ satisfies an inequality of logarithmic Sobolev of constant $\rho_{LS}(\mu^{n}_{Z}):=\max(\rho,1)$. 
This enables us to prove the following: ($T2-$inequality)
\begin{equation}
    \exists\kappa>0\quad\forall n\quad\forall t,\quad\mathcal{W}^{2}_{2}(\mu^{n}_{Z}(t),\mu^{n}_{Z})\leq \kappa e^{-\xi t}\mathbf{H}[\mu^{n}_{Z}(0)|\mu^{n}_{Z}],\quad\mu^{n}_{Z}(0)=\mu^{\otimes n},\quad\mu\in\mathcal{P}_{2}(\R^{d}\times\R^{d}).
\end{equation}
By symmetry, propagation of chaos and Sanov's theorem (LDP), we have respectively
\begin{align}
    \forall i\in\{1,\ldots,n\},\quad n\mathcal{W}^{2}_{2}(\mu^{n,(i)}_{Z}(t),\mu^{n,(i)}_{Z})\leq\mathcal{W}^{2}_{2}(\mu^{n}_{Z}(t),\mu^{n}_{Z}),\quad
    \mu^{n,(i)}_{Z}(t)\overset{n\to+\infty}{\longrightarrow}\mu^{\mathbf{VFP}}_{t},\quad
    \mu^{n,(i)}_{Z}\overset{n\to+\infty}{\longrightarrow}\mu^{Z}_{\infty}.
\end{align}
We have by lower semi-continuity
\begin{align}
    &\forall\mu\in\mathcal{P}_{2}(\R^{d}\times\R^{d}),\quad\mathcal{W}^{2}_{2}(\mu,\mu^{Z}_{\infty})\leq\liminf_{n\to+\infty}\mathcal{W}^{2}_{2}(\mu,\mu^{n,(i)}_{Z})\leq\kappa\liminf_{n\to+\infty}\frac{1}{n}\mathbf{H}[\mu^{n}_{Z}(0)|\mu^{n}_{Z}]=\kappa\mathcal{S}[\mu],\\
    &\frac{1}{n}\mathbf{H}[\mu^{\otimes n}|\mu^{n}_{Z}]=\mathbf{H}[\mu|\alpha\otimes\mathcal{N}(0,\mathbf{Id}_{d})]+\sum_{k=2}^{N}\int U_{n}(W^{(k)})d\mu^{\otimes n}+\frac{1}{n}\log(Z_{n})-\log(C)
    \overset{n\to+\infty}{\longrightarrow}\mathcal{S}[\mu]:=\mathcal{E}[\mu]-\mathcal{E}[\mu^{Z}_{\infty}].
\end{align}
According to \cref{vfpcv1}, we have
\begin{equation}
    \mathcal{S}[\mu^{\mathbf{VFP}}_{t}]\leq C\mathcal{S}[\mu]e^{-\xi t}.
\end{equation}
It follows that
\begin{equation}
    \mathcal{W}^{2}_{2}(\mu^{\mathbf{VFP}}_{t},\mu^{Z}_{\infty})\leq\kappa\mathcal{S}[\mu^{\mathbf{VFP}}_{t}]\leq\kappa C\mathcal{S}[\mu]e^{-\xi t}.
\end{equation}
\end{proof}

\newpage
\appendix
\section{Appendix and Proofs}
\subsection{McKean-Vlasov theory}\label{McKeanTheory}
\begin{theorem}[Existence and uniqueness of solutions of \cref{PMckeanG}]\label{ExistuniqueMcV}
Let us assume that the functions $b$ and $\sigma$ are globally Lipschitz: $\exists K>0$ $\forall (x,y,\mu,\nu)\in\R^{D}\times\R^{D}\times\mathcal{P}_{2}(\R^{D})\times\mathcal{P}_{2}(\R^{D}),$ 
\begin{equation}\label{condLip}
    ||b(x,\mu)-b(y,\nu)||+|||\sigma(x,\mu)-\sigma(y,\nu)|||\leq K\bigg(||x-y||+\mathcal{W}_{2}(\mu,\nu)\bigg),
\end{equation}
 where $||\cdot||$ denotes a vector norm, $|||\cdot|||$ is a matrix norm and $\mathcal{W}_{2}$ denotes the Wasserstein-2 distance. Assume that $\mu_{0}\in\mathcal{P}(\R^{D})$. Then for any $T\geq 0$ the SDE \cref{PMckeanG} has a unique strong solution on $[0, T ]$ and consequently, its law is the unique weak solution to the Fokker-Planck equation \cref{GNFPlanck} and the unique solution to the nonlinear martingale problem \cref{EquaNLMP}.
\end{theorem}
\noindent The proof of this theorem is fairly classical. This proof is based on a \textit{fixed point argument} that is sketched in \cite[Proposition.1]{propachaos1}.
\begin{theorem}[Polynomial Potential]\label{LDPPolinter}
   Let $E$ be a Polish measurable space. Let $\alpha\in\mathcal{P}(E)$. Let us consider a random vector $X^{n}$ in $E^{n}$, distributed according to the Gibbs measure:
   \begin{equation}\label{Gibbs1}
   \mu_{n}(dx):=\frac{1}{Z_{n}}e^{nF(\mu_{x})}\alpha^{\otimes n}(dx),
   \end{equation}
   where $Z_{n}$ is a normalization constant and $F$ is a polynomial function on $\mathcal{P}(E)$ (called the
energy functional) of the form given by \cref{GFFE}. Then (for some symmetric continuous bounded functions $W^{(k)}$)  the laws of $\mu_{X^{n}}$ satisfy a large deviation principle in $\mathcal{P}(\mathcal{P}(E))$ with speed $\frac{1}{n}$ and rate function
\begin{equation}\label{Rate1}
    \mu\longmapsto\mathbf{H}[\mu|\alpha]-F(\mu)-\inf_{\eta\in\mathcal{P}(E)}\{\mathbf{H}[\eta|\alpha]-F(\eta)\}.
\end{equation}
\end{theorem}

\subsection{Gibbs-Laplace Variational Principle}
\begin{definition}[Distribution support]
Let $\mu$ be a probability measure on a Polish space $E$ (or even a measure on a topological space!). We call support of $\mu$ noted $\textbf{supp}(\mu)$ the closed set defined by
\begin{equation}
\bigcap_{F\subset E\textnormal{ closed, }\mu(F)=1}F=\bigg(\bigcup_{O\subset E\textnormal{ open, }\mu(O)=0}
O\bigg)^{C}.
\end{equation}
In other words, the support of a distribution is the complement of the largest open set over which it is zero: the smallest closed set of maximum mass!
\end{definition}
\begin{definition}[Extremum essential]
Let $\mu$ be a probability measure on a Polish space $E$ and $V:E\longrightarrow [-\infty,+\infty]$ measurable. We call infimum $\mu-$essential of $V$ the quantity
\begin{equation}
\mu-\textbf{essinf} V:=\inf\{v\in\R,\quad\mu(\{V\leq v\})>0\}
\end{equation}
\end{definition}
\begin{theorem}[Variational principle]
For any probability measure $\mu$ on a topological space $\Omega$ and any measurable function $V:\Omega\longrightarrow\overline{\R}$, we have
\begin{equation}
\lim_{n\to+\infty}\frac{1}{n}\log\int e^{-n V}d\mu=-\mu-\mathbf{essinf} V.
\end{equation}
Moreover, if $V$ is upper semicontinuous, then
\begin{equation}
\inf_{\mathbf{supp}(\mu)}V=\mu-\mathbf{essinf}V.
\end{equation}
\end{theorem}
\noindent\textit{Proof Sketch:} Suppose $\mu-\textbf{essinf}V$ is finite. Check that we can assume without loss of generality
that $V\geq 0$ and $\mu-\textbf{essinf}V = 0$. Then check $\mathbb{I}_{V \leq\varepsilon} e^{-n\varepsilon}\leq e^{-nV }\leq1$ and conclude. Show that the limit is $+\infty$ with the lower bound when $\mu-\textbf{essinf}V =-\infty$.
\begin{proof}
 \begin{itemize}
    \item $\mu-\textbf{essinf}V$ is finished: 
    \begin{equation}
    \frac{1}{n}\log\int e^{-n V}d\mu+\mu-\textbf{essinf} V= \frac{1}{n}\log\bigg(\int e^{-n(V-\mu-\textbf{essinf}V)}d\mu\bigg).
    \end{equation}
    This implies that we can assume without loss of generality that $V\geq0$ and $\mu-\textbf{essinf}V = 0$ because $V-\mu-\textbf{essinf}V\geq0$ almost surely, its essential infimum under $\mu$ is zero and the convergence that interests us is equivalent to
    \begin{equation}
    \frac{1}{n}\log\bigg(\int e^{-n(V-\mu-\textbf{essinf}V)}d\mu\bigg)\overset{n\to+\infty}{ \longrightarrow}
    0
    \end{equation}
    But for all $\varepsilon>0=\mu-\textbf{essinf}V,$
    \begin{equation}
    \mathbb{I}_{V \leq\varepsilon} e^{-n\varepsilon}\leq e^{-nV }\leq1\Longleftrightarrow\frac{\log\mu(V\leq\varepsilon)}{n }-\varepsilon\leq\frac{1}{n}\log\int e^{-n V}d\mu\leq0.
    \end{equation}
    We deduce that by the bounding limit theorem, we have
    \begin{equation}
    \limsup_{n}\frac{1}{n}\log\int e^{-n V}d\mu=\liminf_{n}\frac{1}{n}\log\int e^{-n V}d\mu=0.
    \end{equation}
    \item $\mu-\textbf{essinf}V=-\infty$: In this case, for all $v\in\R$, we have $\mu(V\leq v)>0$ and
    \begin{equation}
    \int_{\Omega}e^{-nV}d\mu\geq\int_{\{V\leq v\}} e^{-nV}d\mu\geq e^{-nv}\mu(\{V\leq v\}).
    \end{equation}
    It is deduced that
    \begin{align}
    \forall v\in\R,\quad\frac{1}{n}\log\int_{\Omega} e^{-n V}d\mu\geq -v++\frac{\log\mu(V\leq v)}{n}\\
    \Longrightarrow\lim_{n\to+\infty}\frac{1}{n}\log\int e^{-n V}d\mu=+\infty=-\mu-\textbf{essinf} V.\nonumber
    \end{align}
\end{itemize}
\end{proof}
\begin{theorem}[Gibbs measures and deviations]
Let $E$ be a Polish space, $\mu$ a probability measure on $E$ and $V: E\longrightarrow\overline{\R}$ a measurable function. We have:
\begin{itemize}
    \item 
    \begin{equation}
    \inf_{\mathbf{supp}(\mu)}V\leq\mu-\mathbf{essinf}V.
    \end{equation}
    \item If $V$ is upper semicontinuous, then
    \begin{equation}
    \inf_{\mathbf{supp}(\mu)}V\geq\mu-\mathbf{essinf}V\Longrightarrow\inf_{\mathbf{supp}(\mu)}V=\mu-\mathbf{essinf} V.
    \end{equation}
\end{itemize}
In particular, if $V$ is continuous, then the principle of large deviations holds for 
\begin{equation}
\mu_{n}(dx):=\frac{1}{\int e^{-nV}d\mu}e ^{-nV(x)}\mu(dx)
\end{equation}
with rate function $I^{V}:=V+I_{0}-\inf\{V+I_{0}\}$ with 
\begin{equation}
I_{0}(x):=\begin{cases}0\quad\textnormal{if }x\in\mathbf{supp}(\mu),\\+\infty\quad\textnormal{else.} \end{cases} 
\end{equation}
\end{theorem}
\noindent\textit{Proof sketch:}
\begin{itemize}
    \item Show that
    \begin{equation}
    \bigg\{x,\quad V(x)<\inf_{\textbf{supp}(\mu)}V\bigg\}\bigcap\textbf{supp}(\mu)=\emptyset.
    \end{equation}
    Then conclude.
    \item For all $\varepsilon>0$, show that
    \begin{equation}
    \bigg\{x,\quad V(x)<\inf_{\textbf{supp}(\mu)}V+\varepsilon\bigg\}\quad\textnormal{is an open containing a support element: }
    \end{equation}
    their intersection is non-empty; then conclude.
\end{itemize}

\subsection{Principle of contraction and tensorization}\label{ldpcontraction}
Let $f: X\longrightarrow G$ be continuous between two Polish spaces and $(X_{N})$ a random variable sequence of $X$ satisfying the principle of large deviations of rate function $I:X\longrightarrow[0, +\infty]$. Then $((f(X_{N}))$ satisfies the principle of large deviations of rate function $J: G\longrightarrow[0,+\infty]$ such that 
\begin{equation}
J(g):=\inf_{ f^{-1}(\{g\})}I.
\end{equation}
Let $(X_{n})_{n\geq 1}$ and $(Y_{n})_{n\geq 1}$ be sequences with values respectively in $E_{1}$ and $E_{2} $, independent ($\mathbb{P}_{(X_{n},Y_{n})}=\mathbb{P}_{X_{n}}\bigotimes\mathbb{P}_{Y_{n} }$) and both satisfying the principle of large deviations of the respective good rate functions $I_{1}$ and $I_{2}$. Then $((X_{n},Y_{n}))_{n\geq 1}$ satisfies the principle of large deviations on the product space and of good rate function $I$ defined by 
\begin{equation}
I(x, y):=I_{1}(x)+I_{2}(y).
\end{equation}
\subsection{Entropy and Chaos}
\begin{theorem}[Characterization of relative entropy: Sanov's theorem]\label{thmH}
Let $\mu$ and $\nu$ be probability measures (even finite!) on a Polish space $E$ and $(\varphi_{j})_{j\in\N}$ a dense sequence of functions bounded uniformly continuous. we have
\begin{align}
\lim_{k\to\infty}\lim_{\varepsilon\to0}\lim_{n\to\infty}\frac{1}{n}\log\mu^{\otimes n}\bigg(\bigg\{y\in E^{n};\quad\forall j\in\{1,\ldots,k\},\quad\bigg|\int_{E}\varphi_{j}d\nu-\frac{ 1}{n}\sum_{i=1}^{n}\varphi_{j}(y_{i})\bigg|\leq\varepsilon\bigg\}\bigg)
=-\mathbf{H}[\nu|\mu].
\end{align}
We interpret $n$ as the number of particles; the $\varphi_{j}$ a sequence of observables whose mean value is measured; and $\varepsilon$ as the precision of the measurements. This formula concisely summarizes the essential information contained in the Boltzmann function $\mathbf{H}$.
\end{theorem}
\begin{theorem}[(strict) convexity of relative entropy]
Let $\mu\in\mathcal{P}(\Omega)$. $\mathbf{H}[\cdot|\mu]$ has values in $\overline{\R_{+}}$, convex, strictly convex on $\{\nu,\quad\mathbf{H}[\nu|\mu]<+\infty\}$ and is zero only in $\mu.$
\end{theorem}
\begin{theorem}[Tensorization property]
Let $\mu\in\mathcal{P}(\Omega)$, $\nu\in\mathcal{P}(\Omega^n)$ with $\nu_{i}$ its $i-$th marginal. So
\begin{equation}
\mathbf{H}[\nu|\mu^{\bigotimes n}]=\mathbf{H}[\nu|\bigotimes_{i=1}^{n}\nu_{i}]+\sum_{i =1}^{n}\mathbf{H}[\nu_{i}|\mu]
\end{equation}
\end{theorem}
\begin{theorem}[Villani]
Let $(X:=(X_{1},\ldots,X_{n})$ be a random variable on $E^{n}$ with $E$ a Polish space, $\mu_{n}:=\mathbb{P}_{X}\in\mathcal{P}(E^{n})$, $\delta_{X}:=\frac{1}{n}\sum\delta_{X_{i}} $ and $\mu\in\mathcal{P}(E)$. The following assertions are equivalent:
\begin{itemize}
    \item $\delta_{X}$ converges in law to $\mu$:
    \begin{equation}
    \forall\varphi\in\mathcal{C}_{b}(E),\quad\int\varphi d\delta_{X}\overset{n\to+\infty}{\longrightarrow}\int\varphi d\mu\quad\textnormal{almost surely}.
    \end{equation}
    \item 
    \begin{equation}
    \forall\varphi\in\textbf{Lip}_{b}(E),\quad\lim_{n\to+\infty}\mathbb{E}_{\mu_{n}}\bigg[ \bigg|\int\varphi d(\delta_{X}-\mu)\bigg|\bigg]=0.
    \end{equation}
\end{itemize}
\end{theorem}
\noindent Without repeating the proof, we can say that this result is obtained by defining a metric on $\mathcal{P}(E)$ from a dense sequence of Lipschitz functions and then by defining the transport distance  Wasserstein's $\mathcal{W}_{1}$ on $\mathcal{P}(\mathcal{P}(E))$ associated with this metric. Using this result, we can more formally prove the propagation of chaos.
\begin{definition}[U-statistics]\label{Ustats}
Let $E$ be a set, $k\in\mathbb{N}^{*}$ and $\Phi:E^{k}\longrightarrow\R$ a symmetric function. Then the application: ($n\geq k$)
\begin{equation}
X:=(x_{j})_{j=1,\ldots,n}\in E^{n}\longmapsto U(X):=\frac{k!(n-k)!}{n!}\sum_{1\leq i_{1}<i_{2}<\ldots<i_{k}\leq n}\Phi(x_{i_{1}},\ldots,x_{i_{k}})
\end{equation}
is called $U-$statistic of order $k$ and kernel $\Phi$. $U(X)$ is called $U-$statistic of order $k$ and kernel $\Phi$ associated with the sample $X.$ This statistic corresponds to the arithmetic mean of the kernel over all the parts at $ k$ elements of the set of sample values. we often write $U_{n}(\Phi)(X):=U(X)$. If $E$ is a measurable space, we generalize this definition to the space of probabilities by the functional $\mu\longmapsto\mathbb{E}_{\mu^{\bigotimes k}}[\Phi]$.
\end{definition}
\subsection{Proofs}\label{proofs}
all these proofs are inspired by \cite{Ref1} by setting $S=\R^{d}$.
\begin{proof}[Proof of \cref{tcl}]\label{proof1}
Let $\mathfrak{G}_{n}$ be the group of permutations of $\{1,\ldots,n\}$ and $\mathfrak{B}_{n}$ the $\sigma-$algebra defined by
\begin{equation}
\mathfrak{B}_{n}:=\sigma\bigg\{B_{n}\times C_{n}|C_{n}\in\mathcal{B}(E^{[n+1,+\infty[}),\quad B_{n}\in\mathcal{B}(E^{n}),\quad\forall\tau\in\mathfrak{G}_{n},\quad\tau\mathbb {I}_{B_{n}}=\mathbb{I}_{B_{n}}\bigg\}.
\end{equation}
This $\sigma-$algebra is invariant under permutations and verifies for all $n\geq 1$, 
\begin{equation}
\mathfrak{B}_{n+1}\subset\mathfrak{B}_{n}.
\end{equation} 
By integrability, 
\begin{equation}
\forall(i_{1},\ldots,i_{k})\in I^{k}_{n},\quad
\mathbb{E}[\Phi(X_{i_{1}},\ldots,X_{i_{k}})|\mathfrak{B}_{n}]=\mathbb{E}[\Phi( X_{1},\ldots,X_{k})|\mathfrak{B}_{n}],
\end{equation} 
which implies that 
\begin{equation}
U_{n}(\Phi)=\mathbb{E}[\Phi(X_ {1},\ldots,X_{k})|\mathfrak{B}_{n}].
\end{equation} 
According to the limit theorems on martingales (closed martingale) and the law of $0-1$ applied to the asymptotic tribe $\mathfrak{B}_{\infty}:=\underset{n\geq 1}{\bigcap }\mathfrak{B}_{n}$, we deduce that we almost surely have
\begin{equation}
U_{n}(\Phi)\overset{n\to+\infty}{\longrightarrow}\mathbb{E}[\Phi(X_{1},\ldots,X_{k})|\mathfrak{B} _{\infty}]=\mathbb{E}[\Phi(X_{1},\ldots,X_{k})].
\end{equation}
\end{proof}

\begin{proof}[Proof of \cref{jensen}]\label{proof2}
We prove this result by induction. Indeed, for $k=1$ the inequality is verified since we have equality of the two members. Suppose that for $k-1$ the inequality holds. Denote by $B_{k}$ the left side of this inequality. We have
\begin{align}
B_{k}=\log\mathbb{E}^{X^{k}}\bigg[\mathbb{E}\bigg[\exp\bigg(\frac{1}{|I^{k-1} _{n}|}\sum_{(i_{1},\ldots,i_{k-1})\in I^{k-1}_{n}}\sum_{i_{k}\notin\{ i_{1},\ldots,i_{k-1}\}} \frac{1}{n-k+1}\Phi_{i_{1},\ldots,i_{k}}(X^{1 }_{i_{1}},\ldots,X^{k}_{i_{k}})\bigg)
\bigg|X^{k}\bigg]\bigg]
\end{align}
with $X^{k} :=(X^{k}_{1},\ldots,X^{k}_{n}).$
Let's pose
\begin{equation}
\Tilde{\Phi}_{i_{1},\ldots,i_{k-1}}:=\frac{1}{n-k+1}\sum_{i_{k}\notin\{i_{ 1},\ldots,i_{k-1}\}}\Phi_{i_{1},\ldots,i_{k}}(X^{1}_{i_{1}},\ldots,X^ {k}_{i_{k}}).
\end{equation}
By induction hypothesis, we deduce that
\begin{equation}
B_{k}\leq\log\mathbb{E}^{X^{k}}\bigg[\exp\bigg(\frac{n-k+2}{|I^{k-1}_{n }|}\sum_{(i_{1},\ldots,i_{k-1})\in I^{k-1}_{n}}\log\mathbb{E}\bigg[\exp\bigg (\frac{1}{n-k+2}\Tilde{\Phi}_{i_{1},\ldots,i_{k-1}}\bigg)\bigg|X^{k}\bigg] \bigg)\bigg].
\end{equation}
Since 
\begin{equation}
\log\mathbb{E}^{X^{k}}\bigg[\exp\bigg(\frac{n-k+2}{|I^{k-1}_{n}|} \sum_{(i_{1},\ldots,i_{k-1})\in I^{k-1}_{n}}\log\mathbb{E}\bigg[\exp\bigg(\frac {1}{n-k+2}\Tilde{\Phi}_{i_{1},\ldots,i_{k-1}}\bigg)\bigg|X^{k}\bigg]\bigg) \bigg]
\end{equation} 
is upper bounded by
\begin{equation}
\log\mathbb{E}^{X^{k}}\bigg[\exp\bigg(\frac{1}{|I^{k-1}_{n}|}\sum_{(i_{1 },\ldots,i_{k-1})\in I^{k-1}_{n}}\log\bigg(\mathbb{E}\bigg[\exp\bigg(\frac{1}{ n-k+2}\Tilde{\Phi}_{i_{1},\ldots,i_{k-1}}\bigg)\bigg|X^{k}\bigg]\bigg)^{n- k+2}\bigg)\bigg],
\end{equation}
by convexity of $X\longmapsto\log\mathbb{E}[e^{X}]$ (consequence of Holder's inequality), we have
\begin{equation}
B_{k}\leq\frac{1}{|I^{k-1}_{n}|}\sum_{(i_{1},\ldots,i_{k-1})\in I^{ k-1}_{n}}\log\mathbb{E}^{X^{k}}\bigg[\bigg(\mathbb{E}\bigg[\exp\bigg(\frac{1}{n -k+2}\Tilde{\Phi}_{i_{1},\ldots,i_{k-1}}\bigg)\bigg|X^{k}\bigg]\bigg)^{n-k +2}\bigg].
\end{equation}
In this last inequality, for all $(i_{1},\ldots,i_{k-1})$, the logarithmic term verifies
\begin{align}
&\mathbb{E}^{X^{k}}\bigg[\bigg(\mathbb{E}\bigg[\exp\bigg(\frac{1}{n-k+2}\Tilde{\Phi} _{i_{1},\ldots,i_{k-1}}\bigg)\bigg|X^{k}\bigg]\bigg)^{n-k+2}\bigg]\\
&=\mathbb{E }^{X^{k}}\bigg[\bigg(\mathbb{E}\bigg[\exp\bigg(\frac{1}{(n-k+2)(n-k+1)}\sum_{i_{k}\notin\{i_{1},\ldots,i_{k-1}\}}\Phi_{i}(X^{1}_{i_{1}},\ldots,X ^{k}_{i_{k}})\bigg)\bigg|X^{k}\bigg]\bigg)^{n-k+2}\bigg],\nonumber
\end{align}
and by Holder's inequality, we have
\begin{align}
&\mathbb{E}^{X^{k}}\bigg[\bigg(\mathbb{E}\bigg[\exp\bigg(\frac{1}{n-k+2}\Tilde{\Phi}_{i_{1},\ldots,i_{k-1}}\bigg)\bigg|X^{k}\bigg]\bigg)^{n-k+2}\bigg]\\
&\leq\mathbb{E}^{X^{k}}\bigg[\bigg(\prod_{i_{k}\notin\{i_{1},\ldots,i_{k-1}\}}\mathbb{E}\bigg[\exp\bigg(\frac{1}{n-k+2}\Phi_{i_{1},\ldots,i_{k}}(X^{1}_{i_{1}},\ldots,X^{k}_{i_{k}})\bigg)\bigg|X^{k}\bigg]\bigg)^{\frac{n-k+2}{n-k+1} }\bigg].\nonumber
\end{align}
By Jensen's inequality, we also have the upper bound of the right-hand side of this last inequality by
$$
\mathbb{E}^{X^{k}}\bigg[\prod_{i_{k}\notin\{i_{1},\ldots,i_{k-1}\}}\mathbb{E}\bigg[\exp\bigg(\frac{1}{n-k+1}\Phi_{i_{1},\ldots,i_{k}}(X^{1}_{i_{1}},\ldots,X^{k}_{i_{k}})\bigg)\bigg|X^{k}\bigg]\bigg],
$$
and by independence, this quantity is equal to
\begin{equation}
\prod_{i_{k}\notin\{i_{1},\ldots,i_{k-1}\}}\mathbb{E}\bigg[\exp\bigg(\frac{1}{n-k +1}\Phi_{i_{1},\ldots,i_{k}}(X^{1}_{i_{1}},\ldots,X^{k}_{i_{k}})\bigg)\bigg].
\end{equation}
\end{proof}

\begin{proof}[Proof of \cref{decouplage1}]\label{proof3}
Let $((X^{j}_{1},\ldots,X^{j}_{n}))_{j=1,\ldots,k}$ be independent copies of $(X_{1} ,\ldots,X_{n}).$ By the two propositions above, setting for all $i\in I^{k}_{n}$, $\Phi_{i_{1},\ldots ,i_{k}}\equiv W^{(k)}$, we have for all $\lambda>0$
\begin{equation}
\Lambda_{n}(\lambda,W^{(k)})=\frac{1}{n}\log\mathbb{E}\bigg[\exp\bigg(\frac{\lambda n}{| I^{k}_{n}|}\sum_{i\in I^{k}_{n}}W^{(k)}(X_{i_{1}},\ldots,X_{i_{ k}})\bigg)\bigg],
\end{equation}
and it follows that
\begin{align}
\Lambda_{n}(\lambda,W^{(k)})&\leq\frac{1}{n}\log\mathbb{E}\bigg[\exp\bigg(\frac{1}{|I ^{k}_{n}|}\sum_{i\in I^{k}_{n}}\lambda nC_{k}|W^{(k)}|(X^{1}_{i_ {1}},\ldots,X^{k}_{i_{k}})\bigg)\bigg]\\
&\leq\frac{1}{n|I^{k-1}_{n}|} \sum_{i\in I^{k}_{n}}\log\mathbb{E}\bigg[\exp\bigg(\frac{\lambda nC_{k}}{n-k+1}|W ^{(k)}|(X^{1}_{i_{1}},\ldots,X^{k}_{i_{k}})\bigg)\bigg].\nonumber
\end{align}
Gold
\begin{align}
\frac{1}{n|I^{k-1}_{n}|}\sum_{i\in I^{k}_{n}}\log\mathbb{E}\bigg[\exp\bigg(\frac{\lambda nC_{k}}{n-k+1}|W^{(k)}|(X^{1}_{i_{1}},\ldots,X^{k}_{i_{k}})\bigg)\bigg]=
\frac{n-k+1}{n}\log\mathbb{E}\bigg[\exp\bigg(\frac{\lambda nC_{k }}{n-k+1}|W^{(k)}|(X_{i_{1}},\ldots,X_{i_{k}})\bigg)\bigg].
\end{align}
It is deduced that
\begin{align}
\Lambda_{n}(\lambda,W^{(k)})&\leq\frac{n-k+1}{n}\log\mathbb{E}\bigg[\exp\bigg(\frac{\lambda nC_{k}}{n-k+1}|W^{(k)}|(X_{i_{1}},\ldots,X_{i_{k}})\bigg)\bigg]\\
&\leq\frac{1}{k}\log\mathbb{E}\bigg[\exp\bigg(kC_{k}\lambda|W^{(k)}|(X_{1},\ldots,X_{k })\bigg)\bigg],\nonumber
\end{align}
and this last inequality is obtained by growth on $(0,+\infty)$ of $a\longmapsto\frac{1}{a}\log\mathbb{E}[e^{aX}]$ and from the fact that for all $n$ and $k$ such that $n\geq k$, we have $\frac{n}{n-k+1}\leq k.$
\end{proof}

\begin{proof}[Proof of \cref{lower}]\label{proof4}
To do this, we will show that for any probability measure $\mu$ such that $\mathbf{H}[\mu|\alpha]<+\infty$ and for any $k$, $W^{(k )}\in L^{1}(\mu^{\bigotimes k})$, we have $l^{*}(\mathcal{O})\geq-\mathbf{E}_{W}[\mu]$. Let $\mathbb{B}(\mu,\delta)$ be the open ball with center $\mu$
and of radius $\delta>0$ in $\mathcal{M}_{1}(\R^{d})$ endowed with the L\'evy-Prokhorov metric $d_{LP}$ such that $\mathbb{B }(\mu,\delta)\subset\mathcal{O}$. Let's introduce the events
\begin{itemize}
    \item 
    \begin{equation}
    A_{n}:=\bigg\{x\in(\R^{d})^{n}\bigg|\quad L_{n}:=L_{n}(x,\cdot)\in\mathbb{B}(\mu,\delta)\bigg\};\end{equation}
    \item 
    \begin{equation}
    B_{n}:=\bigg\{x\in(\R^{d})^{n}\bigg|\quad\frac{1}{n}\sum_{i=1}^{ n}\log\frac{d\mu}{d\alpha}(x_{i})=\frac{1}{n}\log\bigg(\frac{d\mu}{d\alpha}\bigg )^{\bigotimes n}(x)\leq\mathbf{H}[\mu|\alpha]+\varepsilon\bigg\};
    \end{equation}
    \item 
    \begin{equation}
    C_{n}:=\bigg\{x\in(\R^{d})^{n}\bigg|\sum_{k=2}^{N}U_{n}(W^{ (k)})\leq\sum_{k=2}^{N}\mathbf{W}^{(k)}[\mu]+\varepsilon\bigg\}.
    \end{equation}
\end{itemize}
We deduce that for all $\varepsilon>0$, we have
\begin{align}
\mu^{*}_{n}(L_{n}\in\mathbb{B}(\mu,\delta))\geq\int_{A_{n}}\bigg(\frac{d\mu^ {\bigotimes n}}{d\mu^{*}_{n}}(x)\bigg)^{-1}\mu^{\bigotimes n}(dx)=
\int_{A_{n}}e^{-\sum_{i=1}^{n}\log\frac{d\mu}{d\alpha}(x_{i})}e^{-n\sum_{k=2}^{ N}U_{n}(W^{(k)})}\mu^{\bigotimes n}(dx)
\end{align}
and
\begin{align}
\int_{A_{n}}e^{-\sum_{i=1}^{n}\log\frac{d\mu}{d\alpha}(x_{i})}e^{-n\sum_{k=2}^{ N}U_{n}(W^{(k)})}\mu^{\bigotimes n}(dx)\geq\mu^{\bigotimes n}(A_{n}\cap B_{n}\cap C_{n})e^{-n(\mathbf{H}[\mu|\alpha]+\varepsilon)-\gamma}\quad
\textnormal{with }\gamma:=n(\sum_{k=2}^{N}\mathbf{W}^{(k)}[\mu]+\varepsilon).
\end{align}
Thereby
\begin{equation}
\mu^{*}_{n}(L_{n}\in\mathbb{B}(\mu,\delta))\geq\mu^{\bigotimes n}(A_{n}\cap B_{n }\cap C_{n})e^{-n\mathbf{E}_{W}[\mu]-2n\varepsilon}.
\end{equation}
We will prove that 
\begin{equation}
\mu^{\bigotimes n}(A_{n}\cap B_{n}\cap C_{n})\overset{n\to+\infty}{\longrightarrow}1.
\end{equation} 
Indeed, by the law of large numbers, we have
\begin{equation}
\mu^{\bigotimes n}(A_{n})\overset{n\to+\infty}{\longrightarrow}1,\quad\mu^{\bigotimes n}(B_{n})\overset{n\to+\infty}{\longrightarrow}1.
\end{equation}
Moreover, by the law of large numbers for $U-$statistics (\cref{TCL}), we also have 
\begin{equation}
\mu^{\bigotimes n}(C_{n})\overset{n\to+\infty}{\longrightarrow}1.
\end{equation} 
It is deduced that
\begin{equation}
l^{*}(\mathcal{O})\geq\underset{n\to+\infty}{\liminf}\frac{1}{n}\mu^{*}_{n}(L_{n} \in\mathbb{B}(\mu,\delta))\geq-\mathbf{E}_{W}[\mu]-2\varepsilon,
\end{equation}
and we conclude by letting $\varepsilon$ tend to zero.
\end{proof}

\begin{proof}[Proof of \cref{approx}]\label{proof5}
To do this, consider the truncation function 
\begin{equation}
W^{(k),L}:=\max(-L,\min(W^{(k)},L)).
\end{equation} 
We have by Lebesgue's dominated convergence theorem
\begin{equation}
\log\mathbb{E}[\exp(m|W^{(k),L}-W^{(k)}|(X_{1},\ldots,X_{k}))]\overset{ L\to+\infty}{\longrightarrow}0.
\end{equation}
So we can choose $L=L(m)$ so that
\begin{equation}
\log\mathbb{E}[\exp(m|W^{(k),L(m)}-W^{(k)}|(X_{1},\ldots,X_{k}))] \leq\frac{1}{m}.
\end{equation}
For $m\geq 1$ and $L(m)>0$ fixed, we can find a sequence $(W^{(k),L}_{l})_{l\geq 1}$ of continuous functions bounded such that
\begin{equation}
W^{(k),L}_{l}(X_{1},\ldots,X_{k})\overset{l\to+\infty, L^{1}}{\longrightarrow}W^{( k),L}(X_{1},\ldots,X_{k}),\quad\forall l\geq 1,\quad |W^{(k),L}_{l}(X_{1} ,\ldots,X_{k})|\leq L,
\end{equation}
otherwise, we consider the truncation $\max(-L,\min(W^{(k),L}_{l},L))$. Seen that $\forall l\geq 1,$
\begin{align}
\exp\bigg(m\bigg(|W^{(k)}-W^{(k),L}|(X_{1},\ldots,X_{k} )+|W^{(k)}-W^{(k),L}_{l}|(X_{1},\ldots,X_{k})\bigg)\bigg)\leq
\exp\bigg(m|W^{(k)}-W^{(k),L}|(X_{1},\ldots,X_{k})+2mL\bigg),
\end{align}
by dominated convergence, we have
\begin{align}
\mathbb{E}\bigg[\exp\bigg(m\bigg(|W^{(k)}-W^{(k),L}|(X_{1},\ldots,X_{k}) +|W^{(k)}-W^{(k),L}_{l}|(X_{1},\ldots,X_{k})\bigg)\bigg)\bigg]\overset{ l\to+\infty}{\longrightarrow}
\mathbb{E}\bigg[\exp\bigg(m|W^{(k)}-W^{(k),L}|(X_{1},\ldots,X_{k})\bigg)\bigg].
\end{align}
For $L=L(m)$, we can choose $l=l(m)$ so that
\begin{equation}
\log\mathbb{E}\bigg[\exp\bigg(m\bigg(|W^{(k)}-W^{(k),L}|(X_{1},\ldots,X_{k })+|W^{(k)}-W^{(k),L}_{l}|(X_{1},\ldots,X_{k})\bigg)\bigg)\bigg]\leq\frac{2}{m}.
\end{equation}
By setting $W^{(k)}_{m}=W^{(k),L(m)}_{l(m)}$ bounded continuous function, we have by triangular inequality
\begin{equation}
\log\mathbb{E}\bigg[\exp\bigg(m|W^{(k)}-W^{(k)}_{m}|(X_{1},\ldots,X_{k} )\bigg)\bigg]\leq\frac{2}{m}.
\end{equation}
Since by Jensen's inequality, we have for all $\lambda>0$,
\begin{align}
\forall m\geq\lambda,\quad\log\mathbb{E}\bigg[\exp\bigg(\lambda|W^{(k)}-W^{(k)}_{m}|(X_ {1},\ldots,X_{k})\bigg)\bigg]\leq
\frac{\lambda}{m}\mathbb{E}\bigg[\exp\bigg(m|W^{(k) }-W^{(k)}_{m}|(X_{1},\ldots,X_{k})\bigg)\bigg],
\end{align}
we deduce that
\begin{equation}
\log\mathbb{E}\bigg[\exp\bigg(\lambda|W^{(k)}-W^{(k)}_{m}|(X_{1},\ldots,X_{k })\bigg)\bigg]\overset{m\to+\infty}{\longrightarrow}0.
\end{equation}
For all $\delta>0$ and $\lambda>0$, by the Markov-Tchebychev inequality, we have
\begin{equation}
\mathbb{P}(|U_{n}(W^{(k)})-U_{n}(W^{(k)}_{m})|>\delta)\leq e^{-n \lambda\delta}\mathbb{E}\bigg[\exp\bigg(n\lambda U_{n}(|W^{(k)}-W^{(k)}_{m}|)\bigg )\bigg].
\end{equation}
From the above (\cref{TCL}), we deduce that
\begin{align}
\frac{1}{n}\log\mathbb{P}(|U_{n}(W^{(k)})-U_{n}(W^{(k)}_{m})|>\delta)\leq
-\lambda\delta+\frac{1}{k}\log\mathbb{E}\bigg[\exp\bigg(kC_{k}\lambda|W^{(k)}-W^ {(k)}_{m}|(X_{1},\ldots,X_{k})\bigg)\bigg].
\end{align}
We conclude that we have the expected result when $m\to+\infty$ since $\lambda$ is arbitrary.
\end{proof}
\begin{proof}[Proof of \cref{entropie}]\label{proof6}
Let $Q^{i}(\cdot|x_{[1,i-1]})$ be the conditional distribution of $x_{i}$ knowing $x_{[1,i-1]}:=(x_{ 1},\cdots,x_{i-1})$ (not knowing if $i=1$).
We have:
\begin{align}
\mathbf{H}[Q|\prod_{i=1}^{N}\alpha_{i}]=\mathbb{E}_{Q}\bigg[\log\bigg(\frac{dQ}{d \prod_{i=1}^{N}\alpha_{i}}\bigg)\bigg]=\mathbb{E}_{Q}\bigg[\sum_{i=1}^{N}\log\bigg(\frac{Q^{i}(dx_{i}|x_{[1,i-1]})}{\alpha_{i}(dx_{i})}\bigg)\bigg]
=\mathbb{E}_{Q}\bigg[\sum_{i=1}^{N}\mathbf{H}[Q^{i}(\cdot|x_{[1,i-1]})|\alpha_ {i}]\bigg].
\end{align}
Since 
\begin{equation}
\mathbb{E}_{Q}[Q^{i}(\cdot|x_{[1,i-1]})]=Q_{i}(\cdot),
\end{equation} 
we obtain by convexity of the relative entropy (Jensen's inequality):
\begin{equation}
\mathbb{E}_{Q}\bigg[\mathbf{H}[Q^{i}(\cdot|x_{[1,i-1]})|\alpha_{i}]\bigg]\geq \mathbf{H}[Q_{i}|\alpha_{i}]
\end{equation}
Which shows that we have the result of the proposition.\end{proof}
\begin{proof}[Proof of \cref{bolz}]\label{proof7}
For $f$ a measurable function on $E$, we define:
\begin{equation}
\Lambda_{\mu}(f):=\log(\E_{\mu}[e^{f}])=\log\int e^{f}d\mu\in(-\infty,+\infty]
\end{equation}
the log-Laplace transformation under $\mu$ which is convex in $f$ by Holder's inequality. We have:
\begin{equation}
\Lambda_{\mu_{U}}(f)=\log\int e^{f}d\mu_{U}=\Lambda_{\mu}(f-U)-\Lambda_{\mu}(-U)\leq\frac{1}{p}\Lambda_{\mu}(-pU)+\frac{1}{q}\Lambda_{\mu}(qf)-\Lambda_{\mu}(-U)
\end{equation}
by Holder's inequality considering the conjugate exponent $q:=\frac{p}{p-1}$ of $p$. By the variational formula of Donsker-Varadhan, we deduce that:
\begin{align}
\mathbf{H}[\nu|\mu_{U}]=\underset{f\in\mathcal{M}_{b}(E)}{\sup}\bigg\{\int fd\nu-\Lambda_{\mu_{U}}(f)\bigg\}\geq
\underset{f\in\mathcal{M}_{b}(E)}{\sup}\bigg\{\int fd\nu- \frac{1}{q}\Lambda_{\mu}(qf)\bigg\}+\Lambda_{\mu}(-U)-\frac{1}{p}\Lambda_{\mu}(-pU ).
\end{align}
Gold:
\begin{equation}
\underset{f\in\mathcal{M}_{b}(E)}{\sup}\bigg\{\int fd\nu-\frac{1}{q}\Lambda_{\mu}(qf) \bigg\}+\Lambda_{\mu}(-U)-\frac{1}{p}\Lambda_{\mu}(-pU)=\frac{1}{q}\mathbf{H}[\nu|\mu]+\Lambda_{\mu}(-U)-\frac{1}{p}\Lambda_{\mu}(-pU).
\end{equation}
So if $\mathbf{H}[\nu|\mu_{U}]<+\infty$, $\mathbf{H}[\nu|\mu]<+\infty$ or equivalently, $\log (\frac{d\nu}{d\mu})\in L^{1}(\nu)$ and:
\begin{equation}
\log\bigg(\frac{d\nu}{d\mu_{U}}\bigg)=\log\bigg(\frac{d\nu}{d\mu}\bigg)+U+\Lambda_{\mu}(-U)\in L^{1}(\nu).
\end{equation}
This proves the proposition.
\end{proof}

\vfill
\addcontentsline{toc}{section}{Acknowledgements} 
\section*{Acknowledgements}
\boitesimple{
\begin{quotation}
\noindent The author thanks \href{https://blog.univ-angers.fr/panloup/}{Fabien \textsc{Panloup}} for extensive discussions and suggestions. Thanks also to everyone who has contributed in one way or another to the realization of this work. Last but not least, thankful for the benefits from \href{https://www.lebesgue.fr/en}{\textsc{Henri Lebesgue Center}}\footnote{Program \textbf{ANR-11-LABX-0020-0}} such as \href{https://www.lebesgue.fr/en/content/bourses_master}{\textsc{Master Scholarships}}, \href{https://www.lebesgue.fr/en/node/4878}{\textsc{Lebesgue Doctoral Meeting}}, \href{https://www.lebesgue.fr/en/content/post-doc}{\textsc{Post doc positions}}...
\end{quotation}}

\newpage
\addcontentsline{toc}{section}{References} 
\nocite{*}

\bibliographystyle{abbrv}
\bibliography{biblio}

\end{document}